\documentclass{amsart}

\usepackage[T1]{fontenc}

\usepackage{cite}
\usepackage{tikz}
\usepackage{tikz-cd}
\usepackage{tikz-layers}
\usetikzlibrary{shapes,arrows}

\usetikzlibrary{calc,
	arrows,decorations.pathmorphing,
	backgrounds,fit,positioning,shapes.symbols,chains}
\usetikzlibrary{shapes.geometric, graphs, knots}
\usetikzlibrary{decorations.pathmorphing}
\usetikzlibrary{decorations.markings}
\usetikzlibrary{decorations.pathreplacing,angles,quotes,decorations.markings, arrows.meta,knots}
\usetikzlibrary{positioning,shadows,arrows,trees,fit}
\usetikzlibrary{mindmap}
\usepackage{amssymb,amsmath,amsfonts,latexsym, enumerate}
\usepackage{bm}
\usepackage{yhmath}

\usepackage{multicol}
\usepackage[all,cmtip]{xypic} 
\usepackage{xypic} 
\xyoption{curve}
\usepackage{ragged2e}

\usepackage{standalone}

\usepackage{color}
\usepackage{float}

\usepackage{hyperref}
\usepackage[noabbrev, capitalize]{cleveref}
\usepackage{xfrac}    

\allowdisplaybreaks

\makeatletter
\newcommand*{\centerfloat}{%
	\parindent \z@
	\leftskip \z@ \@plus 1fil \@minus \textwidth
	\rightskip\leftskip
	\parfillskip \z@skip}
\makeatother

\restylefloat*{figure}

\numberwithin{equation}{section}

\newtheorem{thm}[equation]{Theorem}
\newtheorem{lem}[equation]{Lemma}
\newtheorem{prop}[equation]{Proposition}
\newtheorem{cor}[equation]{Corollary}

\restylefloat*{figure}

\newtheorem{defn}[equation]{Definition}

\theoremstyle{remark}
\newtheorem{rmk}[equation]{Remark}

\newtheorem{ex}[equation]{Example}

\renewcommand{\emptyset}{\font\cmsy = cmsy10 at 10pt
	\hbox{\cmsy \char 59}
}

\definecolor{burntsienna}{rgb}{0.91, 0.45, 0.32}
\definecolor{sapgreen}{rgb}{0.31, 0.49, 0.16}
\definecolor{teal}{rgb}{0.0, 0.5, 0.5}

\definecolor{blue}{rgb}{0.38, 0.51, 0.71} 
\definecolor{darkblue}{RGB}{17, 42, 60} 
\definecolor{red}{RGB}{175, 49, 39} 

\definecolor{orange}{RGB}{217, 156, 55} 
\definecolor{green}{RGB}{144, 169, 84} 
\definecolor{palegreen}{RGB}{197, 184, 104} 

\definecolor{yellow}{RGB}{250, 199, 100} 
\definecolor{brokenwhite}{RGB}{218, 192, 166} 
\definecolor{brokengrey}{rgb}{0.77, 0.76, 0.82} 

\mathchardef\mathdash="2D
\newcommand{\ov }{/}

\newcommand{\defeq}{\overset{\textup{def}}{=}}

\newcommand\N{\mathbb{N}}
\newcommand\Z{\mathbb{Z}}

\newcommand{\bbb}{\bm{b}}
\newcommand{\ccc}{\bm{c}} 
\newcommand{\ddd}{\bm{d}}
\newcommand{\mm}[1][m]{\mathbf{#1}}
\newcommand{\nul}{\mathbf{0}}
\newcommand{\one}{\mathbf{1}}
\newcommand{\two}{\mathbf{2}}
\newcommand{\nn}{\mathbf{n}}

\tikzset{
	dot/.style = {circle, fill, minimum size=#1,
		inner sep=0pt, outer sep=0pt},
	dot/.default = 5pt 
}

\tikzset{->-/.style={decoration={
			markings,
			mark=at position #1 with {\arrow{>}}},postaction={decorate}}}
\tikzset{-<-/.style={decoration={
			markings,
			mark=at position #1 with {\arrow{<}}},postaction={decorate}}}

\newcommand\Set{\mathsf{Set}}
\newcommand\Cat{\mathsf {Cat}} 
\newcommand\fin{\mathsf{Set_f}}
\newcommand\sSet{\mathsf{sSet}}
\newcommand\V{\mathsf{X}} 
\newcommand{\CCat}{ \mathsf{C}}

\newcommand{\DCat}[1][D]{ \mathsf{#1}}

\newcommand{\listm} {\small{\mathsf{list}}}
\newcommand\CCC{\mathfrak{C}}
\newcommand\DDD[1][D]{\mathfrak{#1}}

\newcommand{\BD}{\mathsf{BD}}

\newcommand{\BDd}{\mathsf{dBD}}

\newcommand{\Tf}[1][f]{\mathfrak T(#1)}
\newcommand\dipal{\{\wideparen{\uparrow,\ \downarrow}\}}
\newcommand\DiBD{\mathsf{OBD}}

\newcommand{\kcl}{\mathfrak{k}}

\newcommand{\WD}{\mathrm{WD}}

\newcommand{\CWD}[1][(\CCC, \omega)]{\WD^{#1}}

\newcommand{\CA}{{\mathsf{CA}}}
\newcommand{\CCA}[1][(\CCC,\omega)]{\CA^{#1}}

\newcommand{\bigCA}{\bm{\mathsf{CA}}}

\newcommand\E{\mathsf{V}}
\newcommand{\pr}[1]{\mathsf{psh}(#1)}
\newcommand{\sh}[1]{\mathsf{sh}(#1)}
\newcommand{\ElS}[1][S]{\ensuremath{\mathsf{el}{(#1)}}}
\newcommand{\MM}{\mathbb {M}}

\newcommand{\boffcat}[1][\MM, \DCat]{ \Theta_{#1}}

\newcommand\GS{\mathsf {GS}}
\newcommand\oGS{\mathsf{OGS}}
\newcommand\CSM{\bm{\mathfrak {MO}}}
\newcommand{\ElP}[2]{\ensuremath{\mathsf{el}_{#2}(#1)}}

\newcommand{\fiso}{{\mathbf{B}}}
\newcommand{\fisinv}{{\fiso}^{\scriptstyle{\bm \S}}}

\newcommand{\Comm}{{K}}
\newcommand{\CComm}[2][(\CCC,\omega)]{{\Comm}^{#1}_{#2}}

\newcommand{\CommE}[1][\E]{{K^{#1}}}

\newcommand{\CGS}[1][(\CCC,\omega)]{{\GS^{#1}}}

\newcommand{\GSE}[1][\mathsf V]{\GS_{#1}} 
\newcommand{\oGSE}[1][\mathsf V]{{\oGS_{#1}}} 
\newcommand{\prE}[2][\E]{\mathsf{psh_{#1}}(#2)}

\newcommand{\CSME}[1][\E]{{\bm{\mathfrak{MO}}}_{#1}}
\newcommand{\nuCSME}[1][\E]{{\bm{\mathfrak {MO}}}^{-}_{#1}}
\newcommand{\CCO}[1][(\CCC,\omega)]{\mathfrak{CA}^{#1}}
\newcommand{\COE}[1][\E]{{\bm{\mathfrak {CA}}}_{#1}}
\newcommand{\nuCOE}[1][\E]{{\bm{\mathfrak {CA}}}^{-}_{#1}}
\newcommand{\WPE}[1][\E]{{\bm{\mathfrak {WP}}}_{#1}}
\newcommand{\nuWPE}[1][\E]{{\bm{\mathfrak {WP}}}^{-}_{#1}}
\newcommand{\WPin}{{\bm{\mathfrak {WP}}}}

\newcommand\Di{\mathfrak{Di}}
\newcommand{\Disig}{\sigma_{\Di}}
\newcommand{\Dipal}{(\Di, \sigma_{\Di})}%
\newcommand{\Dicomm}{Di}
\newcommand{\DicommE}[1][\E]{Di^{#1}}

\newcommand{\In}{{\mathrm{ in }}}
\newcommand{\Out}{{\mathrm{ out }}}

\newcommand{\elG}[1][\mathcal{G}]{\ensuremath{\mathsf{el}{(#1)}}}
\newcommand{\oelG}[1][\tilde {\mathcal{G}}]{\ensuremath{\mathsf{el}_{Di}{(#1)}}}
\newcommand\nuCSM{{\bm{\mathfrak {MO}}}^-}
\newcommand\CO{{\bm{\mathfrak {CA}}}}

\newcommand{\yet}{\Upsilon}
\newcommand{\yetp}{{\yet_*}}

\newcommand{\diag}{\mathsf{\bm D}}
\newcommand{\GrShape}{\mathsf{psh_f}(\bm D)}
\newcommand{\Grbig}{\mathsf{Gr}\GrShape}
\newcommand{\Gret}{\mathsf{Gr}}
\newcommand{\oGret}{\mathsf{OGr}}
\newcommand\Gr{\mathsf{C}\Gret}

\newcommand{\Griso}{\mathdash\mathsf{CGr_{iso}}}

\newcommand\G{\mathcal{G}}
\renewcommand\H{\mathcal{H}}
\newcommand\W{\mathcal W}
\newcommand\C{\mathcal{C}}
\newcommand{\Lk}[1][k]{\ensuremath{\mathcal L^{#1}}}
\newcommand{\Wl}[1][m]{\ensuremath{\mathcal W^{#1}}}
\newcommand{\Wm}[1][m]{\ensuremath{\mathcal W^{#1}}}

\newcommand\CX[1][X]{{\mathcal C_{#1}}}

\newcommand{\Fgraph}{ \xymatrix{
		E \ar@(lu,ld)[]_\tau&& H \ar[ll]_s \ar[rr]^t&& V}}
\newcommand{\Fgraphdash}{\xymatrix{
		E \ar@(lu,ld)[]_{\tau } &&H \ar[ll]_{s } \ar[rr]^{t }& &V }}

\newcommand{\Fgraphvar}[6]{\xymatrix{
		*[r] {#1}\ar@(ul,dl)[]_{#6} && {#2} \ar[ll]_-{#4} \ar[rr]^-{#5}&& {#3}}}

\newcommand\Cv[1][v]{{\mathcal C_{\mathbf{#1}}}}

\newcommand{\EI}{\ensuremath{{E_\bullet}}}

\newcommand{\vH}[1][v]{\sfrac{H}{#1}}
\newcommand{\vE}[1][v]{\sfrac{E}{#1}}
\newcommand{\nV}[1][n]{V_{#1}}
\newcommand{\nH}[1][n]{H_{#1}}
\newcommand{\nE}[1][n]{E_{#1}}

\newcommand\shorte[1][\tilde e]{{\shortmid_{#1}}}

\newcommand{\esv}[1][v]{\ensuremath{\iota_{{#1}}}}
\newcommand{\ese}[1][\tilde e]{\ensuremath{\iota_{{#1}}}}

\newcommand\X{\mathcal X}
\newcommand\Y{\mathcal Y}

\newcommand\Gg{\mathbf \Gamma}
\newcommand\Gdg{\mathbf \Lambda}
\newcommand\Gid[1][\G]{\mathbf{I}^{#1}}
\newcommand\coGg[1][\G]{{\Gg}({#1})}

\newcommand{\oedge}{\downarrow}

\newcommand{\OOO}{\mathbb {O}}
\newcommand{\TT}{\mathbb T}
\newcommand{\DD}{\mathbb D}
\newcommand{\TTp}{\mathbb T_*}
\newcommand{\LL}{\mathbb L}
\newcommand{\OTTk}{\mathsf O \TTk}
\newcommand{\ODD}{\mathsf \mathbb D}

\newcommand{\TTk}{{\mathbb{ T}^{\times}}} 
\newcommand{\Tk}{{T^{\times}} }
\newcommand{\muk}{{ \mu^{\TTk}} }
\newcommand{\etak}{{ \eta^{\mathbb{T}^{\times}}}}
\newcommand{\OTk}{\mathsf O \Tk}
\newcommand{\OD}{\mathsf OD}

\newcommand{\TTpk}{{\tilde{ \mathbb {L }}\mathbb{ T}_*}} 
\newcommand{\Tpk}{{\widetilde L T_* }} 

\newcommand{\Grisok}{{\mathdash\mathsf{Gr}_{\mathrm{iso}}}}
\newcommand{\GretG}[1][\G]{\ensuremath{\Gret^{(#1)}}}
\newcommand{\GrG}[1][\G]{\ensuremath{\Gr^{(#1)}}}

\newcommand{\Gretp}{\mathsf{Gr}_{*}}
\newcommand{\oGretp}{\mathsf{O Gr_{*}}}
\newcommand{\Grp}{\ensuremath{\mathsf{C}\Gretp}}
\newcommand{\GSp}{\ensuremath{\mathsf{\GS}_*}}

\newcommand{\fisinvp}{\ensuremath{\fisinv_*}}

\newcommand{\elpG}[1][\mathcal{G}]{\ensuremath{\mathsf{el}_*{(#1)}}}

\newcommand{\Gnov}[1][W]{\ensuremath{\mathcal G_{\setminus #1}}}
\newcommand{\Enov}[1][W]{\ensuremath{E_{\setminus #1}}}
\newcommand{\Hnov}[1][W]{\ensuremath{H_{\setminus #1}}}
\newcommand{\Vnov}[1][W]{\ensuremath{V_{\setminus #1}}}
\newcommand{\snov}[1][W]{\ensuremath{s_{\setminus #1}}}
\newcommand{\tnov}[1][W]{\ensuremath{t_{\setminus #1}}}
\newcommand{\taunov}[1][W]{\ensuremath{\tau_{\setminus #1}}}
\newcommand{\delW}[1][W]{\mathsf{del}_{\setminus #1}}

\newcommand\XGrsimp[1][X]{{#1}\mathdash\mathsf{CGr}_{\mathsf{sim}}}

\newcommand{\etap}{\ensuremath{\eta^{\TTp}}}
\newcommand{\Tp}{\ensuremath{T_*}}

\newcommand{\GSEp}[1][\E]{\ensuremath{{\mathsf{\GS}_*}_{#1}}}

\newcommand\Gretsimp{\mathsf{Gr}_{\mathsf{sim}}}
\newcommand\XGretsimp{\mathdash\mathsf{Gr}_{\mathsf{sim}}}

\newcommand{\core}[1]{\mathrm{Core}(#1)}
\newcommand{\Coriso}[1][X]{{{#1}\mathdash\mathsf{Cor}_{\boxtimes}}_{\mathrm{iso}}}
\newcommand{\CorGg}[1][\X]{ \mathsf{Cor}_{\boxtimes} ^{(#1)}}

\newcommand\Klgr{\Xi}
\newcommand{\factcat}[1][\beta]{\ensuremath{\mathsf{fact}_*(#1)}}

\newcommand\Klgrt{{\Xi^\times}}

\title[A nerve theorem for circuit algebras]{Modular operads, iterated distributive laws and a nerve theorem for circuit algebras}
\author{Sophie Raynor}

\date{\today}

\thanks{The author acknowledges the support of Australian
	Research Council grants DP160101519 and FT160100393.}

\begin{document}

\maketitle

\begin{abstract}
	Circuit algebras are a symmetric version of Jones's planar algebras. They originated in quantum topology as a framework for encoding virtual crossings. This paper 
	extends existing results for modular operads to construct a graphical calculus and monad for general circuit algebras and prove an abstract nerve theorem. 
	The proof relies on a subtle interplay between distributive laws and abstract nerve theory, and provides extra insights into the underlying structures. 
	Oriented circuit algebras are equivalent to wheeled props and specialisations of the results to wheeled props follow as straightforward corollaries. 
\end{abstract}

\section{Introduction}

Circuit algebras are a symmetric version of Jones's planar algebras, introduced in \cite{BND17} as a framework for studying finite-type invariants of virtual knotted objects (see also \cite{DF18, Hal16, Tub14}). 
They are graded symmetric monoids equipped with a unary contraction operation satisfying suitable associativity and distributivity axioms. These axioms imply that circuit algebras may be equivalently described as algebras for an operad of wiring diagrams, as symmetric monoidal functors from a category of Brauer diagrams, or as modular operads in a monoid subcategory of symmetric species \cite{RayCA1}.

Besides quantum topology, these structures arise in many areas of mathematics, including invariant theory \cite{DM23, RayCA1} and string theory \cite{Sch96}. In particular, oriented circuit algebras are equivalent to wheeled props \cite{DHR20, RayCA1} which govern algebraic structures 
with trace such as those that appear in contexts including Batalin–Vilkovisky quantisation and deformation theory (see e.g.,~\cite{DM23, Mer10, MMS09}).

	This paper provides a detailed technical account of circuit algebra combinatorics. Building on analogous results for modular operads \cite{Ray20}, a monad and graphical calculus for circuit algebras is constructed and a nerve theorem in the style of Weber  \cite{BMW12, Web07}  is proved.

	\begin{thm}
		[\cref{thm. iterated law}]
		\label{thm: main intro}
		The category $\CO$ of circuit algebras is equivalent to the Eilenberg-Moore (EM) category of algebras for a monad $\OOO$ on Joyal and Kock's graphical species category $\GS$ \cite{JK11}. 
	\end{thm}

	Let $\pr{\CCat}$ denote the category of presheaves (contravariant $\Set$-valued functors) on a category $\CCat$.
	
	\begin{thm}[\cref{nerve theorem}]
		\label{thm nerve intro}\label{thm informal nerve}
		There is a full subcategory $\Klgrt \hookrightarrow \CO$, whose objects are undirected graphs. The nerve functor $N \colon \CO \to \pr{\Klgrt}$ induced by the inclusion $\Klgrt \hookrightarrow \CO$ is fully faithful.
		
		A presheaf $P $ on $\Klgrt$ is equivalent to the nerve of a circuit algebra if and only if it is ``Segal''. 
	\end{thm}
	
Put simply, the graphs encode ways of combining circuit algebra operations, and the ``Segal condition'' says that complex combinations of these operations are completely determined by their constituent parts (see \cref{ssec: Weber} and e.g.,~\cite{Hac22}). 

These theorems apply simultaneously to a very general class of circuit algebras that includes (non)oriented and coloured versions. A monad and nerve theorem for wheeled props, in terms of a category of directed graphs, follow as immediate corollaries (see \cref{cor wheeled props nerve}). 

 The proof of \cref{thm nerve intro} uses Weber's abstract nerve theory \cite{BMW12, Web07}. However, unlike the nerve theorems for categories and operads described in the examples of \cite[Section~4]{Web07}, the Segal condition for circuit algebras in \cref{thm: CO nerve} (as with the Segal condition for modular operads in \cite[Theorem~8.2]{Ray20}) cannot be directly obtained with Weber's machinery: while Theorems \ref{thm: main intro}~and~\ref{thm nerve intro} imply that the circuit algebra monad $\OOO$ is ``nervous'' in the sense of \cite{BG19}, it does not ``have arities'' in the sense of  \cite{BMW12, Web07}.

In \cite[Introduction~\&~Section~6]{Ray20}, I explained how the so-called ``problem of loops'', that is related to contractions of modular operadic units (the unit trace), means that the monad for unital modular operads does not admit a straightforward description in terms of the ``graph-substitution'' notion that underlies analogous constructions for operads \cite{MW07} or prop(erad)s \cite{HRY15, MMS09, Koc18}. As circuit algebras are a special case of modular operads (\cref{sec: definitions} and \cite{RayCA1}), this is also an issue for constructing the monad $\mathbb O$ of \cref{thm: main intro}. Crucially, to obtain the Weber-style nerve theorem \cref{thm nerve intro}, there must be a fully faithful embedding of the graphical category $\Klgrt$ into the category $\CO$ of circuit algebras. This implies that the unit trace is encoded by a map from the graph consisting of a single, isolated (degree 0) vertex to the exceptional edge with no vertices (see \cref{rmk loop intuition} and \cite[Introduction,~Figure~2,~and~Sections~6-8]{Ray20}) and not by a map to an exceptional nodeless loop (as in e.g.,~\cite{HRY19a, HRY19b,Str25}).

The central principle underlying this work (and \cite{Ray20}) is that, under certain conditions, it is possible to build up a description of complicated algebraic structures by suitably combining simpler ones. In particular, if the simpler structures are governed by monads, we may be able to reframe this as a problem of ``composing'' these monads. Distributive laws \cite{Bec69} tell us whether and how monads can be composed and how they interact in the composite. Iterated distributive laws \cite{Che11} describe conditions under which multiple monads may be composed.

In \cite[Section~7]{Ray20}, the modular operad monad was obtained as a composition $\DD\TT$, via a distributive law, of monads $\TT$ (governing composition) and $\DD$ (governing units) on $\GS$.  For circuit algebras (\cref{sec. iterated}), it is even necessary to extend the construction in \cite{Ray20} to include a third monad $\LL$ (governing the monoidal structure) and iterated distributive laws.

\begin{thm}[\cref{prop. iterated law} {\&} \cref{thm. iterated law}]\label{thm: composite intro}	\label{thm. dist intro}
	There is a triple of monads $\LL, \DD$ and $\TT$ on $\GS$ and iterated distributive laws between them such that the induced composite $\LL\DD\TT$ is the circuit algebra monad $\OOO$ from \cref{thm: main intro}.
\end{thm}

The decomposition in \cref{thm: composite intro} is crucial to the proof of \cref{thm nerve intro}. Weber's nerve machinery \cite{BMW12, Web07} relies on a monad $\MM$ on a category $\CCat$ preserving certain colimits in some proper subcategory $\DCat \subsetneq \CCat$, that is said to  \textit{provide arities} for $\MM$, and this is only satisfied under very special conditions. 
However, if $\MM = \MM_n \dots \MM_1$ is a composite monad, then each subset of $\{ \MM_{1}, \dots,\MM_{n}\}$ describes a monad on $\CCat$ and we may consider whether a certain extension (and/or lift) of this monad, if it exists, has arities (\cref{ssec. combine}). In particular, \cref{thm nerve intro} is proved by showing that $\LL\TT$ admits an extension to a monad with arities on the category $\GS^\DD$ of algebras for $\DD$.

The iterated distributive laws used to prove Theorems \ref{thm: main intro}-\ref{thm. dist intro} illuminate the graphical combinatorics of circuit algebras, and dovetail nicely with variations of circuit algebras that appear in the literature: Algebras for $\LL$ are monoids in $\GS$ (\cref{rmk MO not monoidal}) so \cref{thm. dist intro} implies that circuit algebras are equivalent to modular operads ($\DD\TT$-algebras) in the monoid subcategory of $\GS$. Moreover, algebras for the monad $\LL \TT$ are nonunital circuit algebras, whose underlying modular operadic multiplication is not required to have a multiplicative unit (\cref{prop nonunital CA} and \cref{lem. LT is Tk}). By \cref{sec: nonunital}, these are symmetric monoidal functors from subcategories of ``{downward}'' Brauer diagrams \cite{SS15} (see also \cite{RayCA1} and \cite{KRW21, Sto22}). 
In particular, these structures have applications to the representation theory of the infinite dimensional orthogonal, symplectic and general linear groups (see \cite{SS15} and Sections~3.3~\&~6.2 of \cite{RayCA1}). These relationships are explored in the companion paper \cite{RayCA1}. A partial summary is provided in \cite[Table~1]{RayCA1}.

This paper aligns with a larger body of recent works aimed at describing the underlying combinatorics of operad-like structures, such as cyclic and modular operads \cite{DCH19, Hac24, HRY19a, HRY19b,KW24, Ray20, Sto22, War22}, properads and props \cite{HRY15, Koc16, Sto23}. Bek McCann and I are currently using an iterated distributive law approach to provide an operadic description of compact closed categories. A key motivation for much of this research is to find models that characterise generalised operads and their algebras \emph{up-to-homotopy}. The present work is no exception. 

	Recently, Dancso, Halacheva and Robertson have used circuit algebras to relate the graded Kashiwara-Vergne and 
 graded Grothendieck-Teichm\"uller groups, $\mathrm{KRV}$ and $\mathrm{GRT}$ \cite{DHR21}. However, associativity of the circuit algebra operations is not preserved under passage to pro-unipotent completion and, to relate the ungraded groups $\mathrm{GT}$ and $\mathrm{KV}$ directly, it is necessary to weaken the circuit algebra axioms (see \cite[Remark~1.1]{DHR21}).

\cref{thm nerve intro} characterises circuit algebras via a \textit{strong} Segal condition on graphical presheaves. So, it is natural to define homotopy circuit algebras as (simplicial set valued) presheaves that satisfy a weakened Segal condition. Indeed, \cite{HRY19b} and \cite{Ray20} described models for homotopy modular operads -- using different methods and slightly different graphical categories -- as weak Segal presheaves. Unfortunately, neither of these results can be generalised straightforwardly to obtain a model structure on $\pr{\Klgrt}$. It is for future work to explore whether such a model structure exists, or if the weak Segal condition is sufficient to describe a practical homotopy theory for circuit algebras.

 The nonunital case, which avoids the problem of loops, is better understood: Stoeckl \cite{Sto23} has constructed a model for nonunital homotopy wheeled props, building on Ward's proof that the operad for nonunital modular operads is Koszul \cite{War22}, and Kaufmann and Ward \cite{KW24} have shown that the operad governing (nonoriented) nonunital circuit algebras (there called ``Schwarz modular operads'') is Koszul. It would be interesting to compare the model for nonunital homotopy circuit algebras so-obtained with the Segal characterisation suggested by the nonunital case of \cref{thm nerve intro}. It is not expected that these results can be extended to unital circuit algebras. However, the iterated distributive law of \cref{thm. dist intro} together with Markl's notion of distributive laws for operads in \cite{Mar96} -- used in \cite{KW24} to prove that the (monochrome) nonunital circuit algebra operad is Koszul -- gives a  possible approach for exploring Koszulity of the operad for unital circuit algebras. (See also \cref{rmk operad dist} and \cite[Introduction]{RayCA1}.) Batanin and Berger's theory of \textit{tame polynomial monads} \cite{BB17}, together with the distributive law decomposition of the circuit algebra monad, could also shed light on the homotopy theory of unital circuit algebras.

Finally, a word on potential applications of this work outside mathematics. My own interest in circuit algebras, modular operads, and other operadic structures governed by graphs was originally motivated by questions involving complex networks with structure at multiple scales:  Given a circuit algebra $A$, the bar construction for $\LL\DD\TT$ and $A$ may be viewed as a rulebook for \textit{zooming in and out} of networks decorated by $A$.  Indeed, there is an increasing body of work exploring this operadic (or monadic) approach in applications to fields such as systems engineering, dynamical systems, software engineering and neuroscience (see e.g.,~\cite{PBHF23}). Just as they provide the main difficulty for constructing the circuit algebra monad, cycles -- or ``feedback loops'' -- are the drivers of complexity in many real world networks (see e.g.,~\cite{Gig17, Gre23}). The approach of this work (and \cite{Ray20}) highlights the combinatorial structure of loops and may therefore shed light on the properties of feedback loops as they appear in diverse applications.

\subsection{Details and overview}

\cref{sec. cat intro} provides a relatively informal introduction to iterated distributive law and the abstract nerve machinery of \cite{BMW12}, on which the proofs of Theorems \ref{thm: main intro}-\ref{thm. dist intro} are built, as well as some technical remarks on the relationship of this work to the extended algebraic patterns of \cite{CH21} and extensions of strongly cartesian monads described in \cite[Proposition~4.4]{BMW12}. Most of this section should be accessible to readers with just a basic knowledge of category theory.

In \cref{sec: definitions}, circuit algebras and modular operads are introduced as $\E$-valued presheaves over $\GS$ that satisfy certain axioms. It is important to note that the companion paper \cite{RayCA1} and e.g.,~\cite{BND17, Hal16, DHR21, HRY19a, HRY19b} are concerned with circuit algebras (and modular operads) ``enriched'' in a symmetric monoidal category $(\V, \otimes, I)$. From this perspective, different flavours of enriched circuit algebras are described as algebras over different operads. By contrast, this paper gives constructions ``internal'' to a category $\E$ with sufficient (co)limits. The sets of colours (or generating objects) in the enriched version are replaced by ``object objects'' (palettes) in the internal version. This enables all (internal) circuit algebras to be obtained as algebras for a single monad. Of course, the two definitions are equivalent when $\V = \E$ is the category of sets.

The technical work of the paper starts at \cref{s. graphs} which is largely a review of basic graph constructions, using the formalism of Joyal and Kock's ``(Feynman) graphs'' \cite{JK11} (see also \cite {HRY19a, HRY19b, Ray20}), and explains the relationship with Brauer diagrams and wiring diagrams in terms of which circuit algebras are usually defined \cite{BND17,DHR20,DHR21,RayCA1}.

A monad $\TTk$ for nonunital circuit algebras is constructed in \cref{sec: nonunital}. This is almost identical to the monad discussed by Kock in \cite{Koc18} and very similar to the nonunital modular operad monad $\TT$ constructed in \cite{Ray20}. \cref{sec. mod op} reviews the construction in \cite{Ray20} of the modular operad monad $\DD\TT$ by composing $\TT$ with a unit monad $\DD$, and gives an explicit description of the monad $\TTp$, that extends $\TT$ to the EM category of algebras for $\DD$, $\GSp$. 

Despite the similarities between $\TTk$ and $\TT$, the monad for unital circuit algebras cannot be obtained directly from a distributive law involving $\TTk$ and $\DD$. 
Instead, it is necessary to supplement the construction of \cite{Ray20} by observing that $\TTk = \LL\TT$ is, itself a composite monad. 
This happens in \cref{sec. monad CO}, in which an iterated distributive law is described and it is shown circuit algebras are equivalent to algebras for the resulting composite $\LL \DD \TT$ of three monads on $\GS$.

Finally, \cref{sec. nerve} contains a detailed description of the graphical category $\Klgrt$ for circuit algebras, and this leads easily to the proof of the nerve theorem (\cref{thm nerve intro}). The monad $\LL$ lifts to $\DD\TT$ algebras and $\TT$ extends to $\DD$ algebras. Hence, there is a ``lifted'' monad $\widetilde{ \LL }\TTp$  on $\GSp$ whose algebras are circuit algebras. \cref{thm nerve intro} is then implied by the observation that $\widetilde{ \LL }\TTp$ has arities on $\GSp$ and hence 
satisfies the conditions of the abstract nerve machinery of \cite{BMW12}. Three different proofs of this theorem are discussed (c.f.,~Proof of \cref{thm: CO nerve} and Remarks~\ref{rmk bmw proof}~and~\ref{rmk patterns proof}).

\subsection*{Acknowledgements}

I thank Marcy Robertson for motivating me to write this paper and for her unfailing guidance and support. The ideas in this work were incubated and nourished during my time at the Centre of Australian Category Theory, Macquarie University, Dharug Country. I acknowledge the support, friendship, and mathematical input of my friends and colleagues there. Thanks, in particular, to Steve Lack and John Power. Thanks also to Kurt Stoeckl for great suggestions and to Bek McCann and my other students and colleagues at James Cook University, Bindal Country, where this work was written.

\section{Distributive laws and abstract nerve theory}\label{sec. cat intro}

The main result of this work is the construction of a monad (\cref{thm. iterated law}) and a graphical category (\cref{ss. graphical category}) and nerve theorem (\cref{thm: CO nerve}) for circuit algebras. 
\cref{thm: CO nerve} is proved using the machinery of abstract nerve theory \cite{BMW12}, in combination with iterated distributive laws \cite{Bec69, Che11}.

Given the quantum topology origins of circuit algebras \cite{BND17, DHR21} and their connections with Brauer categories and representation theory \cite{DM23,LZ15, RayCA1,SS15}, this overview is aimed at a diverse audience including topologists, representation theorists and category theorists, and is intended to provide an accessible explanation of these formal concepts and how they interact in the proof of \cref{thm: CO nerve}, as well as to establish notation and terminology for what follows. 

More technical details are included as remarks, and I refer those interested in further details on the general theory to \cite{BMW12, Web07}, for abstract nerve theory, and \cite{Bec69, Che11}, for (iterated) distributive laws.

\subsection{Iterated distributive laws} \label{ssec. dist} 

This paper applies the general philosophy that, to understand something difficult, it is often helpful to break it up into pieces and instead study these simpler components and their local interactions. Monads are categorical devices that are often used to encode algebraic structures. In general, monads do not compose. But, in some cases, it may be possible to study a complicated algebraic structure via interactions between monads for simpler structures.

A monad $\MM = (M, \mu^{\MM}, \eta^\MM)$ on a category $\CCat$ is given by an endofunctor $M \colon \CCat \to \CCat$, together with natural transformations $\mu^{\MM} \colon M^2 \Rightarrow M$ (the \textit{monadic multiplication}) and $\eta^{\MM} \colon \mathrm{id}_{\Cat}\Rightarrow M$ (the \textit{monadic unit}) that satisfy axioms making $\MM$ into an associative monoid in the category of endofunctors on $\CCat$. (See, e.g.,~{\cite[Chapter~VI]{Mac98}, for detailed review of the general theory of monads and their algebras.)
	
	An algebra $(c, \theta)$ for $\MM$ consists of an object $ c$ of $\CCat$, and a morphism $\theta \colon Mc \to c$ satisfying two axioms that assert the compatibility of $\theta$ with $\mu^{\MM}$ and $\eta^{\MM}$. Algebras are the objects of the \textit{Eilenberg-Moore (EM) category $\CCat^{\MM}$ of algebras for $\MM$}, whose morphisms $(c,\theta) \to (d, \phi)$  are given by morphisms $f  \in \CCat(c,d)$ such that $\phi \circ Mf   = f \circ  \theta $. The monad $\MM$ induces free-forgetful adjoint functors $\CCat \leftrightarrows \CCat^{\MM}$ where the left adjoint $\CCat \rightarrow \CCat^{\MM}$ is given by $c \mapsto (Mc, \mu^{\MM}(c))$.

	Let $\MM = (M, \mu^{\MM}, \eta^{\MM})$ and $ \MM' = ({M'}, \mu^{\MM'}, \eta^{\MM'})$ be monads on a category $\CCat$.  In general, the composite functor $MM'$ does not carry a monad structure, since the data of $\MM$ and $\MM'$ is not sufficient to define a multiplication $(MM')^2 = MM'MM' \Rightarrow MM'$. 
	Observe, however, that any natural transformation $\lambda\colon{M'}M \Rightarrow M{M'}$ induces a natural transformation $\mu_\lambda \defeq (\mu^{\mathbb{M}}\mu^{\mathbb{{M'}}})\circ (M\lambda M')  \colon( M{M'})^2 \Rightarrow M{M'}$.

	\begin{defn}\label{defn: dist} 
		A \emph{distributive law} for $\MM$ and $ \MM'$ is a natural transformation $\lambda\colon{M'}M \Rightarrow M{M'}$ such that the triple $(M{M'}, \mu_\lambda , \eta^{\mathbb{M}}\eta^{\mathbb{{M'}}})$ defines a monad $\mathbb{M} \mathbb{{M'}}$ on $\CCat$.
		
	\end{defn} 
	
	Distributive laws were introduced by Beck \cite{Bec69} in terms of four axioms encoding their interactions with the monadic multiplications and units.
	
	Let $\lambda\colon {M'}M \Rightarrow M{M'}$ be a distributive law between monads $\MM$ and $\MM'$ on $\CCat$. If $(c, \theta)$ is an $\MM\MM'$ algebra, then $ c$ admits an $\MM'$ algebra structure $\theta \circ \eta^{\MM}M'  \colon M'c \to MM'c \to c$ and an $\MM$ algebra structure $\theta \circ\lambda\circ M\eta^{\MM'}\colon Mc \to M'Mc \to MM'c \to c$. In particular, the induced functor $ \CCat^{\MM\MM'} \to \CCat ^{\MM'}$ has a left adjoint and the monad $\MM$\textit{ lifts} to a monad on $\CCat^{\MM'}$ whose EM category of algebras is isomorphic to $\CCat^{\MM \MM'}$ (see \cite[Section~3]{Bec69}). Indeed, the existence of a functorial lift of $\MM$ to $\MM'$ is equivalent to the existence of a distributive law $\lambda\colon{M'}M \Rightarrow M{M'}$ (see e.g. \cite{Che11}).
	
In general, the functor $ \CCat^{\MM\MM'} \to \CCat ^{\MM}$ does not have a left adjoint and so the monad $\MM'$ does not extend to a monad on $\CCat^{\MM}$. However, in case the left adjoint does exist (in which case it is described by a certain right Kan extension \cite{Mat25, Str72}) and the counit of the adjunction is an isomorphism, then $\MM'$ may be extended to a monad $\MM'_*$ on $\CCat^{\MM}$ such that $(\CCat^{\MM})^{\MM'_*} \cong \CCat^{\MM\MM'}$, and the following square of adjunctions commutes:

\begin{equation}
	\xymatrix@C = .5cm@R = .3cm{
		\CCat^{\MM'} \ar@<-5pt>[rr]\ar@<-5pt>[dd]&\scriptsize{\top}&\CCat^{\MM\MM'} \ar@<-5pt>[ll]\ar@<-5pt>[dd]\\
		\vdash&&\vdash\\
		\CCat \ar@<-5pt>[uu] \ar@<-5pt>[rr]&\scriptsize{\top}& \CCat^\MM. \ar@<-5pt>[ll]\ar@{-->}@<-5pt>[uu]}\end{equation}

\begin{ex}
	\label{ex: category composite} The category of small categories is equivalent to the EM category of algebras for the category monad on the category of directed graphs. This sends a directed graph $G$ to the directed graph $\widetilde G$ with the same vertices, and -- for each directed path (of length $n \geq 0$) from $v$ to $w$ in $G$ -- a directed edge $v \leadsto w$ in $\widetilde G$.
	
	This monad may be obtained as a composite of the \textit{semi-category monad} -- which governs associative composition and assigns to each directed graph $G$ the graph with the same vertices as $G$ and edges corresponding to non-trivial directed paths (of length $n >0$) in $G$ -- and the \textit{reflexive graph monad} that adjoins a distinguished loop at each vertex of a directed graph. The corresponding distributive law encodes the property that the adjoined loops provide identities for the (semi-categorical) composition.

\end{ex}
\begin{ex}\label{ex dist MO}
	Similarly, the monad for modular operads on the category $\GS$ of graphical species (see \cref{defn: graphical species}) is obtained in \cite{Ray20}, as a composite $\DD\TT$: the monad $\TT$ governs the composition structure (multiplication and contraction), and the monad $\DD$ adjoins distinguished elements that encode the combinatorics of the multiplicative unit. The distributive law ensures that the distinguished elements provide units for the multiplication (see also \cref{sec. mod op}).
\end{ex}

Let $\MM_1, \dots,\MM_n$ be an $n$-tuple of monads on a category $\CCat$. 
If there are distributive laws $\lambda_{i,j}: M_i M_j \Rightarrow M_j M_i$ (defined for all pairs $1 \leq i < j \leq n$) such that, for each triple $1 \leq i < j< k \leq n$, the corresponding triple of monads and distributive laws satisfies the so-called \emph{Yang-Baxter} conditions 
\[ \lambda_{j,k}M_i \circ M_j \lambda_{i,k} \circ \lambda_{i,j}M_k =M_k  \lambda_{i,j} \circ \lambda_{i,k} M_j \circ M_i \lambda_{j,k}\colon M_iM_jM_k \Rightarrow M_kM_jM_i, \] 
then, by \cite{Che11}, $\MM_1, \dots,\MM_n$ may be composed to obtain a monad  $\MM_n \dots \MM_1$ on $\CCat$. 
It follows that, for each $1 \leq i < n$, there are composable monads $\MM_n\dots \MM_{i+1}$ and $\MM_i \dots \MM_1$ on $\CCat$, and $\MM_n\dots \MM_{i+1}$ lifts to a monad $\widetilde\MM_{(i+1)}$ on the EM category $\CCat^{\MM_i \dots \MM_{i+1}}$.

\subsection{Presheaves, slice categories and nerve functors} \label{ssec. presheaves}

Let $\E$ be a category and let $\CCat$ be an essentially small category. An $\E$-valued \emph{$\CCat$-presheaf} is a functor $ S\colon \mathsf C^\mathrm{op} \to \E$. The corresponding functor category is denoted by $\prE{\mathsf C}$. 
If $\E = \Set$ is the category of sets and all set maps, then we will write $\pr{\CCat} \defeq \prE{\CCat}$. 

The category of elements $\ElP{P}{\CCat}$ of a $\Set$-valued presheaf $P  \colon \CCat^{\mathrm{op}}\to \Set$ is defined as follows:
\begin{defn}
	\label{defn: general element} 
	Objects of $\ElP{P}{\CCat}$ are pairs $(c, x)$ -- called \emph{elements of $P$} -- where $c$ is an object of $\CCat$ and $x \in P(c)$. Morphisms $(c,x) \to (d,y)$ 
	are given by morphisms $f \in \CCat (c,d)$ such that $P(f)(y) = x$. 	
\end{defn}

A functor $\iota \colon \DCat \to \CCat$ induces a presheaf $\iota^* \CCat(-, c)$, $d \mapsto \CCat(\iota (d),c)$ on $\DCat$. For all $c \in \CCat$, the category $\ElP{\iota^*\CCat(-, c) }{\DCat}$ is called the \textit{slice category of $\DCat$ over $c$} and is denoted by $\iota \ov c$, or simply $\DCat \ov c$. 

The \textit{slice category of $\DCat$ under $c$}, denoted by $c \ov \DCat$ (or $c \ov \iota$) is the category $\iota^{\mathrm{op}} \ov c$ whose objects are morphisms $f \in \CCat(c, \iota(-))$ and whose morphisms $(d,f)\to (d', f)$ are given by $g \in \DCat (d,d')$ such that $ \iota (g) \circ g = f' \in \CCat(c, \iota(d'))$.

A presheaf $P \colon \CCat^{\mathrm{op}} \to \Set$ is \emph{represented} by some $c \in \CCat$ if $P =  \CCat(-, c)$. In this case, $\ElP{P}{\CCat}$ is precisely the \textit{slice category} 
$\CCat\ov c$ whose objects are pairs $(d,f)$ where $f \in \CCat (d,c)$, and morphisms $(d,f) \to (d', f')$ are commuting triangles in $\CCat$:
\[ \xymatrix{ d \ar[rr]^-g \ar[dr]_{f} && d' \ar[dl]^-{f'}\\&c&}\] 

In particular, the \emph{Yoneda embedding} $\CCat \to \pr{\CCat}, c \mapsto \CCat(-,c)$ induces a canonical isomorphism $ \ElP{P}{\CCat} \cong \CCat \ov P$ for all presheaves $P\colon \CCat ^{\mathrm{op}} \to \Set$. These categories will be identified in this work. 

\begin{defn}\label{defn nerve}
	The \emph{nerve functor}
	$N = N(F)\colon \CCat \to \pr{\DCat}$ associated to a functor $F\colon \DCat \to \CCat$  is given by $N(c)(d) = \CCat(F d, c) $ for all $c \in \CCat$ and $ d \in \DCat$. 
	
\end{defn}

Let $\iota \colon \DCat \hookrightarrow \CCat$ be an inclusion of small categories. By  \cite[Proposition~5.1]{Lam66}, the induced nerve $N \colon \CCat \to \pr{\DCat}$ is fully faithful if and only if every object $c \in \CCat$ is canonically the colimit of the forgetful functor $\DCat \ov c \to \CCat$, $ (f \colon \iota (d) \to c) \mapsto \iota (d)$. In this case, the inclusion $\iota \colon \DCat \hookrightarrow \CCat$ is called \textit{dense}.

\begin{ex}\label{ex: cat nerve}
	The nerve of a small category $\CCat$ is described by the image of $\CCat$ under the nerve functor induced by the dense inclusion $\Delta \hookrightarrow \Cat$, where $\Cat$ is the category of small categories, and $\Delta$ is the \textit{simplex category} of non-empty finite ordinals and order preserving morphisms. 
\end{ex}

\subsection{Weber's abstract nerve theory}\label{ssec: Weber}

Every functor admits an (up to isomorphism) unique \textit{bo-ff factorisation} as a bijective on objects functor followed by a fully faithful functor. For example, if $\MM = (M, \mu, \eta)$ is a monad on a category $\CCat$, and $\CCat^{\MM}$ is the EM category of algebras for $\MM$, then the free functor $ \CCat \to \CCat^\MM$ has bo-ff factorisation $\CCat \rightarrow \CCat_{\MM} \rightarrow \CCat^{\MM}$, where $\CCat_{\MM}$ is the {Kleisli category} whose objects are the \emph{free $\MM$-algebras} of the form $(Mc, \mu^{\MM}(c))$ for $c \in \CCat$. See e.g.\ \cite[Section~VI.5]{Mac98}.

Let $\DCat \subset \CCat$ be a dense subcategory and $\Theta_{\MM, \DCat}$ the full subcategory of $\CCat_\MM $ obtained in the bo-ff factorisation of the functor $\DCat \hookrightarrow \CCat \xrightarrow {\text{ free}} \CCat^\MM$. 
Consider the following diagram:  
\begin{equation} \label{eq: arities}
	\xymatrix{ 
		\boffcat\ar[rr]^{\text{f.f.}}						&& 	\CCat^\mathbb{M}\ar@<2pt>[d]^-{\text{forget}}
		\ar [rr]^-{N_{\mathbb{M}, \DCat}}				&&\pr{	\boffcat}\ar[d]^{j^*}	\\
		\DCat\ar@{^{(}->} [rr]^-{\text{}}_-{\text{}}	 \ar[u]_{j}^{\text{b.o.}}		&&\CCat \ar [rr]_{\text{f.f.}}^-{N_{\DCat}}	 \ar@<2pt>[u]^-{\text{free}}	
		&&\pr{\DCat}.
	}
\end{equation} where  $j^*$ is the pullback of the bijective on objects functor $j \colon \DCat \to \Theta_{\MM, \DCat} $. The left square of (\ref{eq: arities}) commutes by definition, and the right square commutes up to natural isomorphism (c.f.,~\cite[Section~1.6]{BMW12}).

If the canonical colimit cocones $\DCat \ov c$ lift to colimit cocones in $\CCat^{\MM}$, then $\DCat$ is said to \textit{provide arities for $\MM$}. In this case, the fully faithful functor $\boffcat \to \CCat^{\MM}$ is also dense and so the induced nerve $N_{\MM, \DCat} \colon \CCat^{\MM} \to \pr{\Theta_{\MM, \DCat}}$ is fully faithful. Moreover, by \cite[Section~4]{Web07}, the image of $N_{\MM, \DCat}$ in $ \pr{	\boffcat}$ consists of precisely those objects that are mapped, up to equivalence, by $j^*$ to the image of $N_{\DCat}$ in $\pr{\DCat}$. This is the \textit{Segal condition} for the nerve theorem.

More generally, if $\DCat\subset \CCat$ is a dense subcategory and $\MM$ is a monad on $\CCat$ such that the induced nerve $N_{\mathbb{M}, \DCat}$ in (\ref{eq: arities}) is fully faithful and the Segal condition is satisfied, then $\MM$ is said to be \emph{$\DCat$-nervous} \cite{BG19}. Having arities is a sufficient, but not necessary condition for a monad to be nervous. Indeed, neither the modular operad monad of \cite{Ray20} nor the circuit algebra monad constructed in this work (\cref{sec. iterated}) have arities, but they are both nervous (c.f., \cref{sec. nerve} and \cite[Section~8]{Ray20}).

\begin{rmk}\label{rmk. groth} In many cases of interest, $\CCat = \pr{\mathcal E}$ is the category of presheaves on some category $\mathcal E$. As the Yoneda embedding $\mathcal E \to \CCat$ is fully faithful and dense, any full subcategory $\DCat \subset \CCat$ that contains the representables is also dense in $\CCat$ (and $\mathcal E \hookrightarrow \DCat$ embeds as a full dense subcategory) so that the induced nerve $ \CCat \to \pr{\DCat}$ is fully faithful. If $\MM$ is a $\DCat$-nervous monad on $\CCat $, then the Segal condition has a particularly nice form: A presheaf $P$ on $\boffcat$ is equivalent to the nerve of an $\MM$-algebra if and only if, for all $d \in \DCat$, 
	\begin{equation}
		\label{eq general Segal}
		P(jd) \cong \mathrm{lim}_{(x, f) \in \mathcal E \ov d}j^*P(x). 
	\end{equation}
	Bourke and Garner \cite{BG19} call this situation \textit{the presheaf context}.
	
	Both the original Segal condition for the simplicial nerve theorem for small categories, and a Segal condition for the dendroidal nerve of an operad, may be obtained from monads with arities in the presheaf context (see e.g.,~\cite[Section~4]{Web07} or \cite[Section~2]{JK11} for details). The presheaf context also applies for the nerve theorems for modular operads \cite[Theorem~8.2]{Ray20} and circuit algebras (\cref{thm: CO nerve}), though the corresponding monads on $\GS$ do not have arities in these cases. 
\end{rmk}

\begin{rmk}
	\label{rmk strongly cartesian} (See also Remarks~\ref{rmk. groth}, \ref{rmk patterns proof} and \ref{rmk bmw proof}.) A special class of monad with arities is described by \textit{strongly cartesian monads} \cite{BMW12} (called \textit{parametric right adjoint} monads in \cite{Web07} and \textit{polynomial monads} in \cite{CH21}). These are \textit{cartesian} monads (they satisfy certain pullback conditions) that are \textit{locally right adjoint} (certain functors of slice categories are right adjoints). For a strongly cartesian monad $\MM$ on $\CCat$, there is a canonical choice of arities $\DCat_{\MM}\subset \CCat$ (see e.g.,~\cite[section~2]{BMW12}). Strongly cartesian monads on presheaf categories --with their associated Segal condition \ref{eq general Segal} -- are central to the theory of \textit{algebraic patterns} developed in \cite{CH21}.
	
	Let $\MM$ be a strongly cartesian monad on $\CCat  = \pr{\mathcal E}$ and let $\DCat \subset \CCat$ provide canonical arities for $\MM$. Consider the diagram (\ref{eq: arities}) for $\MM$ and $\DCat$. In this case, $\boffcat$ is dense in $\CCat^{\MM}$, the induced nerve is fully faithful, and $P \in \pr{\boffcat}$ is, up to equivalence, in the image of the nerve $N _{\MM, \DCat}$ if and only if $P$ satisfies (\ref{eq general Segal}). Since $\MM$ is strongly cartesian, there is, by \cite[Proposition~2.6]{Web07}, a wide subcategory $\boffcat^{\mathrm{ref}} \subset \boffcat$ of so-called \textit{generic morphisms} or \textit{refinements} \cite{JK11} in $\boffcat$ such that every morphism in $\boffcat$ factors as a refinement followed by a morphism in the image of $\DCat$. 
	
	The strongly cartesian monad $\MM$ on $\CCat = \pr{\mathcal E}$ then has the form
\begin{equation}\label{eq. poly}
		M(P)(x) \cong \mathrm{colim}_{(y,f) \in x \ov \boffcat^{\mathrm{ref}} } \mathrm{lim}_{(e,\alpha) \in \mathcal E \ov_{\DCat} y} P(e), \qquad \text{ for all } P \in \CCat, x \in \mathcal E.
\end{equation}

In the terminology of \cite{CH21}, the factorisation system on $\boffcat$, together with the underlying category $\mathcal E$, then describes an \textit{extendable algebraic pattern} with corresponding \textit{free Segal $\boffcat$-space monad} $\MM$. (See Remarks \ref{rmk patterns} and \ref{rmk patterns proof} for an outline of how this applies to the nerve theorems for circuit algebras and modular operads.)
\end{rmk}

\subsection{Applying distributive laws to prove nerve theorems}\label{ssec. combine}
The proof of the nerve theorem \cref{thm: CO nerve} for circuit algebras, like the proof of the modular operad nerve theorem of \cite{Ray20}, exploits the distributive law decompositions of the corresponding composite monad. 

Let $\MM$ be a monad on an essentially small category $\CCat$ with subcategory $\DCat \subset \CCat$, and let $\boffcat\subset \CCat^{\MM}$ be the subcategory obtained in the bo-ff factorisation of $\DCat \hookrightarrow \CCat \to \CCat^{\MM}$. 
Even if $\DCat$ itself does not provide arities for $\MM$, if $\MM  = \MM_n \dots \MM_1$ is a composite monad, then the decomposition of $\MM$ may provide other opportunities to apply the results of \cite{BMW12} in order to prove that the induced nerve $N_{\MM, \DCat} \colon \CCat^{\MM} \to \pr{\Theta_{\MM, \DCat}}$ is fully faithful:

\begin{ex}\label{ex. lifted arities}
	In the simplest case, let $\MM = \MM_2 \MM_1$ be a composite monad on $\CCat$ induced by a distributive law $\lambda \colon M_1 M_2 \Rightarrow M_2 M_1$, and let $\widetilde\MM_2$ be the lift of $\MM_2$ to $\CCat ^{\MM_1}$.  Let $\DCat \subset \CCat$ be a full subcategory, $
	\boffcat$ be the category obtained in the bo-ff factorisation of the functor $ \DCat \to \CCat ^{\MM}$ and let $\DCat_{\MM_1} \subset \CCat^{\MM_1}$ be the full subcategory such that the bijective on objects functor $j \colon \DCat \to \boffcat$ factors through $\DCat_{\MM_1}$. This is illustrated in Diagram (\ref{eq. lift diag}), with $\MM' = \MM_1$ and the free and forgetful functors $\CCat^{\MM'} = \CCat^{\MM_1} \leftrightarrows \CCat^{\MM}$ corresponding to the monadic adjunction associated to the lift $\widetilde\MM_2$ on $\CCat^{\MM_1}$. The functor $\DCat_{\MM_1 }\to \CCat^{\MM}$ induced by the lifted monad $\widetilde\MM_2$ also has bo-ff factorisation through $\boffcat$. So, if $\DCat_{\MM_1}$ provides arities for $\widetilde\MM_2$, then the nerve $N_{\MM, \DCat}$ is fully faithful.
	\begin{equation}
		\label{eq. lift diag}
		\xymatrix@R = .3cm{ 
			\boffcat \ar[rr]^{\text{f.f.}}						&& 	\CCat^{\MM}\ar@<2pt>[dd]^-{\text{forget}}
			\ar [rr]^-{N_{\MM, \DCat}	}		&&\pr{	\boffcat}\ar[dd]\\ &&&&\\
			\DCat_{\MM'}\ar[rr]^{\text{f.f.}}			 \ar[uu]^{\text{b.o.}}		&& 	\CCat^{\MM'}\ar@<2pt>[dd]^-{\text{forget}}\ar@<2pt>[uu]^-{\text{free}}
			\ar [rr]^-{N_{\MM', \DCat}}				&&\pr{\DCat_{\MM'}}\ar[dd]	\\&&&&\\
			\DCat\ar@{^{(}->} [rr]^-{\text{f.f.}}_-{\text{}}	 \ar[uu]^{\text{b.o.}}		&&\CCat \ar [rr] ^-{N_{\DCat}}	\ar@<2pt>[uu]^-{\text{free}}	
			&&\pr{\DCat}.
		}
	\end{equation}

	More generally, let $\MM  = \MM_n \dots \MM_1$ be a composite monad on $\CCat$ and let $\DCat \subset \CCat$ be a full subcategory. For each $1 \leq i < n$, the monads $\MM_n\dots \MM_{i+1}$ and $\MM_i \dots \MM_1$ on $\CCat$ are composable. In particular, $\MM_n\dots \MM_{i+1}$ lifts to a monad $\widetilde\MM_{(i+1)}$ on the EM category $\CCat^{\MM_i \dots \MM_{i+1}}$. 
	
	There is a factorisation of the diagram (\ref{eq: arities}) as
	\begin{equation}
	\label{eq. iterated arities} 
	\xymatrix@R = .3cm{ 
		\boffcat \ar[rr]^{\text{f.f.}}		\ar@{=}[d]				&& 	\CCat^{\MM}\ar@{=}[d]
		\ar [rr]^-{N_{\MM,\DCat} 	}		&&\pr{	\boffcat }\ar@{=}[d]	\\
		\Theta_{\MM_n \dots \MM_1, \DCat}\ar[rr]^{\text{f.f.}}						&& 	\CCat^{\MM_n \dots \MM_1}\ar@<2pt>[dd]^-{\text{forget}}
		\ar [rr]^-{N_{\MM_n \dots \MM_1,\DCat} 	}		&&\pr{		\Theta_{\MM_n \dots \MM_1,\DCat}}\ar[dd]^-{j_{n}^*}	\\ &&&&\\
		\vdots  \ar[uu]_-{j_n}^{\text{b.o.}}&& \vdots \ar@<2pt>[uu]^-{\text{free}}\ar@<2pt>[d] 
		&& \vdots \ar[d]\\	
		\Theta_{\MM_1, \DCat}\ar[rr]^{\text{f.f.}}			\ar[u]			&& 	\CCat^{\mathbb{M}_1}\ar@<2pt>[dd]^-{\text{forget}}\ar@<2pt>[u]
		\ar [rr]^-{N_{\MM_1, \DCat}}				&&\pr{\Theta_{\MM_1, \DCat}}\ar[dd]^-{j_1^*}	\\&&&&\\
		\DCat\ar@{^{(}->} [rr]^-{\text{f.f.}}_-{\text{}}	 \ar[uu]_-{j_1}^{\text{b.o.}}		&&\CCat \ar [rr]_{\text{}}^-{N_{\DCat}}	 \ar@<2pt>[uu]^-{\text{free}}	
		&&\pr{\DCat}.
	}\end{equation}
	
	The induced nerve $N_{\MM, \DCat}$ is fully faithful if, for some $1 \leq i < n$, the lift of $\MM_n \dots \MM_{(i+1)}$ to $ \CCat^{\MM_i \dots \MM_1}$ has arities $	\Theta_{\MM_i\dots \MM_1, \DCat}$ (as illustrated in (\ref{eq. iterated arities})).
\end{ex}

\begin{ex}
	\label{ex monad extension}	\label{ex nerve MO}(See also \cref{ex dist MO}.)  Once again, let $\MM = \MM_2 \MM_1$ be a composite monad on $\CCat$ induced by a distributive law $\lambda \colon M_1 M_2 \Rightarrow M_2 M_1$. 
	If the monad $\MM_1$ extends to a monad ${\MM_{1}}_*$ on $\CCat^ {\MM_2}$ such that $(\CCat^ {\MM_2})^{{\MM_{1}}_*} \cong \CCat^{\MM_2\MM_1}$, then, by letting $\MM' = \MM_2$ in Diagram (\ref{eq. lift diag}), the same technique as in \cref{ex. lifted arities} may be applied. 
	
	For example, in \cite[Theorem~7.46]{Ray20}, the modular operad monad on $\GS$ is obtained, via a distributive law, as a composite $\DD\TT$. The proof of the nerve theorem and Segal condition (\cite[Theorem~8.2]{Ray20}) follows from the fact that the monad $\TT$ extends to a monad $\TTp$, on the category $\GSp$ of $\DD$-algebras, whose EM category of algebras coincides with the EM category of algebras for $\DD\TT$ on $\GS$.  Whilst $\DD\TT$ does not have arities, $\TTp$ does. 

\end{ex}

		In \cref{sec. nerve}, the proof of \cref{thm: CO nerve} extends the proof of the modular operadic nerve theorem (\cref{ex nerve MO}) but  requires a composite $\LL\DD\TT$, obtained via iterated distributive laws, of three monads on $\GS$, and uses both lifts and extensions of monads. The general case is outlined in the following example.
		
		\begin{ex}
			\label{ex. combined arities}Examples \ref{ex. lifted arities} and \ref{ex nerve MO} may be combined: Let $\MM = \MM_3 \MM_2 \MM_1$ be a monad on $\CCat$ and let $\DCat \subset \CCat$ be a full and dense subcategory. Assume that $\MM_1$ extends to a monad $\MM_{1 *} $ on the EM category $\CCat^{\MM_2}$ of algebras for $\MM_2$ and that $\CCat^{\MM_2 \MM_1}\cong (\CCat^{\MM_2})^{\MM_{1 *}}$. By the theory of iterated distributive laws \cite{Che11}, $\MM_3$ lifts functorially to both $\CCat^{\MM_1}$ and $\CCat^{\MM_2 \MM_1}$. Hence, since $\CCat^{\MM_2 \MM_1}\cong(\CCat^{\MM_2})^{\MM_{1 *}}$, the distributive law $M_1 M_3 \Rightarrow M_3 M_1$ between the monads $\MM_1$ and $\MM_3$ on $\CCat$ lifts to a distributive law between the extension ${\MM_{1 *}}$ of $\MM_1$, and the lift $\widetilde {\MM}_3$ of $\MM_3 $ to $\CCat^{\MM_2}$. Moreover, algebras for $ \widetilde {\MM}_3 {\MM_{1 *}}$ on $\CCat^{\MM_2}$ correspond to algebras for $\MM$ on $\CCat$. 
			
			It follows that the nerve induced by the bo-ff factorisation $\DCat \to \boffcat \to \CCat^{\MM}$ is fully faithful if $ \widetilde {\MM}_3 {\MM_{1 *}}$ has arities obtained from the full subcategory $\DCat_{\MM_2} \subset \CCat^{\MM_2}$ with the same objects as $\DCat$ (as in Diagram (\ref{eq. lift diag}) with $\MM' = \MM_2$).
			
		\end{ex}
		
		\begin{rmk}\label{rmk Segal suff}
			
			It is important to note that the conditions in each of the Examples \ref{ex. lifted arities}-\ref{ex. combined arities} are not sufficient to establish a Segal condition in terms of presheaves on the category $\DCat$. The Segal condition in \cite[Theorem~8.2]{Ray20} and \cref{thm: CO nerve}, follows from the observation that the embedding $\fisinv \hookrightarrow \Gretp$ (described in \cref{ss. pointed}) is dense.
		\end{rmk}

	Following Remarks \ref{rmk. groth} and \ref{rmk strongly cartesian}, I conclude this section with remarks on related results, specific to the presheaf context, that can be used to obtain additional proofs of the nerve theorem \ref{thm: CO nerve}. This is further explained in Remarks \ref{rmk patterns proof} and \ref{rmk bmw proof}.
		
		\begin{rmk}
		\label{rmk BMW prop} Let $k \colon \mathcal E_1 \to \mathcal E_2$ be a bijective on objects functor and, for $i = 1,2$, let $\MM_i$ be a monad on $\CCat_i = \pr{\mathcal E_i}$ such that the pullback $k^* \colon \CCat_2 \to \CCat_1$ describes a monad morphism $(\CCat_1, \MM_1) \to (\CCat_1, \MM_1)$ by $ k^* \MM_2 = \MM_1 k^*$. By \cite[Proposition~4.4]{BMW12}, if $\MM_1$ is strongly cartesian, then so is $\MM_2$ and the canonical arities for $\MM_2$ may be obtained from those of $\MM_2$ by extension along $k$.

	In the case of the circuit algebra and modular operad monads $\LL\DD\TT$ and $\DD\TT$ on $\GS$, the forgetful functor $\GSp \to \GS$ associated to the monad $\DD$ on $\GS$ is the pullback of the bijective on objects inclusion $k \colon \fisinv \hookrightarrow \fisinvp$ described in \cref{ssec. D monad}. The corresponding nerve theorems then follow, as in \cref{rmk bmw proof}, by showing that the monads $\LL \TT$ and $\TT$ are strongly cartesian on $\GS$ and admit extension, along $k$, to $\GSp$.

		\end{rmk}
		\begin{rmk}
			\label{rmk patterns} Recall Remarks \ref{rmk. groth} and \ref{rmk strongly cartesian}. Let $\CCat  = \pr{\mathcal E}$ and $\DCat$ be a full subcategory containing the representables. Further, let $\MM$ be a $\DCat$-nervous monad on $\CCat$. So, the nerve $N_{\MM, \DCat}\colon \CCat^{\MM} \to \pr{\boffcat}$ in Diagram (\ref{eq: arities}) is fully faithful and presheaves in its essential image satisfy the presheaf Segal condition (\ref{eq general Segal}). Assume also that $\DCat$ is the smallest full subcategory of $\CCat$ with respect to which $\MM$ is nervous. 
			
	Now, let $\boffcat$ admit an algebraic pattern structure -- with elementary objects in $\mathcal E$ -- as in \cite{CH21}. This means that there is a factorisation system $(\boffcat^{\mathrm{ref}}, \DCat')$ on $\boffcat$ such that presheaves in the essential image of $N_{\MM, \DCat}$ must also satisfy a second Segal condition 
	\begin{equation}
		\label{eq. pattern Segal} 
			P(j'd) \cong \mathrm{lim}_{(x, f') \in \mathcal E' \ov d}{j'}^*P(x), 
	\end{equation}
	where $j' \colon \DCat' \hookrightarrow \boffcat$ is the bijective on objects inclusion and $\mathcal E' \subset \DCat'$ is the full subcategory on the objects from $\mathcal E$. 
	Moreover, if the factorisation of $\boffcat$ describes an extendable algebraic pattern as in \cite[Proposition~8.8]{CH21}, then there is a strongly cartesian monad $\MM'$ on $\CCat' \defeq \pr{\mathcal E'}$, given by 
	\[	M'(P)(x) \cong \mathrm{colim}_{(y,f) \in x \ov \boffcat^{\mathrm{ref}} } \mathrm{lim}_{(e,\alpha') \in \mathcal E' \ov y} P(e), \qquad \text{ for all } P \in \CCat, x \in \mathcal E,\]
	whose EM category of algebras is precisely $\CCat^{\MM}$.
	
	This is summarised in the following diagram, where -- as before (\ref{eq: arities}) -- the left squares commute and the right squares commute up to natural isomorphism: 
	\begin{equation}
		\label{eq pattern arities}
			\xymatrix{ 
					\mathcal E' \ar[rr]^-{\text{f.f.}}_-{\text{dense}}&&
				\DCat'\ar@{^{(}->} [rr]^-{\text{f.f.}}_-{\text{dense}} \ar[d]^{j'}_{\text{b.o.}}		&&\CCat' \ar [rr]_{\text{f.f.}}^-{N_{\DCat'}}	 \ar@<-2pt>[d]_-{\text{free}^{\MM'}}	
				&&\pr{\DCat'}\\
		&&	\boffcat\ar[rr]^{\text{f.f.}}						&& 	\CCat^\mathbb{M}\ar@<2pt>[d]^-{\text{forget}^{\MM}}\ar@<-2pt>[u]_-{\text{forget}^{\MM'}}
			\ar [rr]^-{N_{\mathbb{M}, \DCat}}				&&\pr{	\boffcat}\ar[d]^{j^*}\ar[u]_{{j'}^*}	\\ 
			\mathcal E \ar@{..>}[uu] \ar[rr]^-{\text{f.f.}}_-{\text{dense}}&&
			\DCat\ar@{..>}@/^3.0pc/[uu]\ar@{^{(}->} [rr]^-{\text{f.f.}}_-{\text{dense}} \ar[u]_{j}^{\text{b.o.}}		&&\CCat \ar [rr]_{\text{f.f.}}^-{N_{\DCat}}	 \ar@<2pt>[u]^-{\text{free}^{\MM}}	
			&&\pr{\DCat}.
		}		
	\end{equation}
	
	For example, the situation of \cref{rmk BMW prop} is obtained if there is a functor $i\colon \DCat \to \DCat'$ such that $j'\circ i = j$, and $\MM'$ is the extension -- along the restricted pullback $i_{\mathcal E}^* \colon \CCat' \to \CCat$ - of a strongly cartesian monad $\MM_0$ on $\CCat$. 

It is natural to ask under what conditions such a functor $i\colon \DCat \to \DCat'$ exists, and, furthermore whether it is possible to classify the nervous monads $\MM$ for which (as with the circuit algebra and modular operad nerve constructions of \cref{thm: CO nerve} and \cite[Theorem~8.2]{Ray20}) there exists an iterated distributive law describing $\MM$ from which the strongly cartesian monad $\MM'$ on $\CCat'$ may be obtained as an extension (and/or lift) as in Examples \ref{ex. lifted arities}-\ref{ex. combined arities}.

		\end{rmk}

\section{Graphical species, modular operads, and circuit algebras}\label{sec: definitions}

Given a set $\DDD$, let $\listm (\DDD) =\coprod_{n \in \N} \DDD^n$ be the set of finite ordered sets $\ccc  = (c_1, \dots, c_n)$ of elements of $\DDD$ (so $\listm (\DDD)$ is the set underlying the free associative monoid on $\DDD$).
An \emph{(involutive) palette} is a pair $(\CCC, \omega)$ of a set $\CCC$ with involution $\omega \colon \CCC \to \CCC$ (see also \cite{Ray20, RayCA1}).

Let $(\CCC, \omega)$ be a palette. In \cite{RayCA1}, $(\CCC, \omega)$-coloured circuit algebras and modular operads \textit{enriched }in a symmetric monoidal category $(\V, \otimes, I)$ were characterised axiomatically, as symmetric $\listm(\CCC)$-graded monoids in $\V$ equipped with a contraction operation satisfying certain axioms (see \cite[Proposition~5.4]{RayCA1} and \cite[Definition~5.7]{RayCA1}). 

This section considers structures that satisfy the same 
axioms, but now defined \emph{internally} to a category $\E$ with sufficient (co)limits. Simply put, this means that the ``grading'' is now by some involutive object of $\V$. When $\V = \E = \Set$, the enriched and internal definitions coincide.

\subsection{Symmetric species}\label{ssec. species}	\label{ssec. Symmetric objects}

Let $\fin $ denote the category of finite sets and functions and let $\fiso \subset \fin$ be the groupoid of finite sets and bijections.
For each $n \in \N$, let $ \nn$ denote the set $\{1, 2, \dots, n\}$ (so $\nul = \emptyset$), and let $ \Sigma_n$ denote the group of permutations on $\nn$.  Let $\Sigma = \bigoplus_n \Sigma_n$ be the  \emph{symmetric groupoid} (with $\Sigma(n,n) = \Sigma_n$).

As $\Sigma$ is a skeleton for $\fiso$, for any category $\E$, a $\E$-valued presheaf $P\colon \fiso^\mathrm{op} \to \E$, also called a \emph{(monochrome} or \emph{single-sorted) species \cite{Joy81} in $\E$}, determines a presheaf $\Sigma^{\mathrm{op}} \to \E$ by restriction. 
	And conversely, a $\Sigma$-presheaf $Q$ may always be extended to a $\fiso$-presheaf $Q_{\fiso}$, by setting
		\begin{equation} \label{eq. sigma to fiso}Q_{\fiso}(X) \defeq \mathrm{lim}_{(\nn,f) \in \Sigma \ov X} Q(n) \ \text{ for all } n \in \N.\end{equation}
		Hence, a species $P\colon \fiso^\mathrm{op} \to \E$ is equivalently a sequence $(P(n))_n$ of objects of $\E$ such that $\Sigma_n$ acts on $ P(n)$ for all $n$.

	For a set $\DDD$, the associative (concatenation) product $\ccc \ddd = \ccc\oplus \ddd$ on $\listm (\DDD)$ is given by \[  \ccc \ddd  \defeq (c_{1}, \dots, c_{m}, d_1, \dots, d_n)  \text{ for } \ccc  = (c_1, \dots, c_m), \ddd =(d_1, \dots, d_n) \text{ in } \listm(\DDD) .\] The empty list is the unit for $\oplus$ and is denoted by $\varnothing$ (or $\varnothing_{\DDD}$).
	
	The symmetric groupoid $\Sigma$ acts on $\listm (\DDD)$ from the right by $\sigma \colon  (\ccc \sigma) \defeq (c_{\sigma 1}, \dots, c_{\sigma n}) \mapsto \ccc $, for all $\ccc= (c_1, \dots, c_n)$ and $\sigma \in \Sigma_n$. The monoidal groupoid so obtained is  the \emph{free symmetric groupoid on $\DDD$} denoted by $\Sigma^{\DDD}$. In particular, if $\DDD = \emptyset$, then $\Sigma ^{\DDD} = \{\varnothing\}$.
	
	An \emph{$\DDD$-coloured species in $\E$} is a $\E$-valued presheaf on $\Sigma^{\DDD}$. This is equivalently described by a collection $(P(\ccc))_{\ccc \in \listm(\DDD)}$ of objects of $\E$ such that $\Sigma_n$ (right) acts on $ P(n) = \coprod_{\ccc \in \DDD^{n}} P'(\ccc)$ for all $n$.

\subsection{Graphical species}
\label{ssec: graphical species}
Graphical species in $\Set$, introduced in \cite{JK11} to describe various generalisations of (coloured) operads, generalise the notion of coloured species from \cref{ssec. species} and have been used to describe coloured modular operads (called \textit{compact closed categories} in \cite{JK11}) in \cite{JK11, Ray20, HRY19a, HRY19b}. For more details on graphical species in $\Set$, particularly regarding the involutive palettes $(\CCC, \omega)$ that encode their expressive power, the reader is referred to \cite[Section~1]{Ray20}.

This section provides a short discussion on graphical species taking values in an arbitrary category $\E$. For convenience, $\E$ will always be assumed to have sufficient limits and colimits. The terminal object of $\E$ is denoted by $*$ and the initial object is denoted by $0 = 0_{\E}$.

As before, denote the groupoid of finite sets and bijections by $\fiso$. The category $\fisinv$ (called $\bm{elGr}$ in \cite{JK11}) is obtained from $\fiso $ by adjoining a distinguished object $\S$ that satisfies 
\begin{itemize}
	\item $\fisinv (\S, \S) \cong \Z/2\Z$ with generator $\tau$;
	\item for each finite set $X$ and each element $x \in X$, there is a morphism $ch_x \in \fisinv (\S, X)$ that ``chooses'' $x$, and $\fisinv(\S, X)= \{ch_x\ ,\ ch_x \circ \tau \}_{x \in X} \cong X \amalg X$;
	\item for all finite sets $X$ and $Y$, $\fisinv(X,Y) = \fiso (X,Y)$;
	\item morphisms are equivariant with respect to the action of $\fiso$. That is, $ ch_{f(x)} = f \circ ch_x \in \fisinv (\S, Y)$ for all $ x \in X$ and all bijections $f\colon X \xrightarrow {\cong} Y$.
\end{itemize}

\begin{defn}\label{defn: graphical species}
	A \emph{graphical species $S$ in a category $\E$} is a presheaf $S\colon {\fisinv}^\mathrm{op} \to \E$. The category of graphical species in $\E$ is denoted by $\GSE\defeq\prE{\fisinv}$. When $\E = \Set$, we write $\GS \defeq \GSE[\Set]$.

\end{defn}

So, a graphical species $S$ is described by a $\E$-valued species $(S_X)_{X \in \fiso}$, and an object $S_\S \in \E$ with involution $S_\tau \in \E(S_\S, S_\S)$, 
together with a family of projections $S(ch_x)\colon S_X \to S_\S$, defined for all finite sets $X$, and all $x \in X$, and equivariant with respect to the actions of $\fiso$ and $\tau$.  In the definitions of circuit algebras and modular operads (Sections \ref{subs. MO and Comp}~\&~\ref{ssec CO}), the objects $S_X$ describe the (modular operad) operations, or morphisms, whereas the object $S_\S$ can be thought of as describing the ``colours'', or generating the ``object object''. 

\begin{ex}\label{Comm}
	The terminal graphical species $\CommE$ is the constant (monochrome) graphical species that sends $\S$ and all finite sets $X$ to the terminal object $*$ in $\E$.
	
\end{ex}

\begin{defn}\label{def: GS palette}

A morphism $ \gamma \in \GSE(S, S' )$ is \emph{colour-} or \emph{palette-preserving} if its component at $\S$ is the identity on  $({S_\S}, S_\tau)$.  A graphical species $S \in \GSE$ is called \textit{monochrome} if $S_\S = *$.

\end{defn}

If $(\CCC, \omega)$ is an involutive palette in $\Set$, objects of the subcategory $\CGS \subset \GS$ of \emph{$(\CCC, \omega)$-coloured graphical species in $\Set$} are graphical species $S \in \GS$ such that $(S_\S, S_\tau) = (\CCC, \omega)$.  
Morphisms are palette-preserving morphisms. 

\begin{defn}\label{c arity}

	Let $S \in \CGS$. For a finite set $X$, $ S_X =\coprod_{\ccc \in \CCC^X} S_{\ccc}$, where, for each $\ccc \in \CCC^X$, the set $S_{\ccc}$ is the \emph{$\ccc$-(coloured) arity of $S$} defined as the fibre above $ \ccc $ of the map $(S(ch_x))_{x \in X}\colon {S_X} \to\CCC^{X}$.

\end{defn}

 It follows that $S$ defines a $\CCC$-coloured species $S' \colon (\Sigma^{\CCC})^{\mathrm {op}}\to \Set$ by restriction. Conversely, for any involutive palette $(\CCC, \omega)$, a $\CCC$-coloured species $P \colon (\Sigma^{\CCC})^{\mathrm{op}} \to \Set$ extends to a graphical species $\hat P \in \CGS$ by (\ref{eq. sigma to fiso}).

\begin{ex}\label{ex. terminal coloured species} 
For any palette $(\CCC, \omega)$, the $(\CCC, \omega)$-coloured graphical species $\CComm{}$ given by $\CComm{\ccc}= \{*\}$ for all $\ccc \in \listm(\CCC)$ is terminal in $\CGS$. If $(\CCC, \omega) = \emptyset$, then $K^{\emptyset}_{\nul} = \{*\}$ and $K^{\emptyset}_X = \emptyset $ for $X \not \cong \nul$.

\end{ex}

The element category (\cref{defn: general element}) of a $\Set$-valued graphical species $S \in \GS$ is denoted by $\elG[S] \defeq \ElP{S}{\fisinv}$.

\begin{ex}\label{ex directed graphical species}
Let $\Dipal$ (called $\dipal$ in \cite{RayCA1}) be the involutive palette in $\Set$ described by the unique non-trivial involution $\Disig$ on the two-element set $\Di = \{ \In, \ \Out\}$. This was used in \cite[Section~3.2]{RayCA1} to describe (monochrome) oriented circuit algebras (monochrome wheeled props), and in \cite[Section~1.1]{Ray20} to describe props and wheeled properads.

If $\Dicomm$ is the terminal $\Dipal$-coloured 
graphical species in $\Set$, then, for all finite sets $X$, $\Dicomm_{X}=\Di^X$ is the set of partitions of $X =X_{\In}\amalg  X_{\Out}$ into \textit{input} and \textit{output} sets and morphisms $ \Dicomm_X \to \Dicomm_Y$ are bijections $X \xrightarrow{\cong} Y$ that preserve the partitions. Objects of the category $\CGS[\Dipal]$ are called \emph{monochrome oriented graphical species} in $\Set$. 

Since $(\In) \cong (\Out)$ in $\elG[\Dicomm]$, the element category $\elG[\Dicomm]$ of $\Dicomm$ is equivalent to the category $(\fiso \times \fiso ^{\mathrm{op}})^+$ obtained from $\fiso \times \fiso ^{\mathrm{op}}$ by adjoining a distinguished object $\oedge$ (the \emph{directed exceptional edge}), and morphisms $i_x, o_y \colon (X_{\In}, X_{\Out})\to (\oedge)$ defined for each $x \in X_{\In}$ and $y \in X_{\Out}$.

\end{ex}

	\begin{defn}
		\label{defn oriented GSE}
		The category $\oGSE$ of \emph{oriented graphical species in $\E$} is the presheaf category $ \prE{\elG[\Dicomm]}$. 
	\end{defn}

	For all $\Set$-valued graphical species $S' \in \GS$, there is a canonical equivalence $\GS \ov S' \cong \pr{\elG[S']}$ such that $\gamma \in \GS(S,S')$ describes a presheaf $\tilde S \in \pr{\elG[S']}$ by
		$\tilde S (\phi) =  \gamma^{-1} (\phi)$ for all $\phi \in \elG[S']$. 
		
		\begin{ex}\label{ex. oriented slice}\label{lem oGSE equiv}(See also \cite[Example~1.10]{Ray20}.)
			Assume that $\E$ has finite products and coproducts. Let $\DicommE \in \GSE$ be the graphical species given by  ${\DicommE}_{\S} = *\amalg*$ (and ${\DicommE}_{\tau}$ acts by swapping the components), and, for each finite set $\X$, ${\DicommE}_{X} = \coprod_{(X_{\In}, X_{\Out})\in \Dicomm_X} (* )$.

			There is a canonical equivalence of categories $\oGSE\simeq \GSE \ov \DicommE$. If $(S, \gamma) \in \GSE \ov \DicommE$, then the corresponding presheaf $\tilde S \in  \prE{\elG[\Dicomm]} = \oGSE$ is such that, for all $\C \in \fisinv$ and all $\phi \in S_{\C}$ (so $\phi $ is either a pair of finite sets or an element of $\{\In, \Out\}$), $\tilde S(\phi)$ is given by the pullback 
			\[\xymatrix{\tilde S (\phi) \ar[rr] \ar[d]&& S_{\C} \ar[d]^{\gamma_{\C}}\\ {*} \ar[rr]_-{inc_{\phi}} && \DicommE_{\C}*}\] where  $inc_{\phi}  $ is the obvious map that picks out the component indexed by $\phi$ in $ \DicommE_{\C} = \coprod_{\Dicomm_{\C}} *$. 

		\end{ex}

Since $\E$ is always assumed to have sufficient (co)limits, these two descriptions of oriented graphical species  -- as $\E$-valued presheaves on $\elG[\Dicomm]$ or as objects of $\GSE$ with morphisms to $\DicommE$ -- will be used interchangeably in what follows.

\begin{rmk}\label{rmk monoidal}
	
When $\E$ has finite products and coproducts, addition of natural numbers in $\Sigma$ induces a 
	symmetric monoidal product on the category of species in $\E$ by 
	\begin{equation}\label{eq cauchy prod}
		(S \otimes T)_X \defeq \coprod_{(Y,Z), Y \amalg Z = X} S_Y \times T_Z, \text{ for all } X \in \fiso
	\end{equation} for species $S,T$ and all finite sets $X$. It is easily checked (see also \cite[Section~2.1]{RayCA1}) that monoids for this product correspond to (lax) monoidal presheaves on $\Sigma$. 
	
	This extends to a symmetric monoidal structure on $\GSE$ (and also $\oGSE$) whereby, given $S, T \in \GSE$,  the graphical species $S \otimes T \in \GSE$, has $ (S \otimes T)_X  = \coprod_{(Y,Z), Y \amalg Z = X} S_Y \times T_Z$ as in (\ref{eq cauchy prod}) and $(S \otimes T)_\S \defeq S_\S \amalg T_\S$.

By the universal property of the coproduct, the involutions on $S_\S$ and $T_\S$ induce an involution on $(S \otimes T)_\S$ and the $\fiso$-actions on $S$ and $T$ separately induce a $\fiso$-action on the collection  $((S \otimes T)_X)_{X\in \fiso}$. 
For $x \in X$, the morphism $(S\otimes T)(ch^X_x) \colon (S \otimes T)_X \to (S \otimes T)_\S$ is induced by the maps 
\[ (S \otimes T)(ch^{Y \amalg Z}_x) = \left \{ \begin{array}{ll}
	\iota_S\circ S(ch^Y_x) \circ p_{S_Y}\colon S_Y \times T_Z \to S_\S& \text{ when } x \in Y\\
	\iota_T\circ T(ch^Z_x) \circ p_{T_Z}\colon S_Y \times T_Z \to T_\S & \text{ when } x \in Z.
\end{array}\right . \] 
Here, the maps $p_{(-)}$ from $S_Y \times T_Z$ are the projections and $\iota_{(-)}$ are the maps to the coproduct $(S \otimes T)_\S$.

It is straightforward to check that $S \otimes T\in\GSE$ and that $\otimes $ is associative and the graphical species $I $ defined by $I_{\S} = 0$, $I_{\emptyset} = *$ and $I_X = 0$ for $X \neq \emptyset$ (so $ I =I_\E=K^{\emptyset}$ when $\E = \Set$)
defines a unit for $\otimes$, making $(\GSE, \otimes, I)$ into a symmetric monoidal category. 
 
\end{rmk}

\subsection{Modular operads} \label{subs. MO and Comp}

This section generalises the definition of modular operads in $\Set$ from \cite[Section~1.3]{Ray20} to define modular operads \emph{internal to} arbitrary categories with finite limits. Modular operads internal to a category $\E$ are defined as graphical species in $\E$ equipped with a binary (unital) multiplication operation and unary contraction operation satisfying coherence axioms that say, informally, that the order of applying multiplications and contractions does not matter.

Let $S $ be a graphical species in a category $\E$ with finite limits. For finite sets $X$ and $Y$ with elements $x \in X$ and $y \in Y$, let $(S_X \times S_Y)^{x \ddagger y}  \to S_X \times S_Y$ denote the pullback \begin{equation}\label{eq mult pullback}
	\xymatrix{(S_X \times S_Y)^{x \ddagger y} \ar[rr] \ar[d] && S_X \ar[d]^-{ S(ch_x)}\\
S_Y\ar[rr]_-{S(ch_y \circ \tau)}&& S_\S.}
\end{equation}
More generally, given distinct elements $x_1, \dots, x_k \in X$ and $y_1, \dots, y_k \in Y$, let $(S_X \times S_Y)^{x_1 \ddagger y_1, \dots, x_k \ddagger y_k}$ be the limit of the collection $ S_X \xrightarrow{S(ch_{x_i})} S_\S \xleftarrow{S(ch_{y_i} \circ \tau)} S_Y$, $1 \leq i \leq k$. 

 Let $\sigma _\two$ denote the unique non-identity involution on $\two$. A morphism $\epsilon  \in \E(S_\S, S_\two)$ is \emph{unit-like for $S$} if 
 \begin{equation}
 	\label{eq: unit compatible with involution}	
 	\epsilon \circ  \omega = S(\sigma_\two)  \circ \epsilon\colon S_\S \to S_\two, \text{ and } \ 
 	S(ch_1)\circ \epsilon   = id_{S_{\S}}.
 \end{equation} 

\begin{defn}\label{defn: multiplication}\label{coloured mult cont} 		\label{defn: formal connected unit}
	
	A \emph{multiplication} $\diamond$ on $S$ is a family of morphisms 
\[	\diamond_{X,Y}^{x \ddagger y} \colon (S_X \times S_Y)^{x \ddagger y} \to S_{(X \amalg Y) \setminus \{x,y\}},\] in $\E$, defined for all pairs of finite sets $X$ and $Y$, with elements $x \in X$ and $y \in Y$, and such that 
\begin{enumerate}\item the obvious diagram \[
\xymatrix@R = .5cm{ (S_X \times S_Y)^{x \ddagger y} \ar[rrd]^-{\diamond_{X,Y}^{x \ddagger y}}\ar[dd]_{\cong}&&{}\\
	&&\quad  S_{(X \amalg Y)\setminus\{x,y\}}\\
(S_Y \times S_X)^{y \ddagger x} \ar[urr]_-{ \diamond_{Y,X}^{y \ddagger x}}&& {}}\]
commutes in $\E$,
\item $\diamond $ is equivariant with respect to the $\fisinv$-action on $S$: if $ \hat \sigma \colon X \setminus \{x\} \xrightarrow {\cong } W \setminus \{w\}$ and $ \hat \rho \colon Y\setminus \{y\} \xrightarrow {\cong } Z \setminus \{z\}$ are restrictions of bijections $  \sigma\colon X \xrightarrow{\cong} W$ and $ \rho\colon Y \xrightarrow{\cong} Z$ such that $ \sigma(x) = w$ and $ \rho (y) = z$, then 
		\[ S(\hat \sigma \sqcup  \hat \rho )\  \diamond_{W,Z}^{w \ddagger z} = \diamond_{X,Y}^{x \ddagger y} \ S(  \sigma \sqcup  \rho),\] 
		(where $ \sigma \sqcup \rho\colon X \amalg Y \xrightarrow{\cong} W \amalg Z$ is the block-wise permutation).
\end{enumerate}

	A multiplication $\diamond$ on $S$ is \emph{unital} if $S$ there is a unit-like morphism
$\epsilon  \in \E(S_\S, S_\two)$ such that,
for all finite sets $X$ and all $x \in X$, the composite
		\begin{equation}\label{eq. mult unit} \xymatrix{ S_X \ar[rr]^-{ (id \times ch_x)\circ \Delta}&& S_X \times S_\S \ar[rr] ^-{ id \times \epsilon} &&S_X \times S_\two \ar[rr]^-{\diamond_{X, \two}^{ x \ddagger 2}}&& S_X }\end{equation} is the identity on $S_X$.

\end{defn}

It is immediate from the definition (\ref{eq: unit compatible with involution}) that any unit-like morphism in $\epsilon \in \E(S_\S, S_\two)$ is monomorphic and, by (\ref{eq. mult unit}) and condition (1) of \cref{defn: multiplication}, if a multiplication $\diamond$ on $S$ admits a unit $\epsilon$, then it is unique. 
 
Now, let $S$ be a graphical species in $\E$, and let $x \neq y$ be distinct elements of a finite set $Z$. Let $(S_Z)^{x \ddagger y} = S_{Z}^{y \ddagger x}$ denote the equaliser 
 \begin{equation}
 	\label{eq. dagger pullback}\xymatrix{S_Z^{x \ddagger y} \ar[rr]&& S_Z \ar@<4pt>[rr]^-{S(ch_x)} \ar@<-4pt>[rr]_-{S(ch_y \circ \tau)} && S_\S.}
 \end{equation}
 More generally, for any morphism $ f \colon C \to S_Z$ in $\E$, let $C^{x \ddagger y}  = C^{y \ddagger x}$ denote the pullback of $f$ along the universal map $S_Z^{x \ddagger y}\to S_Z$. 
 Given distinct elements $x_1, y_1, \dots, x_k, y_k$ of $S_Z$, and a morphism $ f \colon C \to S_Z$ in $\E$, then $ C^{x_1 \ddagger y_1, \dots ,x_k \ddagger y_k} \defeq (C^{x_1 \ddagger y_1, \dots ,x_{k-1} \ddagger y_{k-1}} )^{x_k \ddagger y_k}$ is the obvious limit.

 Invariance of $C^{x_1 \ddagger y_1, \dots ,x_k \ddagger y_k}$ under permutations $(x_i, y_i) \mapsto (x_{\sigma i}, y_{\sigma i})$, $\sigma \in \Sigma_k$ follows from invariance of the defining morphisms  \[ \xymatrix {C^{x_1 \ddagger y_1, \dots ,x_k \ddagger y_k}\ar[rr]&& S^{x_1 \ddagger y_1, \dots ,x_k \ddagger y_k}_Z \ar[rr]&& S_{Z}.}\]

 \begin{defn}\label{defn: contraction}
 	A \emph{(graphical species) contraction} $\zeta$ on $ S$ in $\E$ is a family of maps $\zeta^{x \ddagger y}_X \colon S_X^{x \ddagger y}\to S_{X \setminus \{x,y\}}$ defined for each finite set $X$ and pair of distinct elements $x,y \in X$. 
 	
 	The contraction $\zeta$ is equivariant with respect to the action of $\fisinv$ on $S$: If $\hat \sigma \colon X \setminus \{x,y\}\xrightarrow{\cong }Z \setminus\{w,z\}$ is the restriction of a bijection $ \sigma\colon X\xrightarrow{\cong} Z $ with $\sigma(x) = w$ and $\sigma(y) = z$. Then 
 	\begin{equation} \label{eq. contraction equi} S(\hat \sigma) \circ \zeta^{w\ddagger z}_Z = \zeta ^{x \ddagger y}_X \circ S( \sigma) \colon S_Z^{w \ddagger z} \to S_{X \setminus\{x,y\}}.\end{equation}

 \end{defn}

 In particular, by (\ref{eq. contraction equi}), if $\zeta$ is a contraction on $S$, then  $\zeta_X^{x\ddagger y}= \zeta_X ^{y\ddagger x}$ for all finite sets $X$ and all pairs of distinct elements $x, y \in X$.

\begin{defn}\label{def. mod op} 
	A nonunital modular operad $(S, \diamond, \zeta)$ in $\E$ is a graphical species $S \in \GSE$ equipped with a multiplication $\diamond$ and a contraction $\zeta$ such that the following four coherence axioms are satisfied:	
	\begin{enumerate}[(M1)]
		\item \emph{Multiplication is associative} (see \cref{fig. MO axioms}(M1)):\\
		For all finite set $X, Y, Z$ and elements $x \in X$, $z \in Z$, and distinct $y_1, y_2 \in Y$, 
		\[
		\xymatrix@C =.8cm@R =.4cm{
			(	S_{X} \times S_{Y} \times S_{Z} )^{x \ddagger y_1 , y_2 \ddagger z}
			\ar[rr]^-{\diamond_{X, Y}^{x \ddagger y_1}\times id_{S_Z}}
			\ar[dd]_{id_{S_X} \times \diamond_{Y,Z}^{y_2\ddagger z} } &&
			\left (S_{(X \amalg Y)\setminus \{x,y_1\}}\times S_{Z} \right)^{y_2 \ddagger z}
			\ar[dd]^{\diamond^{y_2 \ddagger z}_{ X \amalg Y\setminus \{x,y_1\}, Z} }\\
			&{=}&\\
			\left(S_{X}\times  S_{(Y \amalg Z) \setminus \{y_2,z\}}\right)^{x \ddagger y_1}
			\ar[rr]_-{\diamond^{x\ddagger y_1}_{X, Y \amalg Z\setminus \{y_2,z\}}} &&
			S_{(X \amalg Y \amalg Z) \setminus \{x, y_1,y_2,z\}}. }
		\]
		\item \emph{Contractions commute} (see \cref{fig. MO axioms}(M2)): \\For all finite sets $X$ with distinct elements $w,x,y,z$, the following square commutes:
		\[\xymatrix@C =.8cm@R =.4cm{
			S_{X} ^{ w\ddagger x, y \ddagger z}\ar[rr]^-{\zeta^{w \ddagger x}_X}\ar[dd]_{\zeta^{y \ddagger z}_X} && 
			S_{X \setminus \{w,x\} }^{y \ddagger z}\ar[dd]^{\zeta^{y \ddagger z}_{X \setminus \{w,x\}}}\\
			&{=}&\\\
			S_{X \setminus \{y,z\} }^{w \ddagger x}\ar[rr]_-{\zeta^{w \ddagger x}_{X \setminus\{y,z\}}}&&	S_{X \setminus \{w,x,y,z\} }}.\]
			\item \emph{ Multiplication and contraction commute} (see \cref{fig. MO axioms}(M2)): \\For finite sets $X,Y$,  distinct  $x_1, x_2$ and $x_3$ in $X$, and $y \in Y$, the following diagram commutes:
			\[
			\xymatrix@C =.8cm@R =.4cm{
				(	S_{X }^{x_1 \ddagger x_2} \times S_{Y})^{x_3 \ddagger y}
				\ar[rr]^-{\zeta^{x_1 \ddagger x_2}_X \times id_Y}
				\ar[dd]_ {\diamond^{x_3\ddagger y}_{X,Y}} &&
				(	S_{X \setminus\{x_1, x_2\}} \times S_{Y)} )^{x_3 \ddagger y}
				\ar[dd]^{\diamond^{x_3 \ddagger y}_{X\setminus \{x_1,x_2\},Y}}\\
				&{=}&\\
				S_{(X \amalg Y)\setminus \{x_3,y\}}^{x_1 \ddagger x_2}
				\ar[rr]_-{\zeta^{x_1 \ddagger x_2}_{X \amalg Y \setminus\{x_3,y\}}} &&
				S_{(X \amalg Y)\setminus \{x_1, x_2,x_3,y\}}.}\]
				\item  \emph{``Parallel multiplication'' of pairs} (see \cref{fig. MO axioms}(M4)):\\For finite sets $X$ and $Y$, and distinct elements $x_1, x_2 \in X$ and $y_1 , y_2 \in Y$, the following digram commutes:
				\[{
					\xymatrix@C =.8cm@R =.4cm{
						(	S_{X} \times S_{Y})^{x_1 \ddagger y_1, x_2 \ddagger y_2}
						\ar[rr]^-{ \diamond^{x_1 \ddagger y_1}_{X,Y}}
						\ar[dd]_{ \diamond{x_2 \ddagger y_2}_{X,Y}}&&
						S_{(X \amalg Y) \setminus \{x_1, y_1\}}^{x_2 \ddagger y_2}
						\ar[dd]^{\zeta^{x_2 \ddagger y_2}_{X \amalg Y \setminus \{x_1,y_1\}}}\\&{=}&\\
						S_{(X \amalg Y) \setminus \{x_2, y_2\}}^{x_1 \ddagger y_1}\ar[rr]_-{\zeta^{x_1 \ddagger y_1}{X \amalg Y \setminus \{x_2,y_2\}}}&&
						S_{(X \amalg Y) \setminus \{x_1,x_2, y_1,y_2\}}.}}\]
	\end{enumerate}

\begin{figure}[htb!]
\centerfloat{	\includestandalone[width = .55\textwidth]{CA2M1new}\includestandalone[width = .55\textwidth]{CA2M2new}}
\vskip 5ex
	\includestandalone[width = .55\textwidth]{CA2M3new}\includestandalone[width = .55\textwidth]{CA2M4new}	
	\caption{Modular operad axioms (M1)-(M4).}\label{fig. MO axioms}
\end{figure}
(Axioms (M1)-(M4) are illustrated in \cref{fig. MO axioms}.)

	The category of nonunital, modular operads in $\E$ is denoted by $\nuCSME$.
	
	A \emph{(unital) modular operad $(S, \diamond, \zeta, \epsilon)$ in $\E$} is a graphical species $S \in \GSE$ equipped with a unital multiplication $(\diamond, \epsilon)$ and contraction $\zeta$ such that $(S, \diamond, \zeta) \in \nuCSME$. The category of modular operads in $\E$ is denoted by $\CSME$.
\end{defn}

\begin{rmk}\label{rmk MO not monoidal} As discussed in \cref{ssec: graphical species}, $\GSE$ has a monoidal category structure when $\E$ admits finite coproducts. However, as will be explained in \cref{rmk no tensor on MO}, the categories $\CSM$ and $\nuCSM$ are not closed under the monoidal product in $\GSE$. 

\end{rmk}

\subsection{Circuit algebras}\label{ssec CO}
A \emph{circuit algebra internal to a category $\E$} with finite limits (see \cref{def. co}) is a graphical species $A$ in $\E$ that is equipped with an equivariant contraction and such that $(A(\nn))_{\nn}$ describes a graded associative monoid. The monoid structure and contraction satisfy coherence axioms analogous to the (enriched) circuit algebra axioms $(c1)-(c3), (e1)$ of 
\cite[Proposition~5.4]{RayCA1}.

\begin{defn} \label{def. external}
	A \emph{(commutative) external product} on a graphical species $S$ in $\E$ is given by a $\fiso \times \fiso$-indexed collection of equivariant morphisms $\boxtimes_{X,Y}: S_X \times S_Y \to S_{X \amalg Y},$ in $\E$ 
	such that $\boxtimes_{Y, X} = \boxtimes_{X,Y}\circ \sigma_{Y,X}$ (where $\sigma_{Y,X} \colon S_Y \times S_X \xrightarrow {\cong }S_X \times S_Y$ is the swap isomorphism) and,
	 for all elements $x \in X$, the following diagram commutes in $\E$:
	\[ \xymatrix{S_X \times S_Y \ar[rr]^-{\boxtimes_{X,Y}}\ar[d] &&S_{X \amalg Y} \ar[d]^-{S(ch_x^{X \amalg Y}) } \\
		S_X \ar[rr]_-{S(ch^X_x)} && S_\S}\]
	
	An \emph{external unit} for the external  product $\boxtimes$ is a distinguished morphism $0_S \colon * \to S_\nul$ such that, for all finite sets $X$, the composite 
	\[ \xymatrix{ S_X \ar[r]^-{\cong} &S_X \times *  \ar[rr]^{id \times 0_S }&& S_X \times S_\nul \ar[rr]^- {\boxtimes} && S_{(X \amalg \nul )} \ar@{=}[r] &S_X}\] is the identity on $S_X$ for all finite sets $X$.
	
\end{defn}

\begin{rmk} \label{rmk. monoidal unit} Henceforth, the monoidal unit $0_S $ for $\boxtimes$ will be suppressed in the notation since most of the constructions in Sections \ref{sec: definitions}-\ref{sec. nerve} of this paper carry through without it. 
	To agree with the circuit algebra literature, external products in this paper are always assumed to have a unit. See also \cref{rmk. no empty graph}, and \cite[Paragraph~4.7]{Koc18}.
	
\end{rmk}

Recall (\cref{rmk monoidal}) that, if $\E$ has finite coproducts, then $ \GSE$ is a symmetric monoidal category. It follows immediately from \cref{def. external} that: 
\begin{lem}
	\label{lem external monoid}
	A graphical species $S$ in $\E$ admits an external product $(\boxtimes, 0_S)$ if and only if $(S, \boxtimes, 0_S)$ is a monoid for the monoidal product $ \otimes$ on $\GSE$.
\end{lem}

An external product $\boxtimes$ and contraction $\zeta$ on a graphical species $S$ induce an equivariant multiplication $\diamond$ on $S$: for all finite sets $X$ and $Y$, all $x \in X$ and $y \in Y$, the multiplication $\diamond^{x \ddagger y}_{X,Y}$ is defined by the composition 
\begin{equation}\label{eq. forget mult}
	\diamond_{X,Y}^{x \ddagger y} \colon \xymatrix{(S_X\times S_Y)^{x \ddagger y} \ar[rr]^-{\boxtimes_{X,Y}} && (S_{X \amalg Y} )^{x \ddagger y} \ar[rr]^-{\zeta^{x\ddagger y}}&& S_{(X \amalg Y)\setminus \{x,y\}}.}
\end{equation}

\begin{defn}\label{def. co}
	A \emph{nonunital circuit algebra in $\E$} is a graphical species $S$ in $\E$, equipped with an external product $\boxtimes$ and a contraction $\zeta$, such that the following three axioms (illustrated in \cref{fig. CO axioms}) are satisfied: 
	
	\begin{enumerate}[(C1)]
		\item \emph{External product is associative} (see \cref{fig. CO axioms} (C1)): \\
	For all finite sets $X, Y, Z$,
	the following square commutes:
	\[\xymatrix@C =.8cm@R =.4cm{
		S_{X} \times S_Y \times S_Z \ar[rr]^-{\boxtimes_{X,Y} \ \times\ id_{S_Z}}\ar[dd]_{id_{S_X}\ \times\ \boxtimes_{Y,Z}} && 
		S_{X \amalg Y} \times S_Z \ar[dd]^{\boxtimes_{(X \amalg Y), Z}}\\
		&{=}&\\
		S_{X} \times S_{Y \amalg Z} \ar[rr]_-{ \boxtimes_{X, (Y \amalg Z)}}&&		S_{X\amalg Y \amalg Z }.}\]

\item \emph{Order of contraction does not matter.} This corresponds with \cref{def. mod op} (M2) (see also \cref{fig. MO axioms} (M2)). 

\item \emph{Contraction and external product commute} (see \cref{fig. CO axioms} (C1)):  \\
For all finite sets $X$ and $Y$, and distinct elements $x_1, x_2 \in X$, the following square commutes:
\[\xymatrix@C =.8cm@R =.4cm{
	S_{X}^{x_1 \ddagger x_2} \times S_Y \ar[rr]^-{\boxtimes_{X,Y} }\ar[dd]_{\zeta^{x_1 \ddagger x_2}_X\ \times\  id_{S_Y}} && 
	S_{X \amalg Y} ^{ x_1\ddagger  x_2 } \ar[dd]^{\zeta^{x_1 \ddagger x_2}_{X\amalg Y}}\\
	&{=}&\\
	S_{X \setminus \{x_1, x_2\} }\times S_Y\ar[rr]_-{\boxtimes_{X \setminus \{x_1, x_2\}, Y}}&&		S_{X \setminus \{x_1, x_2\} \amalg Y} .}\]
	\end{enumerate}

\begin{figure}[htb!]
	
	{\centerfloat 
	\includestandalone[width = .55\textwidth]{CA2C1new}\includestandalone[width = .55\textwidth]{CA2C3new}}
	\caption{Circuit algebra axioms (C1) and (C3). For axiom (C2), see \cref{fig. MO axioms}(M2). }\label{fig. CO axioms}
\end{figure}

Let $\nuCOE$ denote the category of nonunital circuit algebras in $\E$ and morphisms of graphical species that preserve the external product and contraction. 

A \emph{unital circuit algebra (in $\E$)} is a nonunital circuit algebra $(S, \boxtimes, \zeta)$ together with a unit-like morphism $\epsilon \colon S_\S \to S_\two $ (\cref{defn: formal connected unit}), such that, for all finite sets $X$ and $x \in X$, the composite 
\begin{equation}
	\label{eq. conn unit} \xymatrix@C =.7cm{ S_X \ar[rr]^-{ (id \times ch_x)\circ \Delta}&& (S_X \times S_\S)^{x \ddagger 2} \ar[rr] ^-{ id \times \epsilon} &&(S_X \times S_\two)^{x \ddagger 2} \ar[r]^-{\boxtimes}&{ (S_{X \amalg\two})^{x \ddagger 2} }\ar[rr]^-{ \zeta^{x \ddagger 2}_{X \amalg \two} }&& S_X }
\end{equation}is the identity on $S_X$. (Here $\Delta \colon S_X \to S_X \times S_X$ is the diagonal and the last map makes use of the canonical isomorphism $( X \amalg \{1\})\setminus\{x\} \cong X$.)

Morphisms in the category $\COE$ of circuit algebras in $\E$ are morphisms in $\nuCOE$ that preserve units.
\end{defn}


For $\E = \Set$, we will write $\CO \defeq \COE$ and $\CSM\defeq \CSME$. Let $\bigCA$ be the category, defined in \cite[Section~5.2]{RayCA1}, of circuit algebras (of all colours) enriched in $\Set$. The following proposition follows immediately from  \cite[Proposition~5.4]{RayCA1} and by unpacking \cref{def. co}: 
\begin{prop}	\label{prop. CA CO}
The inclusion $\Sigma \hookrightarrow \fiso$ induces a canonical equivalence of categories $\bigCA\simeq\CO$.
\end{prop}

\begin{rmk}\label{rmk always internal}
	Henceforth, unless explicitly stated otherwise, circuit algebras and modular operads will be assumed to be objects of the categories $\COE$ and $\CSME$ (and $\nuCOE$, $\nuCSME$) internal to a category $\E$ with sufficient limits, and we will not consider enriched circuit algebras.
\end{rmk}

By (\ref{eq. forget mult}), a (nonunital) circuit algebra $(S,\boxtimes, \zeta )$ admits a (nonunital) multiplication $\diamond_S$ induced by $\boxtimes$ and $\zeta$. It follows immediately from (\ref{eq. forget mult}) and axioms (C1)-(C3), that $(S, \diamond_S, \zeta)$ satisfy (M1)-(M4). Hence, circuit algebras in $\E$ canonically admit a modular operad structure.
 This will be used in Sections \ref{s. graphs}-\ref{sec. iterated}, to construct the composite monad for circuit algebras by modifying the modular operad monad of 
\cite{Ray20}.

In fact, we have the following proposition, analogous to \cite[Proposition~5.11]{RayCA1}, that will follow from \cref{thm. iterated law}, \cref{prop nonunital CA} and \cref{lem. LT is Tk}, together with the classical theory of distributive laws \cite{Bec69}:

\begin{prop}	\label{prop: CO to MO} There are canonical monadic adjunctions 
	\begin{equation}\label{eq: adjunction diag}\small
		{\xymatrix@C = .5cm@R = .3cm{
				\nuCOE \ar@<-5pt>[rr]\ar@<-5pt>[dd]&\scriptsize{\top}&\COE\ar@<-5pt>[ll]\ar@<-5pt>[dd]\\
				\vdash&&\vdash\\ 
				\nuCSME \ar@<-5pt>[uu] \ar@<-5pt>[rr]&\scriptsize{\top}& {\CSME} \ar@<-5pt>[ll]\ar@<-5pt>[uu]}}\end{equation} between the categories of (nonunital) modular operads and (nonunital) circuit algebras in $\E$.
	
\end{prop}
\begin{proof}
	By \cref{thm. iterated law}, $ \COE $ is the EM category of algebras $ \GSE^{\LL\DD\TT}$ for the composite monad $\LL \DD\TT$ on $\GSE$, constructed using iterated distributive laws. Now, $\nuCSME  = \GSE^{\TT}$ by \cref{prop. Talg}, $\CSME  = \GSE^{\DD\TT}$ by \cref{thm CSM monad DT} (see also \cite{Ray20}), and $\nuCOE  = \GSE^{\LL\TT}$ by \cref{prop nonunital CA} and \cref{lem. LT is Tk}. Hence, the vertical adjunctions exist by the classical theory of distributive laws \cite{Bec69}. 
	
	The horizontal adjunctions -- where the left adjoints correspond to freely adjoining modular operadic units -- exist by \cref{prop. extension exists}. 
\end{proof}

The right adjoints of the pairs $\CSME \rightleftarrows \COE$ (and $\nuCSME \rightleftarrows \nuCOE$) in \cref{prop: CO to MO} correspond to the forgetful functor $(S, \boxtimes, \zeta)\mapsto (S, \diamond, \zeta)$ where the multiplication $\diamond$ is defined by suitable compositions $\zeta \circ \boxtimes$ as in (\ref{eq. forget mult}). 
The left adjoint is induced by the free graded monoid monad: 

Let $\core{X \ov \fin}$ be the maximal subgroupoid of the slice category $X \ov \fin$ whose objects are morphisms of finite sets $f \colon X \to Y$ (that may be thought of as $Y$-indexed partitions of $X$), and whose morphisms $(Y,f) \to (Y', f')$ are isomorphisms $g \cong Y \to Y'$ such that $ gf = f'$. If $S$ is a graphical species, then there is a graphical species $LS$ defined by  
\[ LS_X = \mathrm{colim}_{(Y, f) \in \core{X \ov \fin}} \coprod_{y \in Y} S_{f^{-1}(y)}.\] 

The free monoid structure induces the external product on $LS$, and, if $(S, \diamond, \zeta)$ is a (nonunital) modular operad, then the modular operad axioms imply that the pair of operations $\zeta$ and $\diamond$ define a contraction on $LS$ (see \cref{sec. iterated}).

In particular, since units for the modular operadic multiplication are unique:
\begin{cor}\label{cor. inclusion CO nuCO}
	The canonical forgetful functors $\COE \to \nuCOE$ and $\CSME \to \nuCSME$ are inclusions of categories.
\end{cor}

\begin{ex}\label{ex. wheeled prop CO}

By \cite{DHR20, RayCA1}, oriented circuit algebras enriched in a symmetric monoidal category $(\V, \otimes, I)$ are equivalent to (coloured) wheeled props enriched in $\V$. 
The image of a (nonunital) wheeled prop under the forgetful functor $\CO \to \CSM$ is its underlying (nonunital) \textit{wheeled properad}  \cite{HRY15, JY15}. (See also \cref{ex. terminal coloured species},  \cref{ex Wheeled properads} and \cref{cor wheeled props nerve} and Proposition~3.26 and Example~5.10 of \cite {RayCA1}.)

\end{ex}
The graphical species $\DicommE$ (Examples~\ref{ex directed graphical species}~\&~\ref{ex. oriented slice}) trivially admits a circuit algebra structure. Hence, \cref{ex. wheeled prop CO} motivates the following definition:
\begin{defn}\label{def wp int}
	The category $\WPE$ of \emph{wheeled props in $\E$} is defined as the slice category $\COE \ov \DicommE$. When $ \E = \Set$, we write $\WPin \defeq \WPE[\Set]$. 
	
 The category of \emph{nonunital wheeled props in $\E$} is defined similarly by $\nuWPE \defeq \nuCOE \ov \DicommE$. 
\end{defn}

\begin{rmk}
	\label{rmk monoidal question}
	By \cref{prop: CO to MO}, a circuit algebra is a modular operad equipped with an external product that is compatible with the modular operadic multiplication in the sense of (\ref{eq. forget mult}). By \cref{lem external monoid}, a graphical species with external product $(S, \boxtimes, 0_S)$ is just a monoid object in $\GSE$ so it is natural to ask whether there is a monoidal structure on $\CSME$ such that circuit algebras are precisely the monoid objects in $\CSME$. This is not the case since, as will be explained in \cref{rmk no tensor on MO}, $\CSME$ is not closed under the monoidal product on $\GSE$. 
\end{rmk}

\newcommand{\defretract}[5]{\xymatrix{*[r]{#1} \ar@<1ex>[rrr]^-{#3} \ar@(ul,dl)[]_{#5} &&& #2 \ar@<1ex>[lll]^-{#4}}}

\section{Background on graphs}\label{s. graphs}
The combinatorics of modular operads have been described in \cite[Section~8]{Ray20} in terms of a category $\Klgr$ of connected graphs. The remainder of the present paper is concerned with an analogous construction and graphical category $\Klgrt$ for circuit algebras.

Sections \ref{subs:feynman graphs} and \ref{ssec Sgraph} provide a review of the category (Feynman) graphs and \'etale morphisms, first introduced in \cite{JK11}, that will underpin the rest of the paper. 
The interested reader is referred to \cite[Sections~3~\&~4]{Ray20} for details and explicit proofs related to the graphical formalism.  

\cref{ssec Brauer} gives a brief description of the operads of wiring diagrams and Brauer diagram category used to describe enriched circuit algebras (see e.g.,~\cite{BND17, DHR20, DHR21} and \cite{RayCA1}), and explains their relationship with Feynman graphs. This is purely for context and may be omitted.

\subsection{Graphs and \'etale morphisms}\label{subs:feynman graphs}

The following definition is due to Joyal and Kock \cite{JK11} (see also \cite{Koc16, Ray20, HRY19a, HRY19b}):

\begin{defn}\label{defn: graph}
	
	A \emph{(Feynman) graph} $\G$ is a diagram of finite sets
	\begin{equation} \label{eq. graph} \G = \Fgraph,\end{equation}
	such that $s\colon H \to E$ is injective, and $\tau\colon E\to E $ is an involution without fixed points. 
	
	A \emph{subgraph} $\H \subset \G$ is a triple of subsets $E' \subset E, H' \subset H, V' \subset V$ that inherit a graph structure from $\G$. 
\end{defn}

Elements of $V$ are \emph{vertices} of $\G$ and elements of $E$ are called \emph{edges} of $\G$. The set of $\tau $-orbits in $E$ is denoted by $\tilde E$.
Elements $h$ of the set $H$ are \textit{half-edges} of $\G$.

\begin{defn}\label{defn: inner edges} \label{defn. stick comp}
	The set $\EI \subset E$ of \emph{inner edges} of $\G$ is the maximal subset of $im(s) \subset E$ that is closed under $\tau$. The set of \emph{inner $\tau$-orbits} $\tilde e \in \widetilde E$ with $e \in \EI$ is denoted by $\widetilde {\EI}$. 
	Elements of the set $E_0 \defeq E\setminus im(s)$ are \emph{ports} of $\G$. 
	
	A \emph{stick component} of a graph $\G$ is a pair $\{e, \tau e\}$ of edges of $\G$ such that $e$ and $\tau e$ are both ports. 	
	
\end{defn}

\begin{defn}\label{defn. conn graph}
	A graph is \emph{connected} if it cannot be written as a disjoint union of non-empty graphs. 
A \textit{connected component} $\H$ of $\G$ is a maximal connected subgraph of $\G$.

\end{defn}

In particular, 
any graph $\G$ is the disjoint union of its connected components, and the inclusion $\H \hookrightarrow \G$ of a connected component $\H$ of $\G$ describes a pointwise injective \'etale morphism. 

\begin{ex}
A subgraph $\H \subset \G$ of the form $ \xymatrix{
	\{e, \tau e\} & \emptyset \ar[l] \ar[r]&\emptyset}$ describes a connected component of $\G$ if and only if $\{e, \tau e\}$ is a  stick component of $\G$.
\end{ex}
\begin{rmk}\label{geom real}A graph $\G$ can be thought of as a rule for gluing closed intervals along their boundaries and may be realised geometrically as the one-dimensional space $|\G|$, obtained by taking the discrete space $\{ *_v\}_{v \in V}$, and, for each $e \in E$, a copy $[0,\frac{1}{2}]_e$ of the interval $[0,\frac{1}{2}]$ subject to the identifications $0_{s(h)}\sim *_{t(h)}$ for $h \in H$ and $(\frac{1}{2})_{e} \sim (\frac{1}{2})_{\tau e}$ for all $e \in E$. 
	
	Disjoint union commutes with geometric realisation: for all graphs $\G$ and $\H$, $|\G \amalg \H| = |\G| \amalg |\H|$. So, a graph $\G$ is connected if and only if its realisation $|\G|$ is a connected space.
	
\end{rmk}

\begin{ex}\label{ex stick} (See also \cref{ex generating diagrams}.)
	Up to isomorphism, there are four distinct connected graphs with a single edge orbit:
	\begin{enumerate}[(a)]
		\item 
		the \emph{$\one$-corolla } $\C_{\one}$ has one vertex, no inner edges, and one port (see \cref{fig. one orbit}(a))
		\[ \xymatrix{
			\C_{\one}\defeq&\mathbf 2\ar@(lu,ld)[] &  \one \ar[l] \ar[r]&\one }\] 
		\item 
		the \textit{stick graph} $(\shortmid)$ has no vertices, no inner edges and two ports (see \cref{fig. one orbit}(b))
		\[ \xymatrix{
			( \shortmid)\defeq&\mathbf 2\ar@(lu,ld)[] & \mathbf 0 \ar[l] \ar[r]&\mathbf  0}\] 
		\item 
		the \textit{wheel graph} $\W$ has one vertex, one inner edge orbit and no ports (see \cref{fig. one orbit}(c));
		\[ \xymatrix{
			\W \defeq&\mathbf 2\ar@(lu,ld)[] & \two  \ar[l]_{id} \ar[r]&\mathbf 1}\] 
		\item finally, there is a graph with two vertices, one inner edge orbit and no ports (\cref{fig. one orbit}(d)):
		\[ \xymatrix{
			&\mathbf 2\ar@(lu,ld)[] & \two  \ar[l]_{id} \ar[r]^{id}&\two}.\] 
	\end{enumerate}
	
	The stick and wheel graphs $(\shortmid)$ and $\W$ are particularly important in what follows. 
		\end{ex}
	
	\begin{figure}[htb!]
		\begin{tikzpicture}[scale =.85]
				
				\node at (-1,2.5) {(a)};
				\draw (5,1)--(5,2);
				\node at (4.7,1.3) { \small{$2$}};
				\node at (4.7, 1.8) {\small{$ 1$}};

				\node at (3,2.5) {(b)};
				\draw (0,1)--(0,2);
				\node at (-.3,1.2) { \small{$1^\dagger$}};
				\node at (-.3, 1.9) {\small{$ 1$}};
				\filldraw 
				(0,1)circle (2pt);

				\node at (7,2.5) {(c)};
				\node at (9,1.5)
				{\begin{tikzpicture}
						\draw (0,0) circle (15pt);
						\filldraw (.53,0) circle (2pt);
				\end{tikzpicture}};	
				\node at (12,2.5) {(d)};
				\draw (14,1)--(14,2);
				\filldraw 
				(14,1)circle (2pt);
				\filldraw 
				(14,2)circle (2pt);

			\end{tikzpicture}
			\caption{Realisations of connected graphs with one edge orbit: (a) the 1-corolla $\C_\one$; (b) the stick graph $(\shortmid)$; (c) the wheel graph $\W$ consists of a single inner edge orbit with one end vertex; (d) a single inner edge orbit with distinct end vertices.} \label{fig. one orbit}
		\end{figure}

	For any set $ X$, let $X^\dagger \cong X$ denote its formal involution. 
	\begin{ex}\label{corolla}(See also \cref{fig. one orbit}(d).) 
		The \emph{$X$-corolla} $\CX$ associated to a finite set $X$ has the form
		\[ \xymatrix{
			\CX: && *[r] {X \amalg X^\dagger }
			\ar@(lu,ld)[]_{\dagger}&& 
			X^\dagger \ar@{_{(}->}[ll]_-{\text{inc}}\ar[r]& \{*\}.
		}\]

	\end{ex}

	\label{Hv Ev} 
	Let $\G$ be a graph with vertex and edge sets $V$ and $E$ respectively. For each vertex $v$, define $\vH \defeq t^{-1}(v) \subset H$ to be the fibre of $t$ at $v$, and let $\vE\defeq s(\vH) \subset E$.
	
	\begin{defn}\label{defn: valency} 
		
		Edges in the set $\vE$ are said to be \emph{incident on $v$}.

		The \emph{valency $|v|$ of a vertex $v \in V$} is given by $ |v| \defeq |\vH|$ and $\nV\subset V$ is the set of \emph{$n$-valent} vertices of $\G$. A vertex $v$ is \emph{bivalent} if $|v| = 2$ and a \emph{bivalent graph} is a graph $\G$ with $V = \nV[2]$. A vertex $v \in V$ with $|v| = 0$ is called an \emph{isolated vertex} of $\G$.

	\end{defn}
	
	Bivalent and isolated vertices are particularly important in \cref{ss. pointed}.
	
	\label{Hn En} Vertex valency induces an $\N$-grading on the edge set $E$ 
	of $\G$: For $n \geq 1$, define $\nH \defeq t^{-1}(\nV)$ and $\nE \defeq s(\nH)$. Then $s(H) = \coprod_{n \geq 1} \nE= E\setminus E_0 $, so  
	$E = \coprod_{n \in \N } \nE.$
	
	\begin{ex}
		The stick graph $(\shortmid)$ (\cref{ex stick}) has $E_0(\shortmid) = E(\shortmid) = \two$, and $\EI(\shortmid)   = \emptyset$. Conversely, the wheel graph $\W$ has $ E_0(\W) = \emptyset$ and $\EI(W) = E_2 (\W) = E(\W) \cong \two. $
		
		The $X$-corolla $\CX$ (see \cref{corolla}) with vertex $*$ has set of ports $E_0 (\CX) = X$ and $\vE[*]  = X^\dagger$. So, if $X \cong \nn$, then $|*| = n$, $ V = 
		\nV$, and $E = \nE \amalg \nE[0]$ where $\nE[i] \cong X$ for $i = 0, n$.\end{ex}

	\begin{defn}
		\label{defn. etale morphism} 
		An \emph{\'etale morphism} $f \colon \G \to \G'$ of graphs is a commuting diagram of finite sets 
		\begin{equation}\label{morphism}
			\xymatrix{
				\G\ar[d]&&   E \ar@{<->}[rr]^-\tau\ar[d]&&E \ar[d]&& H \ar[ll]_s\ar[d] \ar[rr]^t&& V\ar[d]\\
				\G'&&  E' \ar@{<->}[rr]_-{\tau'}&&E'&& H' \ar[ll]^{s'} \ar[rr]_{t'}&& V' }\end{equation}
		such that the right-hand square is a pullback. The category of graphs and \'etale morphisms is denoted by $\Gret$.
	\end{defn}
	
	It is straightforward to verify that the right hand square in (\ref{morphism}) is a pullback if and only if, for all $v \in V$, the map $\vE \to \sfrac{E'}{f(v)}$ induced by the restriction of the map $V\to V'$ is bijective. Hence, as the name suggests, \'etale morphisms describe local isomorphisms. 
	
	\begin{ex}\label{ex: choose}
		Let $\G$ be a graph with edge set $E$. Then $\Gret(\shortmid, \G)  =  \{ch_e\}_{e \in E} \cong E$ where for each edge $e \in \E$, $ch_e \colon (\shortmid) \to \G$ is the morphism $1 \mapsto e$ that \emph{chooses} $e$. 
		
\end{ex}

Pointwise disjoint union defines a symmetric strict monoidal structure on $\Gret$ with unit given by the empty graph $\oslash\defeq  \xymatrix{
	{ \ \emptyset \ }
	&
	\ \emptyset \ \ar[l]_-{}\ar[r]&\ \emptyset
}$.

\begin{ex}\label{ex disjoint corollas}
	Let $X$ and $Y$ be finite sets. The disjoint union $\CX \amalg \CX[Y]$ of $X$ and $Y$-corollas has ports $E_0(\CX \amalg \CX[Y]) = X \amalg Y = E_0(\CX)\amalg E_0(\CX[Y])$ and no internal edges. The canonical inclusions $\iota_X \colon \CX  \hookrightarrow\CX \amalg \CX[Y] \hookleftarrow \CX[Y] \colon \iota_Y$ are \'etale.
	
\end{ex}

\begin{ex}\label{Shrubs}
	Up to isomorphism, $(\shortmid)$ is the only connected graph with no vertices.
	As in \cite{Ray20}, and following \cite{Koc16}, a \textit{shrub} $\mathcal S $ is a disjoint union of stick graphs.  
	Any commuting diagram of the form (\ref{morphism}), where $\G = \mathcal S$ is a shrub, is trivially \'etale and hence defines a morphism in $\Gret$.

\end{ex}

\begin{ex}
	\label{deg loop}   (C.f.~\cite[Remark~3.12]{RayCA1}.) 
	Let $\diag$ be the small category $\xymatrix{
		\bullet \ar@(lu,ld)[]_{}\ar[r]&\bullet &\bullet\ar[l]}$. So $\Gret$ is a subcategory of $\pr{\diag}$. 
		The endomorphisms $id, \tau\colon (\shortmid) \rightrightarrows (\shortmid)$ of $(\shortmid)$ in $\Gret$ are coequalised in $\pr{\diag}$ by $* \leftarrow \emptyset \rightarrow \emptyset$, which is not a graph. So, $id_{(\shortmid)}$ and $ \tau$ do not admit a coequaliser in $\Gret$. This provided a motivating example for \cite{Ray20}.
	
	The two endomorphisms $id_\W, \tau_\W \colon \W \to \W$, viewed as morphisms of presheaves on $\diag$, have coequaliser $* \leftarrow * \rightarrow *$ which is terminal in $\pr{\diag}$, and is not a graph. In particular, $\Gret$ does not have a terminal object and is therefore not closed under (finite) limits.
	
\end{ex}

As \cref{deg loop} shows, $\Gret$ does not, in general, admit finite colimits. However, if $\G $ is a graph and $ e_1, e_2 \in E_0$ are ports of $\G$ with $e_1 \neq \tau e_2$, then the colimit $\G^{e_1 \ddagger e_2}$ of the parallel morphisms \begin{equation}
	\label{eq. dagger graph}
	ch_{e_1}, ch_{e_2} \circ \tau \colon (\shortmid) \rightrightarrows \G
\end{equation} is the graph described by identifying $e_1 \sim \tau e_2$ and $e_2 \sim \tau e_1$ in the diagram (\ref{eq. graph}) defining $\G$.

\begin{ex}
	The graph $\C_\two^{e\ddagger \tau e}$ obtained by identifying the ports of the $\two$-corolla is isomorphic to the wheel graph $\W$. 
\end{ex}

\begin{ex}\label{ex: N graph} More generally, for $X$ a finite set with distinct elements $x$ and $y$, the graph $\C_X^{x \ddagger y}$ 
	has ports in $E_0 = X\setminus \{x,y\}$, one inner $\tau$-orbit $\{x,y\}$ (bold-face in \cref{fig: MN}), and one vertex $v$:
	\[\xymatrix{
		*[r] { 
			(	X \setminus \{x,y\}) \amalg X^\dagger}
		\ar@(lu,ld)[]_{\tau} &&&  
		X^\dagger \ar@{_{(}->}[lll]_-{s} \ar[r]^-{t}& \{v\} , }\]
	where $\tau x^\dagger = y^\dagger $ and $\tau z = z^\dagger$ for $z \in X \setminus \{x,y\}$. 
	Graphs of the form 	$\C_X^{x \ddagger y}$ encode formal contractions in graphical species.

\end{ex}

Since the monoidal structure on $\Gret$ is induced by the cocartesian monoidal structure on diagrams of finite sets, colimits in $\Gret$, when they exist, commute with $\amalg$. In particular, 
for any graphs $\G$ and $ \H$ and distinct ports $e_1, e_2$ of $\G$, 
\begin{equation}
	\label{eq. C3 motivation} \G^{e_1 \ddagger e_2} \amalg \H = (\G \amalg \H)^{e_1 \ddagger e_2}.
\end{equation}

\begin{ex}\label{ex: M graph}
Let $X, Y$ be finite sets with elements $x \in X$, and $y \in Y$. By identifying the edges $x \sim \tau y, y \sim \tau x$ in $ \CX \amalg \CX[Y]$ for, we obtain a graph $ \left(\CX \amalg \CX[Y]\right)^{x \ddagger y}$, that has two vertices and one inner edge orbit $\{x,y\}$, highlighted in bold-face in \cref{fig: MN}.
	\[\xymatrix{
		*[r] { \left(\small
			{
				(X \amalg Y)\setminus \{x,y\} \amalg (X \amalg Y)^\dagger 
			}\right)
		}
		\ar@(lu,ld)[]_{\tau} &&&&
		{ \left(\small
			{ (X \amalg Y)^\dagger 
			}\right)
		}\ar@{_{(}->}[llll]_-{s} \ar[rr]^-{t}&& \small{\{v_X, v_Y\}}}\] with $s$ the obvious inclusion and the involution $\tau$ described by $ z \leftrightarrow z^\dagger$ for $z \in (X\amalg Y)\setminus\{x,y\} $ and $x^\dagger \leftrightarrow y^\dagger$. The map $t$ is described by $t^{-1}(v_X ) = \left(X\setminus\{x\}\right)^\dagger \amalg \{y\}$ and $t^{-1}(v_Y) = \left(Y\setminus\{y\}\right)^\dagger\amalg \{x\}$. 
	
	In the construction of modular operads, graphs of the form $(\CX \amalg \CX[Y])^{x\ddagger y}$ are used to encode formal multiplications 
	in graphical species. \end{ex}

\begin{figure}[htb!] 
	\begin{tikzpicture}
		
		\node at(0,0){
			\begin{tikzpicture}{
					\node at (6,0) 
					{\begin{tikzpicture}[scale = 0.5]
							\filldraw(0,0) circle(3pt);
							\foreach \angle in {0,120,240} 
							{
								\draw(\angle:0cm) -- (\angle:1.5cm);

							}
							\draw [ ultra thick] (0,0) -- (1.5,0);
							\node at (.5, -0.3) {\tiny $x$};
							\node at (2.2, -0.3) {\tiny$ y$};

					\end{tikzpicture}};
					
					\node at (7,0){\begin{tikzpicture}[scale = 0.5]
							\filldraw(0,0) circle(3pt);
							\foreach \angle in {-0,90,180,270} 
							{
								\draw(\angle:0cm) -- (\angle:1.5cm);

							}
							\draw [ ultra thick] (0,0) -- (-1.5,0);

					\end{tikzpicture}};

					\node at (6.5,-2){ |$(\CX \amalg \CX[Y])^{x\ddagger y}$|};
		}\end{tikzpicture}};
		
		\node at (5, 0){	\begin{tikzpicture}
				
				\node at (5,0) 
				{\begin{tikzpicture}[scale = 0.5]
						\filldraw(0,0) circle(3pt);
						\draw (0,-1.5)--(0,1.5);
						\draw[ultra thick] 
						(0,0)..controls (3,3) and (-3,3)..(0,0);
						\node at (-1,.6) {\tiny $x$};
						\node at (1, .6) {\tiny $y$};
				\end{tikzpicture}};

				\node at (5,-2){ |$\C_{X}^{x \ddagger y}$|};
				
			\end{tikzpicture}
		};
	\end{tikzpicture}
	\caption{ Realisations of $(\CX \amalg \CX[Y])^{x\ddagger y}$ and $\C_{X}^{x\ddagger y}$ for $X \cong \mathbf{4}, \ Y \cong \mm[3]$.} 
	\label{fig: MN}
\end{figure}
\label{Hv Ev}    

\begin{ex}\label{def Lk} 
	For $k \geq 0$, the \emph{line graph} $\Lk$ is the connected bivalent graph (illustrated in \cref{fig. line and wheel}) with ports in $ E_0 = \{1_{\Lk}, 2_{\Lk}\}$, and 
	\begin{itemize}
		\item ordered set of edges $E (\Lk)= (l_j)_{j = 0}^{2k+1}$ where $l_0 = 1_{\Lk} \in E_0$ and $l_{2k+1} = 2_{\Lk} \in E_0$, and the involution is given by $\tau (l_{2i}) = l_{2i +1}$, for $0 \leq i \leq k$, 
		\item ordered set of $k$ vertices $ V(\Lk) = (v_i)_{i = 1}^k $, such that $ \vE[v_i] = \{ l_{2i -1}, l_{2_i}\}$ for $1 \leq i \leq k$.
	\end{itemize}
	So, $\Lk$ is described by a diagram of the form 
	$\Fgraphvar{\ \two \amalg 2(\mathbf k)\ }{2(\mathbf k)}{\mathbf k.}{}{}{} $
	
	The line graphs $(\Lk)_{k \in \N}$ may be defined inductively by gluing copies of $\CX[\two]$:
	\[ \begin{array}{llll}
		\Lk[\nul]& = & (\shortmid) &\\
		\Lk[k+1] & \cong  & (\Lk \amalg \C_\two)^{ 2_{\Lk}\ddagger 1_{\C_\two}}, & k \geq 0.
	\end{array}\]
\end{ex}

\begin{ex}%
	
	\label{wheels}
	For $m \geq 1$, the {wheel graph} $\Wm $ (illustrated in \cref{fig. line and wheel}) is the connected bivalent graph with no ports, obtained as the coequaliser in $\Grbig$ of the morphisms $ch_{ 1_{\Lk[m]}}, ch_{2_{\Lk[m]}} \circ \tau \colon (\shortmid) \rightrightarrows \Lk[m]$. So, $\Wm = {\Lk[m]}^{(1 \ddagger 2)}$ has
	\begin{itemize}
		\item $2m$ cyclically ordered edges $E(\Wl) = (a_j)_{j = 1}^{2m}$, such that the involution satisfies $\tau (a_8) = a_1$ and $\tau (a_{2i}) = a_{2i +1}$ for $1 \leq i \leq m-1$,
		\item $ m$ {cyclically} ordered vertices $V(\Wl) = (v_i)_{i = 1}^m$, that $ \vE[v_i] = \{ a_{2i -1}, a_{2i }\}$ for $1 \leq i \leq m$, 
		
	\end{itemize} 
	and is described by a diagram of the form $\Fgraphvar{\ 2(\mathbf m)\ }{2(\mathbf m)}{\mathbf m.}{}{}{}$

	\begin{figure}[htb!]
		\begin{tikzpicture}
			\node at (0,-1.5) {$\mathcal L^4$};
			\node at (0,0) { \begin{tikzpicture}[scale = 0.4]
					\draw (0,0) -- (10,0);
					\filldraw (2,0) circle (3pt);
					\filldraw (4,0) circle (3pt);
					\filldraw		(6,0) circle (3pt);
					\filldraw		(8,0) circle (3pt);
					\node at (0,-.7) {\tiny $l_0$};
					\node at (1.7,-.7) {\tiny $l_1$};
					\node at (10,-.7) {\tiny $l_9$};
				\end{tikzpicture}
			};
			\node at (6,-1.5) {$\mathcal W^4$};
			\node at(6,0){\begin{tikzpicture}[scale = 0.4]
					\draw (0,0) circle (2cm);
					\filldraw (2,0) circle (3pt);
					\filldraw (0,2) circle (3pt);
					\filldraw		(-2,0) circle (3pt);
					\filldraw		(0,-2) circle (3pt);
					
					\node at (-.6,2.2) {\tiny{$a_1$}};	
					\node at (.6,2.2) {\tiny{$a_2$}};
					\node at (-2.4,0.3) {\tiny{$a_8$}};
			\end{tikzpicture}};
			
		\end{tikzpicture}
		
		\caption{Line and wheel graphs.}
		\label{fig. line and wheel}
	\end{figure}

\end{ex}

By \cite[Proposition~4.23]{Ray20}, a connected bivalent graph is isomorphic to $\Lk$ or $\Wl$ for some $k, m$. 

In general, given ports $e_1, e_2$ of a graph $\G$ such that $e_1 \neq \tau e_2$, then $\G^{e_1 \ddagger e_2}$ is formed by ``gluing'' the ports $e_1$ and $e_2$ (see \cref{fig: MN}). As the following examples shows, the result of forming $\G^{e_1 \ddagger e_2 }$ is somewhat surprising when $e_1 = \tau e_2$, so $e_1, e_2$ are the edges of the same stick component:

\begin{ex}\label{ex. stick colimit} 
	Observe that 
	$ch_1^{(\shortmid)} \colon (\shortmid) \to (\shortmid)$, $1 \mapsto 1$ is the identity on $(\shortmid)$ and $ch_2^{(\shortmid)} \colon (\shortmid) \to (\shortmid)$, $1 \mapsto 2$ is precisely $\tau$. Hence 
	$ ch_2^{(\shortmid)} \circ \tau = \tau^2 = id_{\shortmid}, $
	and therefore \[ (\shortmid)^{1 \ddagger 2} = (\shortmid).\]
	
	It follows that, if a pair of ports $e_1, e_2 = \tau e_1 \in E_0(\G)$ form a stick component of a graph $\G$ (\cref{defn. stick comp}), then $\G^{e_1 \ddagger e_2} \cong \G$.
	
\end{ex}
\begin{rmk}
	\label{rmk loop intuition} (See also \cite{Ray20}, in particular \cite[Introduction~\&~Section~6]{Ray20}.) 
It is worth contrasting \cref{deg loop} with \cref{ex. stick colimit} above. Intuitively (graphically), the (formal) coequaliser of  $id, \tau\colon (\shortmid) \rightrightarrows (\shortmid)$ in \cref{deg loop} can be thought of as gluing the two ends of the exceptional edge together to form a nodeless loop, and one may even wish to formally add the diagram $* \leftarrow 0 \rightarrow 0$ to $\Gret$ to encode this operation. However, as \cref{ex. stick colimit} shows, it is in fact the case that $(\shortmid)^{1 \ddagger 2} = (\shortmid)$. (In \cref{wheel deletion}, it will be seen that $(\shortmid)$ is also the outcome of ``substituting $(\shortmid)$ into $\W$''.)  Indeed, in \cite{HRY19a, HRY19b} graphs are defined as in \cref{defn: graph} together with additional boundary information and the nodeless loop is described by the same graph as $\shortmid$, but with empty boundary. 

This work, like \cite{Ray20}, takes a different approach, partly informed by the Weber nerve construction. No nodeless loop graph is introduced. Instead contracted units are defined in terms of a map (in the Kleisli category of the monad $\DD$, \cref{sec. mod op}) from the isolated vertex to the exceptional edge $(\shortmid)$. (See also the Introduction, \cref{sec. iterated} and  \cite[Sections~6~\&~7]{Ray20}.)

\end{rmk}

\begin{defn}
	\label{lem. mono}
	Following \cite{HRY19a}, a graph \emph{embedding} 
		is a morphism $f \in \Gret (\G, \H)$ that factors 
	as 
	\[ \G \to \G^{e_1 \ddagger e'_1, \dots, e_k \dagger e'_k} \hookrightarrow \H \] where the first morphism describes a colimit $\G \to \G^{e_1 \ddagger e'_1, \dots, e_k \dagger e'_k}$ of pairs of parallel maps 
	\[ch_{e_i}, ch_{e'_i} \circ \tau \colon (\shortmid) \rightrightarrows \G,\] (with $e_1, e'_1, \dots, e_k, e'_k$ a possibly empty set of distinct ports of $\G$), and the second morphism is a pointwise injection  in $\Gret$. 
\end{defn}

Embeddings form a subcategory of $\Gret$ that includes all isomorphisms. They may be thought of as analogous to inclusions of open subsets in topology. The category $\Gret$ also admits constructions that play the role of neighbourhoods of points in topology. (See \cref{ex topology} for a description of the Grothendieck topology on $\Gret$. This is discussed in more detail in \cite{JK11,Ray20}.)

\begin{defn}\label{def. neighbourhood v}
	A \emph{neighbourhood} of a vertex $v \in V$ (respectively, an edge $ e \in E$) is an embedding $u \colon \H \to \G$ such that $v$ (respectively $e$) is in the image of $u$. 
\end{defn}

There is a canonical choice of neighbourhood for each vertex $v$ and edge $e$ of a graph $\G$:

\begin{ex}\label{ex. essential}
	Let $\G$ be a graph. 
	For each orbit $\tilde e \in \tilde E$, the inclusion of the stick graph 
	$\shorte \defeq \xymatrix{	\{e, \tau e\}
		& \emptyset \ar[l] \ar[r]&\emptyset} $ 
	canonically induces an embedding $\ese \colon (\shorte) \hookrightarrow \G$. 
	
	Recall that, for each $v \in V$,  $\vE \defeq s(t^{-1}(v))$ is the set of edges incident on $v$. Let $\mathbf{v} = (\vE)^\dagger$ denote its formal involution. 
	
	Then the corolla $\Cv$ is given by 
	\[ \Cv  =  \qquad  \xymatrix{*[r] { \left(\vE \amalg ({\vE})^\dagger \right)}\ar@(lu,ld)[] && \vH \ar[ll]_-s \ar[r]^t&{\{v\}}. 
	}\] 
	
	The inclusion $\vE \hookrightarrow E$ induces an embedding $\esv^\G$ or $\esv \colon \Cv \to \G$. 
	Observe that, whenever there is an edge $e$ such that $e$ and $\tau e $ are both incident on the same vertex $v$ -- so $\{e, \tau e\}$ forms a \textit{loop} at $v$ -- then $\esv$ is not injective on edges.
	
	If $ \vE$ is empty -- so $\Cv$ is an isolated vertex -- then $ \Cv \hookrightarrow \G$ is a connected component of $\G$.
	
	For each $v \in V$ (respectively, $ e \in E$), an embedding $u \colon \H \to \G$ is a neighbourhood of $v$ (respectively $e$) precisely when $\esv \colon \Cv \to \G$ (respectively $\ese \colon (\shorte) \hookrightarrow \G$) factors through $u$. 
\end{ex}

\subsection{Evaluating graphical species on graphs}\label{ssec Sgraph}

The assignments $\S \mapsto (\shortmid)$, $X \mapsto \CX$ (for finite sets $X$) define a full inclusion $ \Phi \colon \fisinv\hookrightarrow \Gret$. 
Let	$ \yet \colon \Gret \to \GS$ be the induced nerve, given by $\yet \G  \colon \C \longmapsto \Gret (\C, \G)$ for $ \C \in \fisinv$.

\begin{defn}
	\label{def. elements of a graph} For $\G \in \Gret$, the category $\elG[\yet \G] \cong \fisinv \ov \G$ is called the \emph{category of elements of $\G$} and denoted simply by $\elG$.
\end{defn}

The following lemma is proved in \cite[Lemmas~4.16~\&~4.24]{Ray20} and makes precise the intuitive idea that any graph $\G$ is obtained by gluing together edges and corollas (the elements of $\G$):
\begin{lem}\label{lem: essential cover}
	For all graphs $\G$,
	$\G    =  \mathrm{colim}_{ (\C, b) \in \elG} \C$ canonically. In other words, $\fisinv$ is dense in $\Gret$. 
\end{lem}

If $ \E$ is a category with all small limits, then the pullback functor $\Phi^* \colon  \prE{\Gret} \to \GSE$ along $\Phi \colon \fisinv \hookrightarrow \Gret$ has a right adjoint $\Phi_*\colon\GSE \hookrightarrow \prE{\Gret}$ given by
\begin{equation}\label{eq. S graph} S \longmapsto \left(\G \mapsto \mathrm{lim}_{ (\C, b) \in \elG} S(\C)\right ).\end{equation}  Since $\Phi$ is fully faithful, so is $\Phi_*$ \cite[Expos\'e~ii]{SGA4}. The same notation will be used for graphical species viewed as ($\E$-valued) presheaves on $\fisinv$ and for their image as presheaves on $\Gret$.

\begin{ex}\label{ex. product union}
	By (\ref{eq. S graph}), $ S(\G) \times S(\H) = S(\G \amalg \H)$ for all graphs $\G$ and $\H$. In particular, 
	\begin{equation}
		\label{eq. product union}S_X \times S_Y = S(\CX) \times S(\CX[Y]) = S(\CX \amalg \CX[Y]) \text{ for all finite sets } X \text{ and } Y.
	\end{equation}
	So, an external product $\boxtimes$ on $S$ (\cref{def. external}) may be defined by a collection of equivariant maps $S(\CX \amalg \CX[Y]) \to S_{ X \amalg Y}$. 
\end{ex}

\begin{ex}\label{ex. N-structures} \label{ex. mult} Let $X$ be a finite set with distinct elements $x$ and $y$. 
	By (\ref{eq. dagger pullback}), $S_X^{x \ddagger y} \in \E$ is the equaliser of the morphisms $ S(ch_x), S(ch_y \circ \tau)  \colon S_X \rightrightarrows S_\S $. Since the graph $\C_X^{x \ddagger y}$ is defined as the coequaliser of the morphisms $ ch_x, ch_y \circ \tau  \colon (\shortmid) \to \CX$ (see \cref{ex: N graph} ), it follows from (\ref{eq. S graph}) that
	\begin{equation}
		\label{eq. N-structures}S^{x \ddagger y}_X = S(\C_X^{x \ddagger y}).	
	\end{equation} 
	Hence, a contraction $\zeta$ on a graphical species $S$ in $\E$ (\cref{defn: contraction}) is defined by a collection of maps $S(\C^{x \ddagger y}_X) \to S_{ X \setminus \{x,y\}}$ where $X$ is a finite set with distinct elements $x, y \in X$.
	
	Likewise, if $X,Y$ are finite sets with elements $x \in X$, $y \in Y$, then by (\ref{eq mult pullback},
	\begin{equation}
		\label{eq. mult structure}(S_X \times S_Y)^{x \ddagger y} =(S(\CX \amalg \CX[Y]))^{x \ddagger y} =  S((\CX \amalg \CX[Y])^{x \ddagger y}),
	\end{equation} 
	since $(\CX \amalg \CX[Y])^{x \ddagger y}$ is the coequaliser of the morphisms $ ch_x, ch_y \circ \tau  \colon (\shortmid) \to \CX \amalg \CX[Y]$.

	So, a multiplication $\diamond$ on $S$ (\cref{defn: multiplication}) is defined by a collection of maps $S((\CX \amalg \CX[Y])^{x \ddagger y}) \to S_{ X \amalg Y \setminus \{x,y\}}$ where $X, Y$ are finite sets and $x \in X, y \in Y$. 
	
\end{ex}

\begin{defn}\label{ex S structured graphs}
	When $\E = \Set$ and $S \in \GS$, an element $\alpha \in S(\G)$ is called an \emph{$S$-structure on $\G$}. Objects of the category $\Gret (S)$, of \emph{$S$-structured graphs}, are pairs $(\G, \alpha)$ with $\alpha \in S(\G)$, and morphisms $(\G, \alpha) \to (\G', \alpha')$ in $\Gret (S) $ are given by morphisms $f \in \Gret (\G, \G')$ such that $S(f) (\alpha') = \alpha$.

\end{defn}
By \cref{lem: essential cover}, $\Gret$ may be viewed as a full subcategory of $\GS$ and for graphs $\G \in \Gret$, I will write $\G \in \GS$, rather than $\yet \G \in \GS$ where there is no risk of confusion.
\begin{rmk}
	\label{ex topology}
	Let $\G$ be a graph. Families of jointly surjective morphisms define the covers at $\G$ for a canonical \emph{\'etale topology} $J$ on $\Gret$. Since $\elG$ refines every cover at $\G$, the category $\GS$ of graphical species in $\Set$ is equivalent to the category $\sh{\Gret,J}$ of \'etale sheaves on $\Gret$. It follows, in particular, that 
	there is a canonical isomorphism 
	$S(\G) \cong \GS (\yet \G, S)$. Hence, $\Gret(S)$ is canonically isomorphic to the slice category $\Gret \ov S$.
	
\end{rmk}

\begin{ex}\label{def directed graphs} \label{ex OGS on graphs}	
	
	The terminal oriented graphical species $\Dicomm \in \GS$ was defined in \cref{ex directed graphical species}. 
	By \cref{ex topology}, a $\Dicomm$-structure $\phi \in \GS (\G,\Dicomm)$  on a graph $\G$ is given by a bijection $\{ e, \tau e\} \xrightarrow {\cong}\Di$ for each edge orbit $\{e, \tau e\}$ of $\G$. In other words, $\phi$ defines an orientation on $\G$. Hence, the \emph{category $\oGret \cong \Gret (\Dicomm)$ of directed graphs} is defined as the slice category $\Gret \ov \Dicomm$. This is a full subcategory of $\oGS=\GS \ov \Dicomm $ by \cref{lem: essential cover} and each directed graph $\tilde \G = (\G, \phi)\in \oGret$ is canonically the colimit of its induced oriented element category $\oelG$ obtained by pulling back the orientation $\phi$ on $\G$ along each element $(\C,b) \in \elG$. (See also~\cite[Section~4.5]{Ray20}.)
	
	In particular, if $\tilde S \colon {\elG[\Dicomm]}^{\mathrm{op}} \to \E$ is an oriented graphical species in $\E$ corresponding to a pair $(S, \gamma) \in \GSE \ov \DicommE$, then $\tilde S$ extends to a presheaf on $\oGret$ such that, for all graphs $\G \in \Gret$,
	\[ S (\G) \cong\coprod_{ \phi \in \Di(\G)} \tilde S (\G, \phi).\]

\end{ex}
\subsection{Relationship between graphs and Brauer diagrams and wiring diagrams}\label{ssec Brauer}

In \cref{thm. iterated law}, circuit algebras internal to a category $\E$ will be constructed as algebras for a composite monad $\LL\DD\TT$ on $\GSE$ that is described in terms of graphs in $\Gret$. 
 By contrast, enriched circuit algebras are usually defined as algebras over  operads of wiring diagrams (see~e.g.,~\cite{BND17,  DF18, Hal16, Tub14, DHR21} and \cite{RayCA1}). 

In \cite[Section~4]{RayCA1}, the operad $\WD$ of (monochrome unoriented) wiring diagrams is defined using the symmetric monoidal category $\BD$ of Brauer diagrams, used in the representation theory of orthogonal and symplectic groups \cite{LZ15, SS15, RayCA1}. In fact, (enriched) circuit algebras may equivalently be described by symmetric lax monoidal functors from categories of (coloured) Brauer diagrams. 

A partial overview of how different variations of circuit algebras are associated, via Schur-Weyl duality, to group representations, and how they relate to the composite circuit algebra monad $\LL\DD\TT$ is provided in \cite[Introduction]{RayCA1}. This short section reviews the definitions of the category $\BD$ and operad $\WD$, describes how wiring diagrams are related to graphs as in \cref{defn: graph}, and indicates how composition of wiring diagrams is encoded by the monads $\LL, \DD$ and $\TT$ introduced in the following Sections \ref{sec: nonunital}-\ref{sec. iterated}.

Objects of $\BD$ are natural numbers. Morphisms are \emph{(monochrome, unoriented) Brauer diagrams}, where a Brauer diagram in $f \colon m \to n$ is a pair $f =(\tau, \kcl)$ such that 
$\kcl \in \N_{\geq 0}$ is a natural number and $\tau$ is a fixed-point free involution (or \textit{matching}) on $\mm \amalg \nn$ (as usual, $\nn=\{1,\dots, n\}$ for all $n \geq 1$ and $\nul = \emptyset$). 

A Brauer diagram $(\tau, \kcl) \colon m \to n$ may be visualised, as in \cref{fig. pairing comp} (b) by $\kcl$ closed loops next to a univalent graph that comprises of 
a row of $n$ vertices bijectively labelled by $\nn$ above a row of $m$ vertices labelled by $\mm$, 
such that there is an edge connecting vertices $x$ and $y$ in $\mm \amalg \nn$ if and only if $x = \tau y$. 
The monoidal product $\oplus$ on $\BD$ is given by $f_1 \oplus f_2 = (\tau_1 \amalg \tau_2, \kcl_1 + \kcl_2)\in \BD(m_1 + m_2, n_1 + n_2) $ for
$f_1  = (\tau_1, \kcl_1) \in\BD(m_1, n_1)$ and $f_2  = (\tau_2, \kcl_2) \in\BD(m_2, n_2)$. It is represented graphically by juxtaposition. 
Categorical composition is described by ``diagram stacking'' as in  \cref{fig. pairing comp} (see also \cite[Section~3]{RayCA1}).

\begin{figure}[htb!]
	
	\begin{tikzpicture}
		\node at (-2.5,1){(a)};
		\node at (-.5,0){	\begin{tikzpicture}[scale = .3]
				\begin{pgfonlayer}{above}
					\node [dot, violet]  (12) at (-2, 16) {};
					\node  [dot, violet] (13) at (-1, 16) {};
					\node  [dot, violet] (14) at (0, 16) {};
					\node  [dot, violet] (15) at (1, 16) {};
					\node [dot, violet]  (16) at (2, 16) {};
					\node  [dot, violet] (17) at (3, 16) {};
					\node  [dot, violet] (18) at (4, 16) {};
					\node  [dot, sapgreen] (20) at (-0.5, 18) {};
					\node  [dot, sapgreen](21) at (1, 18) {};
					\node  [dot, sapgreen](22) at (2.5, 18) {};
					\node  [dot, cyan](28) at (-2, 12) {};
					\node   [dot, cyan](29) at (-0.5, 12) {};
					\node   [dot, cyan](30) at (1, 12) {};
					\node   [dot, cyan](31) at (2.5, 12) {};
					\node  [dot, cyan] (32) at (4, 12) {};
					\node [dot, violet]  (33) at (-2, 14) {};
					\node  [dot, violet] (34) at (-1, 14) {};
					\node  [dot, violet] (35) at (0, 14) {};
					\node  [dot, violet] (36) at (1, 14) {};
					\node [dot, violet]  (37) at (2, 14) {};
					\node  [dot, violet] (38) at (3, 14) {};
					\node [dot, violet]  (39) at (4, 14) {};
				\end{pgfonlayer}
				\begin{pgfonlayer}{background}
					\filldraw[draw = white, fill = violet, fill opacity  = .1]	(1,16) ellipse (3.6cm and .6cm); 
					\filldraw[draw = white, fill = violet, fill opacity  = .1]	(1,14) ellipse (3.6cm and .6cm);
					\filldraw[draw = white, fill = sapgreen, fill opacity  = .1]	(1,18) ellipse (3.6cm and .6cm);
					\filldraw[draw = white, fill =cyan , fill opacity  = .1]	(1,12) ellipse (3.6cm and .7cm);
					\draw [red, bend left=300, looseness=0.75] (20.center) to (12.center);
					\draw [red, in=90, out=-150] (21.center) to (14.center);
					\draw [red, bend left=60] (22.center) to (18.center);
					\draw [red, bend left=90, looseness=1.25] (16.center) to (17.center);
					\draw [red, bend left=60] (13.center) to (15.center);
					\draw [blue, bend right=75, looseness=1.25] (36.center) to (37.center);
					\draw [blue, in=135, out=-60] (35.center) to (31.center);
					\draw [blue, bend left=75] (28.center) to (29.center);
					\draw [blue, bend left=60, looseness=0.75] (30.center) to (32.center);
					\draw [blue, bend right] (34.center) to (38.center);
					\draw [blue, bend right] (33.center) to (39.center);
				\end{pgfonlayer}
		\end{tikzpicture}};
		\draw [thick, dashed, -> ] (1.5,0)--(2.5,0);
		\node at (5,0){\begin{tikzpicture}[scale = .3]
				\begin{pgfonlayer}{above}
					%
					\node [dot, violet]  (12) at (-2, 15) {};
					\node  [dot, violet] (13) at (-1, 15) {};
					\node  [dot, violet] (14) at (0, 15) {};
					\node  [dot, violet] (15) at (1, 15) {};
					\node [dot, violet]  (16) at (2, 15) {};
					\node  [dot, violet] (17) at (3, 15) {};
					\node  [dot, violet] (18) at (4, 15) {};
					\node  [dot, sapgreen] (20) at (-0.5, 18) {};
					\node  [dot, sapgreen](21) at (1, 18) {};
					\node  [dot, sapgreen](22) at (2.5, 18) {};
					\node  [dot, cyan](28) at (-2, 12) {};
					\node   [dot, cyan](29) at (-0.5, 12) {};
					\node   [dot, cyan](30) at (1, 12) {};
					\node   [dot, cyan](31) at (2.5, 12) {};
					\node  [dot, cyan] (32) at (4, 12) {};
					\node [dot, violet]  (33) at (-2, 15) {};
					\node  [dot, violet] (34) at (-1, 15) {};
					\node  [dot, violet] (35) at (0, 15) {};
					\node  [dot, violet] (36) at (1, 15) {};
					\node [dot, violet]  (37) at (2, 15) {};
					\node  [dot, violet] (38) at (3, 15) {};
					\node [dot, violet]  (39) at (4, 15) {};
				\end{pgfonlayer}
				\begin{pgfonlayer}{background}
					\filldraw[draw = white, fill = violet, fill opacity  = .1]	(1,15) ellipse (3.6cm and .6cm); 
					
					\filldraw[draw = white, fill = sapgreen, fill opacity  = .1]	(1,18) ellipse (3.6cm and .6cm);
					\filldraw[draw = white, fill =cyan , fill opacity  = .1]	(1,12) ellipse (3.6cm and .7cm);
					\draw [red, bend left=300, looseness=0.75] (20.center) to (12.center);
					\draw [red, in=90, out=-150] (21.center) to (14.center);
					\draw [red, bend left=60] (22.center) to (18.center);
					\draw [red, bend left=90, looseness=1.5] (16.center) to (17.center);
					\draw [red, bend left=60, looseness=1.5] (13.center) to (15.center);
					\draw [blue, bend right=75, looseness=1.5] (36.center) to (37.center);
					\draw [blue, in=135, out=-60] (35.center) to (31.center);
					\draw [blue, bend left=75] (28.center) to (29.center);
					\draw [blue, bend left=60, looseness=0.75] (30.center) to (32.center);
					\draw [blue, bend right = 75, looseness=1] (34.center) to (38.center);
					\draw [blue, bend right = 90, looseness=1] (33.center) to (39.center);
				\end{pgfonlayer}
			\end{tikzpicture}
		};
		
		\node at (8,1){(b)};
		\node at (11,0){\begin{tikzpicture}[scale = .4]
				\begin{pgfonlayer}{above}
					\node   (12) at (-2, 15) {};
					\node   (13) at (-1, 15) {};
					\node   (14) at (0, 15) {};
					\node   (15) at (1, 15) {};
					\node   (16) at (2, 15) {};
					\node   (17) at (3, 15) {};
					\node   (18) at (4, 15) {};
					\node  [dot, sapgreen] (20) at (-0.5, 16) {};
					\node  [dot, sapgreen](21) at (1, 16) {};
					\node  [dot, sapgreen](22) at (2.5, 16) {};
					\node  [dot, cyan](28) at (-2, 12) {};
					\node   [dot, cyan](29) at (-0.5, 12) {};
					\node   [dot, cyan](30) at (1, 12) {};
					\node   [dot, cyan](31) at (2.5, 12) {};
					\node  [dot, cyan] (32) at (4, 12) {};
					\node   (33) at (-2, 15) {};
					\node   (34) at (-1, 15) {};
					\node   (35) at (0, 15) {};
					\node   (36) at (1, 15) {};
					\node   (37) at (2, 15) {};
					\node  (38) at (3, 15) {};
					\node   (39) at (4, 15) {};
				\end{pgfonlayer}
				%
				%
				\begin{pgfonlayer}{background}

					\filldraw[draw = white, fill = sapgreen, fill opacity  = .1]	(1,16) ellipse (3.2cm and .6cm);
					\filldraw[draw = white, fill =cyan , fill opacity  = .1]	(1,12) ellipse (3.6cm and .7cm);
					\draw [ in=90, out=-150] (21.center) to (31.center);
					
					\draw [ bend right=90, looseness=1.5] (20.center) to (22.center);
					
					\draw [ bend left=75] (28.center) to (29.center);
					\draw [ bend left=60, looseness=0.75] (30.center) to (32.center);
					\draw [violet] (-1, 14) circle (.5cm);
					
				\end{pgfonlayer}
			\end{tikzpicture}
		};
	\end{tikzpicture}
	\caption{(a) Composition of Brauer diagrams $ f = (\tau_f, 0) \in \BD(5,7)$ and $ g = (\tau_g, 0)\in \BD(7,3)$ 
		(b) the resulting Brauer diagram $gf = (\tau_g\tau_f, 1)\in \BD(5,3)$ with a single loop formed in the composition. }\label{fig. pairing comp}	
\end{figure}

The operad $\WD$ is the $\N$-coloured \emph{operad underlying $\BD$}. This means that, for each tuple of natural numbers $(m_1, \dots, m_k;n)$, the set of \emph{(monochrome) wiring diagrams of type $(m_1, \dots, m_k;n)$} is given by $\WD(m_1, \dots, m_k;n) \defeq \BD (\sum_{i = 1}^k m_i;n)$. For $f \in \BD (\sum_{i = 1}^k m_i;n)$, the corresponding wiring diagram is denoted by $\overline f \in \WD(m_1, \dots, m_k;n)$. The (strict) symmetric monoidal category structure on $\BD$ makes $\WD$ into an $\N$-coloured operad (see \cite[Section~4]{RayCA1}) with composition 
\[ \gamma \colon \WD(m_1, \dots, m_k;n) \otimes \bigotimes_{i = 1}^{k} \WD(l_{i,1}, \dots, l_{i,k_i};m_i) \to \WD(l_{1,1}, \dots, l_{k,k_i};n)  \] given by \[ \gamma \left(\overline g, (\overline f_i)_i\right)  \defeq  \overline{\left(g \circ  (f_1 \oplus \dots \oplus f_n)\right)}.\]
It follows, in particular, that an algebra over $\WD$ (a monochrome circuit algebra) is, equivalently, a lax monoidal functor from $\BD$  (see \cite[Sections~3~\&~4]{RayCA1}).

\begin{rmk}\label{rmk representations}
	Extending the Brauer's 1937 work \cite{Bra37}, categories of (oriented) Brauer diagrams have been used to study the representation theory of the (general linear) orthogonal and symplectic groups (see, e.g.,~\cite{LZ15, RS20, SS15}). 
	
	Building on the wheeled prop characterisation of algebras for the general linear groups in \cite{DM23}, \cite[Theorem~6.3]{RayCA1} characterises algebras over the finite dimensional orthogonal and symplectic groups as (enriched) circuit algebras satisfying two simple relations. 
\end{rmk}

The relationship between (monochrome) wiring diagrams and graphs is straightforward: If a monochrome wiring diagram $\overline f \in \WD (m_1, \dots, m_k; n)$ is induced by a Brauer diagram $ f \in \BD(\sum_{i = 1}^k m_i, n)$ of the form $f = (\tau_f,0)$ (so $f$ has no closed loops), then $f$ describes a graph:
\begin{equation}\label{eq. wiring graph}
	\Fgraphvar{\left(\nn \amalg (\coprod_{i =1}^k \mm_i)\right)}{\coprod_{i =1}^k \mm_i}{\mathbf{k}  }{\text{ inc}}{p}{\tau_{f}}
\end{equation} where $p \colon \coprod_{i =1}^k \mm_i \to \mm[k] $ is the canonical projection. 

\begin{ex}
	\label{ex corolla identity}
	When $X = \nn$, the corolla $\CX$ defined in \cref{corolla} corresponds by (\ref{eq. wiring graph}) to the identity wiring diagram $\overline{id_n} \in \WD(\nn;\nn)$. 
\end{ex}

By \cite[Theorem~2.6]{LZ15} or \cite[Proposition~2.15]{Ban16}, the category $\BD$ is generated, under composition and monoidal product, by the Brauer diagrams $\sigma_\two  \in \BD(2,2)$ (the non-identity bijection between source and target sets that generates the symmetry in $\BD$) and $id_1 \in \BD(1,1)$, $\cap\in \BD(2,0)$ and $\cup \in \BD(0,2)$. 
\begin{ex}\label{ex generating diagrams} With the exception of $\sigma_\two$, each of the generating morphisms of $\BD$ has the form $(\tau_\two, 0)$, where $\tau_\two$ is the unique matching on the two-element set.  Hence $\tau_{\two}$ induces distinct wiring diagrams 
	$\overline{id_1} \in \WD(1;1)$, $\overline \cup \in \WD(-;2)$ and $\overline \cap_{(\two; \nul)} \in \WD(2;0)$ and $\overline \cap _{(\one, \one; \nul)}\in WD(1,1;0)$. By (\ref{eq. wiring graph}), these describe the four graphs with a single edge orbit from \cref{ex stick}:
	\begin{enumerate}[(a)]
		\item for $\overline{id_1} \in \WD(1;1)$, the corresponding graph is the \emph{$\one$-corolla } $\C_{\one}$ (\cref{fig. one orbit}(a));
		\item for $\overline \cup \in \WD(-;2)$, the corresponding graph is the \textit{stick graph} $(\shortmid)$  (\cref{fig. one orbit}(b)); 
		\item for $\overline \cap_{(\two; \nul)} \in \WD(2;0)$, the corresponding graph is the \textit{wheel graph} $\W$ 
		( \cref{fig. one orbit}(c));
		\item for $\overline \cap _{(\one, \one; \nul)}\in WD(1,1;0)$, the corresponding graph is given by \cref{ex stick}(d) (\cref{fig. one orbit}(d)).
	\end{enumerate}

\end{ex}

\begin{rmk}
	Wiring diagrams are, essentially, symmetric versions of Jones's \emph{planar diagrams} \cite{Jon94}. As such, they are commonly represented, not by Brauer diagrams, but in terms of immersions of compact 1-manifolds in punctured discs as in \cref{fig. disc rep} (see e.g., \cite{DHR20} and \cite[Section~4.2]{RayCA1} for details).
	
	Consider the punctured disc representation of the wiring diagram $\overline g$ in \cref{fig. disc rep}. It is clear that, by shrinking the punctures to points, and deleting the outer disc, we obtain a representation of a graph with $3$ vertices of degree $2,4$ and $3$ respectively, and $3$ ports. This corresponds, via (\ref{eq. wiring graph}), to the geometric realisation (\cref{geom real}) of a Feynman graph.  This also holds for each of the wiring diagrams $\overline {f_i}$. Note, however, that a loop is formed in the composition $\overline{\left(g \circ  (f_1 \oplus f_2 \oplus f_3)\right)}$, so shrinking the punctures to points in this diagram is not sufficient to obtain a Feynman graph. 
	
\end{rmk}

\begin{figure}
	[htb!]
	\begin{tikzpicture}
		\node at (0,0){
			\begin{tikzpicture}[scale = .33]
				\begin{pgfonlayer}{above}
					\draw[thick,red] (.5,-.5)--(.5,3.5);
					\node [dot] (1) at (1,0) {};
					\node [dot] (2) at (2,0) {};
					\node [dot] (3) at (3,0) {};
					\draw[ultra thick,cyan] (3.5,-.5)--(3.5,.3);
					\node [dot] (4) at (4,0) {};
					\draw[thick,red] (4.5,-.5)--(4.5,3.5);
					\node [dot] (5) at (5,0) {};
					\node [dot] (6) at (6,0) {};
					\draw[ultra thick,cyan] (6.5,-.5)--(6.5,.3);
					\draw[thick,red] (8.5,-.5)--(8.5,3.5);
					\node [dot] (9) at (9,0) {};
					\draw[ultra thick,cyan] (9.5,-.5)--(9.5,.3);
					\node [dot] (10) at (10,0) {};
					\draw[ultra thick,cyan] (10.5,-.5)--(10.5,.3);
					\node [dot] (11) at (11,0) {};
					\draw[thick,red] (11.5,-.5)--(11.5,3.5);
					
					\node [dot] (12) at (2,3) {};
					\node [dot] (13) at (3,3) {};
					\node [dot] (15) at (5,3) {};
					\node [dot] (16) at (6,3) {};
					\node [dot] (17) at (7,3) {};
					\node [dot] (18) at (8,3) {};
					\node [dot] (19) at (9,3) {};
					\node [dot] (20) at (10,3) {};
					\node [dot] (21) at (11,3) {};
					
					\node [dot] (24) at (4,6) {};
					\node [dot] (26) at (6,6) {};
					\node [dot] (28) at (8,6) {};
				\end{pgfonlayer}
				\begin{pgfonlayer}{background}
					\draw [in=-90, out=90] (1.center) to (13.center);
					\draw [in=-90, out=90] (3.center) to (12.center);
					\draw [in=-90, out=90] (5.center) to (16.center);
					\draw [in=-90, out=90] (6.center) to (18.center);
					\draw [in=-90, out=90] (9.center) to (19.center);
					\draw [in=-90, out=90] (12.center) to (26.center);
					\draw [in=-90, out=90] (16.center) to (24.center);
					\draw [bend left=100, looseness=1] (2.center) to (4.center);
					\draw [bend left=100, looseness=2] (10.center) to (11.center);
					\draw [in=-90, out=90] (19.center) to (28.center);
					\draw [bend right=100, looseness=1] (15.center) to (17.center);
					\draw [bend right=100, looseness=2] (20.center) to (21.center);
					\draw [bend left=100, looseness=1] (15.center) to (17.center);
					\draw [bend left=100, looseness=.5] (13.center) to (20.center);
					\draw [bend left=100, looseness=.8] (18.center) to (21.center);
				\end{pgfonlayer}
		\end{tikzpicture}};
		\draw[gray, dashed, ->, line width = 1](3.4,0)--(4.6,0);
		\node at (4,.2){$\gamma$};
		\node at (8,0){
			\begin{tikzpicture}[scale = .33]
				\begin{pgfonlayer}{above}
					\draw[ultra thick,cyan] (.5,-.5)--(.5,.3);
					\node [dot] (1) at (1,0) {};
					\node [dot] (2) at (2,0) {};
					\node [dot] (3) at (3,0) {};
					\draw[ultra thick,cyan] (3.5,-.5)--(3.5,.3);
					\node [dot] (4) at (4,0) {};
					\draw[ultra thick,cyan] (4.5,-.5)--(4.5,.3);
					\node [dot] (5) at (5,0) {};
					\node [dot] (6) at (6,0) {};
					\draw[ultra thick,cyan] (6.5,-.5)--(6.5,.3);
					\draw[ultra thick,cyan] (8.5,-.5)--(8.5,.3);
					\node [dot] (9) at (9,0) {};
					\draw[ultra thick,cyan] (9.5,-.5)--(9.5,.3);
					\draw[ultra thick,cyan] (10.5,-.5)--(10.5,.3);
					\draw[ultra thick,cyan](11.5,-.5)--(11.5,.3);
					\node [dot] (10) at (10,0) {};
					\node [dot] (11) at (11,0) {};
					
					\node [] (12) at (2,3) {};
					\node [] (13) at (3,3) {};
					\node [] (15) at (5,3) {};
					\node [] (16) at (6,3) {};
					\node [] (17) at (7,3) {};
					\node [] (18) at (8,3) {};
					\node [] (19) at (9,3) {};
					\node [] (20) at (10,3) {};
					\node [] (21) at (11,3) {};
					
					\node [dot] (24) at (4,6) {};
					\node [dot] (26) at (6,6) {};
					\node [dot] (28) at (8,6) {};
				\end{pgfonlayer}
				\begin{pgfonlayer}{background}
					\draw [in=-90, out=90] (3.center) to (26.center);
					\draw [in=-90, out=90] (5.center) to (24.center);
					\draw [in=-90, out=90] (9.center) to (28.center);
					\draw [bend left=100, looseness=1] (2.center) to (4.center);
					\draw [bend left=100, looseness=1] (1.center) to (6.center);
					\draw [bend left=100, looseness=2] (10.center) to (11.center);
					\draw [bend right=100, looseness=1] (19.center) to (21.center);
					\draw [bend left=100, looseness=1] (19.center) to (21.center);
				\end{pgfonlayer}
		\end{tikzpicture}};
	\end{tikzpicture}
	\medspace
	\includegraphics[width=0.95\textwidth]{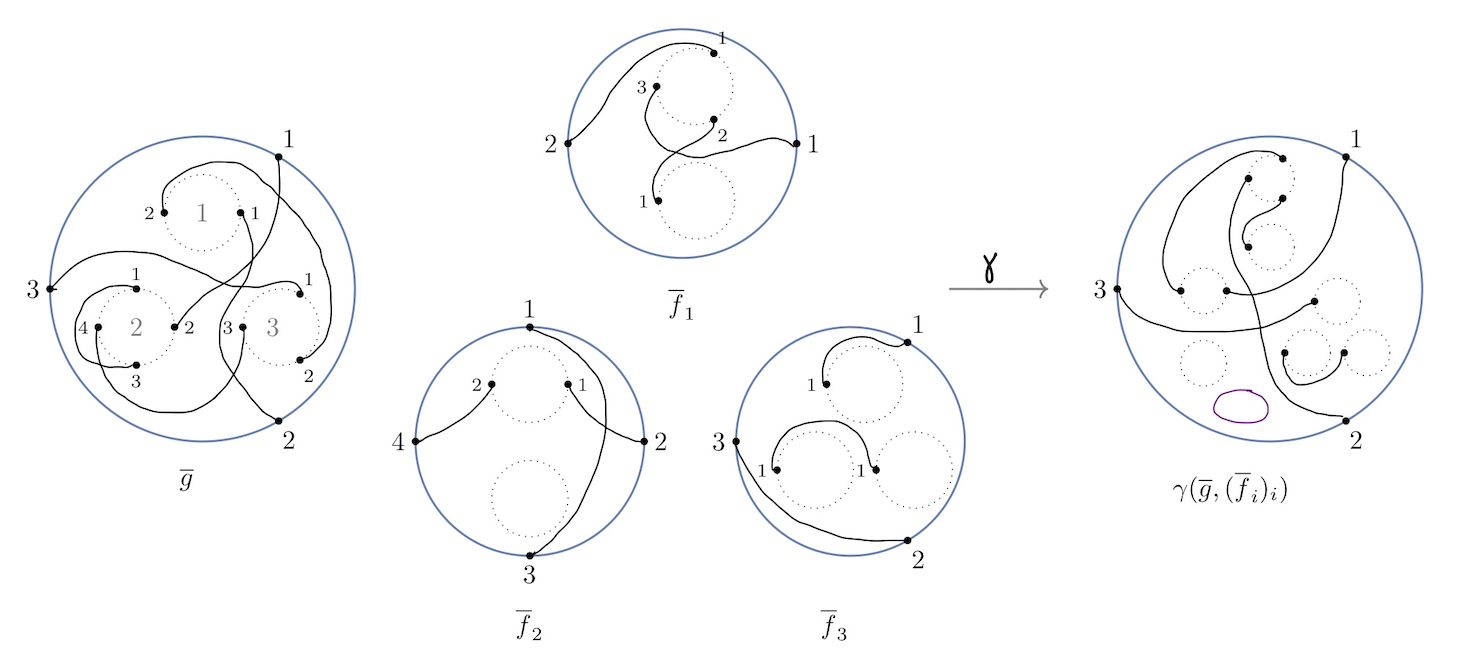}

	\caption{Brauer diagram and punctured disc representations of the same composition of wiring diagrams.}
	\label{fig. disc rep}\label{fig. pictorial}\label{fig. CWD comp}
\end{figure}

Inclusions of connected components in $\Gret$ describe factors of monoidal products in $\BD$. But, in general, morphisms in $\Gret$ do not admit a nice description in terms of $\BD$ (or $\WD$). Nonetheless, since the category $\CO$ of circuit algebras in $\Set$ is equivalent to the category $\bigCA$ described in \cite{RayCA1}, and objects of $\bigCA$ are algebras over operads of wiring diagrams \cite[Section~4]{RayCA1}, there must be constructions on graphs that play the role of operadic composition, and are able to encode the formation of closed loops, in $\WD$. 

\begin{rmk}\label{rmk dBD} (See also Remarks \ref{rmk wiring diag comp} and \ref{rmk. down Brauer graph of graphs})
	The subcategory $\BDd \subset  \BD$ of \emph{downward Brauer diagrams} is generated by $\sigma_\two  = (\{s_1 \mapsto t_2, s_2 \mapsto t_1\}, 0) \in \BD(2,2)$, $id_1 \in \BD(1,1)$ and $\cap\in \BD(2,0)$ (but not $\cup \in \BD (0,2)$). Morphisms in $\BDd(m,n)$ have the form $f = (\tau, 0)$ where, for all $t_i \in \Tf$, $\tau (t_i) \in Sf$. In particular, Brauer diagrams in $\BDd$ are ``loop-free". 
	
	By \cite{SS15} finite-length (vector-space valued) functors from $\BDd$ are equivalent to representations of the infinite dimensional orthogonal (or symplectic) group, and representations of the infinite dimensional general linear group are described by functors from the category of downward \textit{oriented} Brauer diagrams. 
	
	Punctured disc style representations of downward wiring diagrams are precisely those that do not include loops or wires that meet the outer disc at both ends. 
	
In \cite[Section~4]{RayCA1}, nonunital circuit algebras are defined as lax symmetric monoidal functors from (coloured versions of) $\BDd$. In the category $\Set$, these are equivalent to algebras for the monad $\TTk = \LL \TT$ described in \cref{sec: nonunital}, and composition of downward wiring diagrams is given by taking colimits of nondegenerate graphs of graphs as described in \cref{subs: gluing} (see also \cref{rmk. down Brauer graph of graphs}).

In particular, the categories of algebras for the infinite orthogonal, symplectic and general linear groups admit descriptions in terms of nonunital circuit algebras. (See the discussions in \cite[Introduction~\&~Section~6.2]{RayCA1}, and \cite[Theorem~6.13]{RayCA1} in particular.)

\end{rmk}

\begin{rmk}
	\label{rmk wiring diag comp} (See also Remarks \ref{rmk dBD} and \ref{rmk. down Brauer graph of graphs})
	To describe units for circuit algebras, the morphism $\cup \in \BD (0,2)$ is also required. This generates punctured disc representations that may include loops and wires that meet the outer disc at both ends. Composition of general wiring diagrams -- that describes operations of unital circuit algebras \cite{BND17,DHR20,Ray20} -- is related to taking colimits of (possibly degenerate) graphs of graphs (\cref{subs: gluing}). The graphical combinatorics of this are encoded by the unit monad $\DD$ that will be described in \cref{sec. mod op} and the iterated distributive laws that will be described in \cref{sec. iterated}.
\end{rmk}

\begin{rmk} \label{rmk CBD}
	The (internal) construction of circuit algebras discussed in this work characterises \textit{all} circuit algebras as algebras for a single monad on $\GSE$, with special cases, such as wheeled props (oriented circuit algebras) obtained as subcategories or via slice constructions. By contrast, different versions of (enriched) circuit algebras -- such as oriented or nonoriented -- are algebras over distinct (different coloured) operads of wiring diagrams.
	
	New categories of Brauer diagrams -- and hence new operads of wiring diagrams -- may be obtained by adding colours to the connected components of the underlying graphs as in \cite[Section~3.2]{RayCA1}.  In particular, the category $\DiBD$ of \emph{oriented monochrome diagrams}, whose objects are elements of the monoid $\listm \{ \In, \Out \}$, may be defined similarly to $\BD$ by orienting the edges of the Brauer diagrams. (See \cite[Section~3.2]{RayCA1} for a precise definition in terms of the palette $\Dipal$ defined in \cref{ex directed graphical species}.) This category is important in the representation theory of general linear groups \cite{LZ15, SS15, RayCA1}.

\end{rmk}

\section{A monad for nonunital circuit algebras}\label{a free monad}\label{sec: nonunital}\label{nonunital csm section}

This section gives a short outline, closely related to \cite{JK11} and \cite[Section~5]{Ray20}, of the construction of the nonunital circuit algebra monad $\TTk$ on $\GSE$ where, as usual, the category $\E$ is assumed to have enough limits and colimits for all the given constructions.

\subsection{Gluing constructions and labelled graphs}\label{subs: gluing}

As shown in \cref{deg loop}, the graph category $\Gret$ is not closed under finite colimits. However, there is a class of diagrams in $\Gret$ that always admit colimits, and are central to the definition of the nonunital circuit algebra monad $\TTk$ in \cref{ssec nonunital CO}.

The following terminology is based on \cite{Koc16}:
\begin{defn} \label{graph of graphs}
	Let $\G$ be a graph. A \emph{$\G$-shaped graph of graphs} is a functor $ \Gg\colon \elG \to \Gret$ (or $\Gg^\G$) such that
	\[\begin{array}{ll}
		\Gg(a) = (\shortmid) & \text{ for all } (\shortmid, a) \in \elG, \\
		E_0(\Gg(b)) = X &  \text{ for all } (\CX, b) \in \elG,
	\end{array}\]
	and, for all $(\C_{X_b},b) \in \elG $ and all $ x \in X_b$,
	\[ \Gg( ch_x) = ch^{\Gg(b)}_x \in \Gret(\shortmid, \Gg(b)).\]
	\label{defn degenerate}
	A $\G$-shaped graph of graphs $ \Gg\colon \elG \to \Gret$ is \emph{nondegenerate} if, for all $  v \in V$, $\Gg ( \esv)$ has no stick components. Otherwise, $\Gg$ is called \emph{degenerate}.

\end{defn}


%
\begin{figure}[!htb]
	
	\includestandalone[width = .4\textwidth]{substitutionstandalone}
	\caption{A $\G$ shaped graph of graphs $\Gg$ describes \textit{graph substitution} in which each vertex $v$ of $\G$ is replaced by a graph $\G_v$ according to a bijection $E_0(\G_v)  \xrightarrow{\cong} (\vE)^\dagger$. If $\Gg$ is nondegenerate, then taking its colimit corresponds to erasing the inner (blue) nesting. } \label{fig: graph nesting}
\end{figure}

A nondegenerate $\G$-shaped graph of graphs may be viewed informally as a rule for substituting graphs into vertices of $\G$ as in \cref{fig: graph nesting}.

\begin{prop}\label{colimit exists}
	
	A $\G$-shaped graph of graphs $\Gg$ admits a colimit $\Gg(\G)$ in $\Gret$, and, for each $(\C,b) \in \elG$, the universal map $b_{\Gg}  \colon \Gg(b) \to \coGg$ is an embedding in $\Gret$.
	
	If $\Gg$ is nondegenerate and $\Gg(b) \neq \oslash$ for all $(\C, b) \in \elG$, then 
	\begin{itemize}
		\item$\Gg$ describes a bijection $ E(\coGg) \cong E(\G) \amalg \coprod_{v \in V} \EI (\Gg(\esv))$ of sets;
		\item $\Gg$ induces an identity $E_0(\G) \xrightarrow{=} E_0(\coGg)$ of ports;
		\item there is a canonical surjective map $V(\Gg(\G)) \twoheadrightarrow V(\G)$ induced by the embeddings  $\Gg(\esv)\colon \Gg(\esv) \to \Gg(\G)$ for each vertex $v $ of $\G$. 
	\end{itemize}
	
\end{prop}

\begin{proof}
	If $\Gg$ is nondegenerate, then $\Gg$ admits a colimit $\Gg(\G) $ in $\Gret$ by \cite[Proposition~5.16]{Ray20}. In case $\Gg$ is degenerate, the existence of a colimit in $\Gret$ follows from \cite[Proposition~7.15]{Ray20}.

	The other statements follow directly from \cite[Corollary~5.19]{Ray20}.
\end{proof}

\begin{rmk}
	\label{rmk. down Brauer graph of graphs} 
	
	(See also Remarks \ref{rmk dBD} and \ref{rmk wiring diag comp}.) By \cref{prop. CA CO}, there is an equivalence of categories between (internal) circuit algebras in $\Set$ and 
	algebras for operads of wiring diagrams. Under this equivalence, graphs of graphs play the role of operadic composition of wiring diagrams.

	Let $\G$ be a graph with $n$ ports, $k$ vertices, and such that each vertex $v_i \in V$ has $m_i = |\vE[v_i]|$ adjacent edges. Up to orderings, $\G$ describes a wiring diagram $\overline g \in \WD (\mm_1, \dots, \mm_k; \nn)$ with no closed components. A $\G$-shaped graph of graphs $\Gg$ such that $ \Gg(b) \neq \oslash$ for all $(\C,b) \in \elG$ describes a choice of $k$ wiring diagrams $ \overline f^i \in \WD (\bm l _1, \dots, \bm l_{k_i};\mm_i)$ for $( l _1, \dots,  l_{k_i}) \in \listm( \N)$ and $1 \leq i \leq k$.
	
	Non-degeneracy of $\Gg$ corresponds to the condition, discussed in \cite[Remark~3.12]{RayCA1}, that, for each $ \overline f^i$ in the composition, the underlying morphism $f^i $ of Brauer diagrams is a morphism of the downward Brauer category $ f^i  \in \BDd (l_1 + \dots + l_{k_i}, m_i)$. 
	If $\Gg$ is nondegenerate, then its colimit $\Gg(\G)$ is described by the composition $\gamma \left( \overline g , (\overline f_i)_i \right)$ in $\WD$. (Colimits of degenerate graphs of graphs are discussed in \cref{ss. pointed}.) 
	
\end{rmk}

\begin{ex}\label{ex. id Gg}
	By \cref{lem: essential cover}, every graph $\G$ is the colimit of the (nondegenerate) \textit{identity $\G$-shaped graph of graphs} $\Gid$ given by the forgetful functor $\elG \to \Gret$, $(\C, b) \mapsto \C$.  
	
	The identity $\G$-shaped graph of graphs corresponds, as in \cref{rmk. down Brauer graph of graphs} to a composition of wiring diagrams of the form $\gamma \left(\overline g , (\overline{ id}_{m_i})_{i = 1}^k \right)$.

\end{ex}

\begin{defn}\label{defn: graphs of graphs cat}
	The category $\GretG$ of \emph{nondegenerate $\G$-shaped graphs of graphs} is the full subcategory of nondegenerate $\G$-shaped graphs of graphs $\Gg$ in the category of functors $\elG \to \Gret$.
\end{defn}

\begin{defn}
	\label{def. X graph}
	
	Let $X $ be a finite set. A generalised \emph{$X$-graph} (henceforth $X$-graph) is a $\CX$-shaped graph of graphs. An $X$-graph is \emph{admissible} if it is nondegenerate.  
	
	The maximal subgroupoid $\core{\GretG[\CX]}\subset \GretG[\CX]$ is called the \emph{groupoid of (admissible) $X$-graphs} and \emph{$X$-isomorphisms} and is denoted by $X\Grisok$. 
	
	More generally, for any graph $\G$ and any $(\C,b) \in \elG$, $[b]\Grisok$ denotes the groupoid of nondegenerate $\C$-shaped graphs of graphs and isomorphisms.
\end{defn}

Hence, a generalised $X$-graph is given by a pair $\X = (\G, \rho)$, where $\G $ is a graph and $\rho\colon E_0\xrightarrow{\cong} X$. An $X$-graph is \textit{admissible} if it has no stick components, and an $X$-isomorphism in $ X\Grisok (\X, \X')$ is an isomorphism $ g \in \Gr(\G, \G')$ that preserves the $X$-labelling: $\rho' \circ g_{E_0} = \rho\colon E_0 \to X.$

\begin{ex}\label{ex. Xv graph of graphs}
	For any graph $\G$, a (nondegenerate) $\G$-shaped graph of graphs $\Gg \colon \elG \to \Gret$ is precisely a choice of (admissible) $X_b$-graph $\X_b$ for each $(\CX[X_b], b) \in \elG$. 
\end{ex}

Let $S$ be a graphical species in $\E$, and $\G$ a graph. If $\Gg \colon \elG \to \Gret$ is a nondegenerate $\G$-shaped graph of graphs with colimit $\Gg(\G)$ in $\Gret$, then the composition 
\[ \elG^{\mathrm{op}} \xrightarrow{\Gg^{\mathrm{op}}} \Gret^{\mathrm{op}}\xrightarrow {S} \E\]
defines a functor $ \elG^{\mathrm{op}} \to \E$ with limit $S(\Gg(\G))$.

\begin{ex} Recall \cref{ex S structured graphs}. If $S$ is a graphical species in $\Set$ and $\G$ is any graph, then a nondegenerate $\G$-shaped graph of graphs $\Gg$ and choice of $S$-structure $\alpha \in S(\Gg(\G))$ determines a functor $\elG \to \Gret(S)$ -- called a \textit{$\G$-shaped graph of $S$-structured graphs} -- by $b \mapsto S(b_{\Gg})(\alpha)$ (where, for all $(\C, b) \in \elG$, $b_{\Gg}$ is the universal map described in \cref{colimit exists}). 
\end{ex}

\subsection{Nonunital circuit algebras}\label{ssec nonunital CO}

It is now straightforward to construct a monad $\TTk = (\Tk, \muk, \etak)$ on $\GSE$, whose algebras are nonunital circuit algebras in $\E$.

The underlying endofunctor $\Tk  = \Tk_{\E}\colon \GSE \to \GSE$ is defined, for all $S \in \GSE$, by 
\begin{equation}\label{free}
\begin{array}{llll}
\Tk S_\S &= & S_\S, &\\
\Tk S_X &= & \mathrm{colim}_{\X\in X{\Grisok}}  S(\X) & \text{ for all finite sets } X.
\end{array}
	\end{equation}

This acts on isomorphisms in $\fisinv$, $\Tk S$ acts by relabelling graph ports. Let $g \colon \X' \to \X'$ be a port-preserving isomorphism in $X\Grisok$, and for each $x \in X$, let $ch_x^{\X} \in \Gr(\shortmid, \G)$ be the map $ch_{\rho^{-1}(x)}$ defined by $ 1 \mapsto \rho^{-1}(x) \in E_0 (\G)$. 
 Then, \[S(g \circ ch^\X_x) = S(ch^\X_x)S(g) = S(ch^{\X'}_x). \]  Hence, the projections $\Tk S(ch_x)$ 
 induced by
$S({ch^{\X}_x}) \colon S(\X) \to S(\shortmid) = S_\S$, 
are well-defined. So $\Tk S$ extends to all morphisms in $\fisinv$ and therefore defines a graphical species in $\E$. The assignment $S \mapsto \Tk S$ is clearly natural on $S \in \GSE$, and so $\Tk$ defines an endofunctor on $\GSE$.

As in \cref{colimit exists}, for an $X$-graph $\X$ and nondegenerate $\X$-shaped graph of graphs $\Gg$ in $\Gret^{(\X)}$, let $\Gg(\X)$ denote the colimit of $\Gg$ in $\Gret^{(\X)}$. 
\begin{lem}
	\label{lem. Tk mult}
	For any graphical species $S$ in $\E$ and any finite set $X$, there is a canonical isomorphism
	\[({\Tk})^2 S_X \xrightarrow {\cong} \mathrm{colim}_{ \X \in X \Grisok} \left(\mathrm{lim}_{\Gg \in \Gret^{(\X)} }S (\Gg(\X))\right).\] 
\end{lem}
\begin{proof}
	
By definition, for all finite sets $X$,
 \[\begin{array}{lll}
({\Tk})^2 S_X&=&\mathrm{colim}_{ \X \in X \Grisok} \Tk S(\X)\\
& = &\mathrm{colim}_{ \X \in X \Grisok} \mathrm{lim}_{ (\C, b) \in \elG[\X]} \Tk S(\C)\\
& = &\mathrm{colim}_{ \X \in X \Grisok} \mathrm{lim}_{ (\C, b) \in \elG[\X]} \mathrm{colim}_{ Y \in [b]\Grisok} S(\mathcal Y). \end{array}\]

Since $ [b]\Grisok$ is a groupoid by definition,  $\mathrm{colim}_{ \mathcal Y \in [b]\Grisok} S(\mathcal Y) =\mathrm{lim}_{ \mathcal Y \in [b]\Grisok} S(\mathcal Y)$ for all $(\C, b) \in \elG[\X]$. 
By \cref{defn: graphs of graphs cat}, \[ \mathrm{lim}_{\Gg \in \Gret^{(\X)} }S (\Gg(\X)) \cong \mathrm{lim}_{ (\C, b) \in \elG[\X]}\left( \mathrm{lim}_{ \mathcal Y \in [b]\Grisok} S(\mathcal Y)\right)\] canonically, and therefore, by the universal property of (co)limits, there is a unique isomorphism
 \begin{equation}\label{eq. T2 iso}\begin{array}{lll}
(\Tk)^2 S_X&
\cong & \mathrm{colim}_{ \X \in X \Grisok} \left(\mathrm{lim}_{ \Gg \in \Gret^{(\X)} } S(\Gg(\X))\right). 
\end{array}\end{equation}\end{proof}

By \cref{colimit exists}, each $\Gg(\X)$ inherits the structure of an $X$-graph from $\X$. Hence, by an application of \cref{lem. Tk mult}, the assignment $(\X, \Gg)\to \Gg(\X)$ describes a $\E$-morphism $\mu^{\TTk} S_X \in E \left(({\Tk})^2 S_X,\Tk S_X\right)$ in $\E$ given by 
\begin{equation}\label{eq. T mult} \mu^{\TTk} S_X  \colon (\Tk)^2 S_X =  \mathrm{colim}_{ \X \in X \Grisok}\left( \mathrm{lim}_{ \Gg \in \Gret^{(\X)} } S(\Gg(\X))\right)\to \mathrm{colim}_{ \X' \in X \Grisok} S(\X') = \Tk S_X.\end{equation}

This is natural in $X$ and in $S$ and hence defines a natural transformation $\mu^{\TTk}\colon (\Tk)^2 \Rightarrow \Tk$.

\label{eta defn} 
 The inclusion $\fisinv \hookrightarrow \Gret$ induces a natural transformation $ \etak  \colon id_{\GS} \Rightarrow \Tk$, $ S_X \xrightarrow = S(\CX) \to \Tk S_X$ that provides a unit for the endofunctor $\Tk$ on $\GSE$. It is straightforward to verify that the triple $\TTk = (\Tk, \mu^{\TTk}, \eta^{\TTk})$ defines a monad. 

\begin{ex}\label{ex. free circuit algebra} For $S \in \GSE$, $\Tk S$ admits a canonical nonunital circuit algebra structure:

		Disjoint union of graphs induces an external product $\boxtimes$ 
		on $\Tk S$. For, let $X,Y$ be finite sets and let $\X$ be an $X$-graph and $\Y$ a $Y$-graph. 
Then, the following diagram -- in which the unlabelled maps are the defining morphisms for $\Tk$ -- commutes in $\E$ and describes an external product $\boxtimes$ on $\Tk S$:
	\[\xymatrix{ S(\X \amalg \Y ) \ar[rrd] \ar@{=}[r]& S(\X) \times S(\Y) \ar[r]& \Tk S_X \times \Tk S_Y \ar@{=}[r] \ar@{-->}[d]_-{\boxtimes} & \Tk S(\CX \amalg \C_Y) \ar[r] & (\Tk)^2 S_{X \amalg Y} \ar[dll]^-{\mu^{\TTk}} \\
		&& \Tk S_{X \amalg Y}.&&}\]
		Since $\oslash \in \nul \Grisok$, there is a morphism $ *  = S(\oslash) \to \Tk S_\nul$ from the terminal object $* $ in $\E$. This provides an external unit for $\boxtimes$.

	Similarly, the contraction $\zeta$ on $\Tk S$ is induced by commuting diagrams of the form 
\[\xymatrix{ S(\X^{x \ddagger y})\ar[rr] \ar[rrd] && \Tk S(\C^{x \ddagger y}_X) \ar[rr] \ar@{-->}[d]_-{\zeta_X^{x \ddagger y}} && (\Tk)^2 S_{X\setminus \{x,y\}} \ar[dll]^-{\mu^{\TTk}} \\
	&& \Tk S_{X\setminus \{x,y\}}&&}\] where $X$ is a finite set with distinct elements $x, y \in X$, and $\X$ an $X$-graph.
	
It follows immediately from the definitions (and the diagrams in the proof of \cref{prop nonunital CA}below) that $\boxtimes $ and $\zeta$ satisfy (C1)-(C3) of \cref{def. co}.

\end{ex}

\begin{prop}\label{prop nonunital CA}
	There is a canonical isomorphism of categories $\GSE^\TTk  \cong \nuCOE$.
\end{prop}

\begin{proof} 
Let $(A,h)$ be an algebra for $\TTk$. 

For finite sets $X$ and $Y$, the morphism $\boxtimes_{X,Y} \colon A_X \times A_Y \to A_{X \amalg Y}$ in $\E$ is obtained by composing the structure morphism $ A_X \times A_Y  = A(\C_X \amalg \C_Y) \to {\Tk} A_{X \amalg Y}$ with ${h} \colon {\Tk} A_{X \amalg Y} \to A_{X \amalg Y}$. Moreover, the composition
$ *  = A(\oslash) \to \Tk A_\nul \xrightarrow {h} A_\nul$ provides a unit for the external product $\boxtimes $ so defined.

If a finite set $X$ has distinct elements $x$ and $y$, then we may define a morphism $\zeta_{X}^{x \ddagger y} \colon A_X \to A_{X\setminus \{x,y\}}$  in $\E$ by composing 
$ A_X^{x\ddagger y}   = A(\C_X^{x\ddagger y} )\to {\Tk} A_{X\setminus \{x,y\}}$ with $h$. 

The proof that $ (A, \boxtimes, \zeta)$ satisfies the axioms for a circuit algebra proceeds by representing each side of the equations in (C1)-(C3) of \cref{def. co} in terms of graph of graphs, and then showing that these have the same colimit in $\Gret$. 

For (C1), the external product $\boxtimes_{X,Y}$ is represented by the graph $\C_X \amalg \C_Y$, and so $\boxtimes_{X \amalg Y, Z} \circ  (\boxtimes_{X , Y}\  \times \  id_{A _Z})$ is represented by the $ \C_{X \amalg Y} \amalg \C_Z$-shaped graphs of graphs $\left (\iota_{X \amalg Y}\mapsto  \CX \amalg \CX[Y], \ \iota_Z \mapsto \CX[Z]\right )$. This has  
 colimit $\CX \amalg \CX[Y] \amalg \CX[Z]$ which is also the colimit of the $\CX \amalg \C_{Y \amalg Z}$-shaped graph of graphs $\left(\iota_X \mapsto \CX, \  \iota_{Y \amalg Z} \mapsto \CX[Y] \amalg \CX[Z]\right)$. 

More precisely, it follows from the monad algebra axioms that, for all finite sets $X, Y$ and $Z$, the following diagram,  in which the maps $ A_{X \amalg Y} \times A_Z \rightarrow  A_{X \amalg Y \amalg Z} \leftarrow A_X \times A_{ Y \amalg Z}$ are just $\boxtimes_{ X \amalg Y, Z}$ and  $\boxtimes_{X,Y \amalg Z}$, commutes:
\[\small{\xymatrix@C=.43cm@R = .7cm{
&&& A_{\scriptscriptstyle X} \times A_{\scriptscriptstyle Y }\times A_{\scriptscriptstyle Z}\ar[dl]_-{=} \ar[d]^{=}\ar[dr]^-{=}&&\\
&A_{\scriptscriptstyle {X \amalg Y}} \times A_{\scriptscriptstyle Z } \ar@{=}[d]&\ar@{-->}[l]_-{(\boxtimes,id)}A(\C_X \amalg \C_Y) \times A(\C_Z)\ar[r]^-{=}\ar[d]& A(\C_X \amalg \C_Y \amalg \C_Z)\ar[dd]& A(\C_X) \times A(\C_Y \amalg \C_Z)\ar[d] \ar[l]_-{=}\ar@{-->}[r]^-{(id, \boxtimes)}& A_{\scriptscriptstyle  X }\times A_{\scriptscriptstyle { Y \amalg Z}} \ar@{=}[d]\\
&A(\C_{\scriptstyle {X \amalg Y}} \amalg \C_Z)  \ar[d]&{\Tk} A (\C_{\scriptstyle {X \amalg Y}} \amalg \C_Z) \ar[l]^-{h}\ar[d]&& {\Tk} A(\CX \amalg \C_{\scriptstyle { Y \amalg Z}}) \ar[d]\ar[r]_-{h}\ar[d]& A (\C_X \amalg \C_{\scriptstyle {Y\amalg Z}}) \ar[d]\\
&{\Tk} A_{\scriptscriptstyle{ X \amalg Y \amalg Z}}\ar[rrdd]_-{h}& \ar[l]^{{\Tk} h}{\Tk} ^2A_{ \scriptscriptstyle{ X \amalg Y \amalg Z}} \ar[r]_-{ \mu^{\TTk} k A}& {\Tk} A_{\scriptscriptstyle {X \amalg Y \amalg Z}}\ar[dd]^-{h}& \ar[l]^-{ \mu^{\TTk} k A}{\Tk} ^2A_{ \scriptscriptstyle {X \amalg Y \amalg Z}}\ar[r]_{{\Tk} h} &  {\Tk} A_{\scriptscriptstyle {X \amalg Y \amalg Z}}\ar[ddll]^-{h}\\
&&&&&\\
&&& A_{X \amalg Y \amalg Z}. &&}}\]
Hence, $\boxtimes $ satisfies (C1) of \cref{def. co}.

For (C2), if $w,x,y,z$ are mutually distinct elements of $X$, then, by the monad algebra axioms, the following diagram commutes in $\E$, and hence $\zeta$ satisfies (C2):
	\[\small{\xymatrix@C=.43cm@R = .65cm{
		&& (A_{\scriptscriptstyle X})^{\scriptscriptstyle {w \ddagger x, y\ddagger z} } \ar[dl]_-{=} \ar[d]^{=}\ar[dr]^-{=}&&\\
	\left	(	A_{\scriptscriptstyle {X \setminus \{w,x\}} }\right )^{\scriptscriptstyle {y \ddagger z}} \ar@{=}[d]&\ar@{-->}[l]_-{\zeta^{\scriptscriptstyle {w \ddagger x}}}\left(A(\C_{\scriptscriptstyle X}^{ \scriptscriptstyle {w \ddagger x}})\right)^{\scriptscriptstyle {y \ddagger z}} \ar[r]^-{=}\ar[d]& A\left ((\C_X \amalg \C_Y) ^{\scriptscriptstyle {w \ddagger x, y \ddagger z}}\right )\ar[dd]& \left(A(\C_{\scriptscriptstyle X}^{ \scriptscriptstyle {y \ddagger z}})\right)^{\scriptscriptstyle {x \ddagger y}}\ar[d] \ar[l]_-{=}\ar@{-->}[r]^-{\zeta^{y \ddagger z}}& \left	(	A_{\scriptscriptstyle {X \setminus \{y,z\}}}\right )^{\scriptscriptstyle{ w \ddagger x}}\ar@{=}[d]\\
		A(\C_{\scriptscriptstyle {X \setminus \{w,x\}}}^{\scriptscriptstyle{ y \ddagger z}} ) \ar[d]&{\Tk} A (\C_{\scriptscriptstyle {X\setminus \{w,x\}}}^{\scriptscriptstyle {y \ddagger z}}) \ar[l]_-{h}\ar[d]&&{\Tk} A (\C_{\scriptscriptstyle {X\setminus \{y,z\}}}^{\scriptscriptstyle {w \ddagger x}}) \ar[d]\ar[r]^-{h}\ar[d]& 	A(\C_{\scriptscriptstyle {X \setminus \{y,z\}}}^{\scriptscriptstyle {w \ddagger x}})  \ar[d]\\
		{\Tk} A_{\scriptscriptstyle {X\setminus \{w,x,y,z\}}}\ar[rrdd]_-{h}& \ar[l]_{{\Tk} h}{\Tk} ^2A_{\scriptscriptstyle {X\setminus \{w,x,y,z\}}} \ar[r]^-{ \mu^{\TTk} k A}& {\Tk} A_{\scriptscriptstyle {X\setminus \{w,x,y,z\}}}\ar[dd]^-{h}& \ar[l]_-{ \mu^{\TTk} k A}{\Tk} ^2A_{\scriptscriptstyle {X\setminus \{w,x,y,z\}}}\ar[r]^{{\Tk} h} &  {\Tk} A_{\scriptscriptstyle {X\setminus \{w,x,y,z\}}}\ar[ddll]^-{h}\\
	&&&&\\
		&& A_{\scriptscriptstyle {X\setminus \{w,x,y,z\}}}.&&}}\]

Finally, if $X$ and $Y$ are finite sets and $x_1$ and $x_2$ are distinct elements of $X$, then the diagram
	\[\scriptsize{\xymatrix@C=.4cm@R = .7cm{
			&& (A_{\scriptscriptstyle {X} }\times A_{\scriptscriptstyle {Y}})^{\scriptscriptstyle{ x_1 \ddagger x_2} }\ar[dl]_-{=}\ar[dr]^-{=}&&\\
			\left	(	A_{\scriptscriptstyle{ X \amalg Y}} \right )^{\scriptscriptstyle {x_1 \ddagger x_2}} \ar@{=}[d]&\ar@{-->}[l]_-{\boxtimes}\left(A(\C_X \amalg \C_Y)\right)^{\scriptscriptstyle {x_1\ddagger x_2}} \ar[ddr]\ar@{=}[rr]\ar[d]&& A(\C_X^{\scriptscriptstyle { x_1 \ddagger x_2}})\times A_Y \ar[ddl]\ar[d] \ar@{-->}[r]^-{\zeta^{x_1 \ddagger x_2}}& 	A_{\scriptscriptstyle {X \setminus \{x_1,x_2\}}}\times A_Y\ar@{=}[d]\\
			A(\C_{\scriptstyle {X \amalg Y}}^{\scriptscriptstyle {x_1 \ddagger x_2}} ) \ar[d]
			&{\Tk} 	A(\C_{\scriptstyle{ X \amalg Y}}^{\scriptscriptstyle {x_1 \ddagger x_2}} ) \ar[l]_-{h}\ar[d]&&{\Tk} A (\C_{\scriptstyle {X\setminus \{x_1, x_2\}}}\amalg \C_Y) \ar[d]\ar[r]^-{h}\ar[d]& 	A (\C_{\scriptstyle {X\setminus \{x_1, x_2\}}}\amalg \C_Y)   \ar[d] \\
						{\Tk} A_{\scriptscriptstyle {(X \amalg Y)\setminus \{x_1, x_2\}}}\ar[rrdd]_-{h}& \ar[l]_-{{\Tk} h} {\Tk} ^2A_{\scriptscriptstyle {(X \amalg Y)\setminus \{x_1, x_2\}}}  \ar[r]^-{ \mu^{\TTk} k A}& {\Tk} A_{\scriptscriptstyle {(X \amalg Y)\setminus \{x_1, x_2\}}}  \ar[dd]^{h}& {\Tk} ^2A_{\scriptscriptstyle{ (X \amalg Y)\setminus \{x_1, x_2\}}}  \ar[r]^-{{\Tk} h} \ar[l]_-{ \mu^{\TTk} k A} &  {\Tk} A_{\scriptscriptstyle {(X \amalg Y)\setminus \{x_1, x_2\}}} \ar[ddll]^-{h} \\
			&&&&\\
			&& A_{\scriptscriptstyle{ (X \amalg Y)\setminus \{x_1, x_2\}}} &&}}\] commutes.
		Therefore, (C3) is satisfied and $(A, \boxtimes, \zeta)$ is a nonunital circuit algebra. 
		
		Conversely, let $(S, \boxtimes, \zeta)$ be a nonunital circuit algebra. For finite sets $X_1, \dots, X_n$, let $\boxtimes_{ X_1, \dots, X_n}: \prod_{i = 1}^n S_{X_i} \to S_{ \coprod_i X_i}$ be the obvious morphism induced by $\boxtimes$. This is well defined since $ (S, \boxtimes, \zeta)$ satisfies (C1)-(C3). Recall from \cref{ssec CO} that, if $x_1, y_1, \dots, x_k, y_k$ are mutually distinct elements of $\coprod_i X_i$, then $(\prod_i S_{X_i})^{x_1 \ddagger y_1, \dots, x_k \ddagger y_k }$ is the limit of the diagram of parallel morphisms 
		\[\boxtimes_{ X_1, \dots, X_n} \circ S(ch_{x_i}), \boxtimes_{ X_1, \dots, X_n} \circ S(ch_{y_i}\circ \tau) \colon \prod_{i = 1}^n S_{X_i} \to S_{ \coprod_i X_i} \to S_\S, \] and that this is independent of the ordering of the pairs $(x_i, y_i)$.

	Let $X$ be a finite set and let $\X$ be an admissible 
		$X$-graph. For each $v \in V$, let $(\CX[Z_v], b_v) \in \elG[\X]$ be a {neighbourhood} of $v$, and let $Z_v \cong |v|$ be such that $\coprod_v Z_v = X \amalg \{y_1, z_1,\dots, y_k, z_k\}$.
		
		Since $\X$ 
		has no stick components, 
		\begin{equation}\label{eq. SX resolution}
		S(\X) = \left(\prod_{v \in V(\X)} S_{Z_v}\right )^{y_1 \ddagger z_1, \dots, y_k \ddagger z_k} .
		\end{equation} 
	In particular, the morphism $\prod_v S_{Z_v} \to S_{(\coprod_v Z_v)}$ induces a morphism \[\boxtimes^V \colon S(\X) \to S_{ \coprod_{v\in V} Z_v}^{y_1 \ddagger z_1, \dots, y_k \ddagger z_k}. \]
	
	Let $\hat h_\X \colon S(\X) \to S(\C_X)  = S_X$ be the map (illustrated in \cref{fig. COnu structure map}) defined, using (\ref{eq. SX resolution}), by the composition 
	\begin{equation}
		\label{eq. prod eq} \xymatrix{ S(\X) \ar[rr]^-{\boxtimes^{V}} && 
			S_{ \coprod_{v\in V} Z_v}^{y_1 \ddagger z_1, \dots, y_k \ddagger z_k}\ar[rr]^- {\zeta^{(k-1)}}&&S_X}
	\end{equation}
	where $\zeta^{(k-1)}$ is obtained by the obvious iterated contractions. Since $(S, \boxtimes, \zeta)$ satisfies (C1) and (C2), these are independent of choices of ordering.

\begin{figure}[htb!]
	 \begin{tikzpicture} [scale =  0.9]
	
	\node at(0,3){
		\begin{tikzpicture}[scale = 0.3]\begin{pgfonlayer}{above}
		\node [dot, red] (0) at (0, 0) {};
		\node  [dot, red](1) at (2, 0) {};
		\node [dot, red] (2) at (-2, 1) {};
		\node [dot, red] (3) at (-1, -1) {};
		\node  (4) at (-4.5, 2.25) {};
		\node  (5) at (-3, -3) {};
		\node  (6) at (0, -4) {};
		\node  (7) at (4, 2) {};
		\node  (8) at (5, -1) {};
		\node  (9) at (5, -3) {};
		\node  (10) at (-1, 2) {};
		
		\end{pgfonlayer}
		\begin{pgfonlayer}{background}
		\draw [bend left=45, looseness=1.75] (0.center) to (1.center);
		\draw [bend right=60, looseness=1.25] (2.center) to (3.center);
		\draw [in=165, out=15, looseness=1.25] (4.center) to (2.center);
		\draw [in=60, out=-105, looseness=1.25] (3.center) to (5.center);
		\draw [bend left=15] (3.center) to (6.center);
		\draw [bend left=45] (1.center) to (7.center);
		\draw [in=105, out=-30] (1.center) to (8.center);
		\draw [in=-180, out=-45, looseness=1.25] (1.center) to (9.center);
		\draw [in=135, out=90, looseness=1.50] (2.center) to (10.center);
		\draw [in=315, out=-15] (2.center) to (10.center);
		\draw [bend right=60, looseness=1.50] (0.center) to (1.center);
		\end{pgfonlayer}\end{tikzpicture}};
	
	\node at (-6.5,0){\begin{tikzpicture}[scale = 0.2]\begin{pgfonlayer}{above}
		\node [dot, red] (0) at (-0.25, 0) {};
		\node [dot, red] (1) at (1, 0) {};
		\node [dot, red] (2) at (-1.5, 0.5) {};
		\node  [dot, red](3) at (0, -1) {};
		\node  (4) at (-7.5, 3) {};
		\node  (5) at (-3.5, -4.75) {};
		\node  (6) at (-0.5, -5.75) {};
		\node  (7) at (7, 2) {};
		\node  (8) at (7, -1.75) {};
		\node  (9) at (6.75, -4) {};
		\node  (10) at (-5, 6.5) {};
		\draw[draw = blue, fill = blue, fill opacity = .2] (-.2,.1) circle (2cm);
		\end{pgfonlayer}
		\begin{pgfonlayer}{background}
		\draw [bend left=120, looseness=22.00] (0.center) to (1.center);
		\draw [bend right=60, looseness=1.25] (2.center) to (3.center);
		\draw [in=165, out=15, looseness=1.25] (4.center) to (2.center);
		\draw [in=60, out=-105, looseness=1.25] (3.center) to (5.center);
		\draw [bend left=15] (3.center) to (6.center);
		\draw [bend left=45] (1.center) to (7.center);
		\draw [in=105, out=-30] (1.center) to (8.center);
		\draw [in=-180, out=-45, looseness=1.25] (1.center) to (9.center);
		\draw [in=-180, out=90, looseness=0.75] (2.center) to (10.center);
		\draw [in=0, out=75, looseness=0.50] (2.center) to (10.center);
		\draw [bend left=105, looseness=16.25] (0.center) to (1.center);
		\end{pgfonlayer}\end{tikzpicture}};
	\draw  [->, dashed,  line width = 1.5, draw= gray](-5,0)--(-2,0);

	\node at (-3.5,.5){\tiny{apply $\boxtimes$ locally at vertices}
	};
	
	\node at (0,0){ 
		\begin{tikzpicture}[scale = 0.2]\begin{pgfonlayer}{above}
		\node [dot,blue] (0) at (0, 0) {};
		\node  (1) at (0, 0) {};
		\node  (2) at (0, 0) {};
		\node  (3) at (0, 0) {};
		\node  (4) at (-7, 0.25) {};
		\node  (5) at (-3, -4.25) {};
		\node  (6) at (0.25, -5) {};
		\node  (7) at (6.25, 3.25) {};
		\node  (8) at (5.25, -1.75) {};
		\node  (9) at (5, -4) {};
		\node  (10) at (-2.25, 1) {};
		\node  (11) at (0, 2.25) {};
		\node  (12) at (0, 1.5) {};
			\node  (13) at (-1.5, -1.8) {};
		\draw[draw =green, fill = green, fill opacity = .2] (-.2,.1) circle (2.5cm);
		\end{pgfonlayer}
		\begin{pgfonlayer}{background}
	\draw [in=-150, out=15, looseness=1.25] (4.center) to (2.center); 
	\draw [in=60, out=-105, looseness=1.25] (3.center) to (5.center);
	\draw [bend left=15] (3.center) to (6.center);
	\draw [in=150, out=15] (1.center) to (7.center);
	\draw [in=105, out=-30] (1.center) to (8.center);
		\draw [in=-180, out=-45, looseness=1.25] (1.center) to (9.center);
		\draw [in=-135, out=-165] (2.center) to (10.center);
		\draw [in=45, out=75, looseness=0.50] (2.center) to (10.center);
		\draw [in=135, out=165, looseness=0.75] (12.center) to (3.center);
		\draw [in=60, out=-15, looseness=0.75] (12.center) to (3.center);
		\draw [in=150, out=180] (11.center) to (3.center);
		\draw [in=30, out=0] (11.center) to (3.center);
		\draw [in=135, out=165, looseness=0.5] (13.center) to (3.center);
		\draw [in=60, out=-15, looseness=0.5] (13.center) to (3.center);
		\end{pgfonlayer}\end{tikzpicture}};
	\draw  [->, dashed, line width = 1.5, draw= gray](2,0)--(5,0);
	\node at (3.5,.5){\tiny{apply $\zeta$ on inner edges}
	};

	\node at (6.5,0){ 
		\begin{tikzpicture}[scale = 0.2]	\begin{pgfonlayer}{above}
		\node [dot, green] (0) at (0, 0) {};
		\node  (1) at (0, 0) {};
		\node  (2) at (0, 0) {};
		\node  (3) at (0, 0) {};
		\node  (4) at (-7, 0.25) {};
		\node  (5) at (-3, -4.25) {};
		\node  (6) at (0.25, -5) {};
		\node  (7) at (6.25, 3.25) {};
		\node  (8) at (5.25, -1.75) {};
		\node  (9) at (5, -4) {};
		\end{pgfonlayer}
		\begin{pgfonlayer}{background}
		\draw [bend right=60, looseness=1.25] (2.center) to (3.center);
		\draw [in=-150, out=15, looseness=1.25] (4.center) to (2.center);
		\draw [in=60, out=-105, looseness=1.25] (3.center) to (5.center);
		\draw [in=75, out=-75] (3.center) to (6.center);
		\draw [in=150, out=15] (1.center) to (7.center);
		\draw [in=105, out=-30] (1.center) to (8.center);
		\draw [in=-180, out=-45, looseness=1.25] (1.center) to (9.center);
		\end{pgfonlayer}\end{tikzpicture}};\end{tikzpicture}
\caption{ A representation of the structure map $\hat h_\X \colon S(\X)\to S_X$.} \label{fig. COnu structure map}
\end{figure}

The maps $\hat h_\X$ so defined are the components of a morphism $\hat h \colon \Tk S \to S$ in $\GSE$.

Moreover, for any $\X$-shaped graph of graphs $\Gg$, $\Gg(b_v)$ is a $Z_v$-graph. So, by (C3), we obtain morphisms $\hat h_{ \Gg(b_v)} \colon S(\Gg(b_v)) \to S_{Z_v}$, and by iterating (\ref{eq. prod eq}), it is straightforward to show that $S$, together with $\hat h \colon \Tk S \to S$ satisfy the monad algebra axioms. (See \cref{fig. COnu h alg} for an illustration of coherence of $\hat h$ and $\mu^{\TTk}$.)
\end{proof}

\begin{figure}[htb!]
\begin{tikzpicture} [scale = .9]
\node at(-5,0){ 
	\begin{tikzpicture}[scale = .4]
	\begin{pgfonlayer}{above}
	\node [dot, red] (0) at (0, 0) {};
	\node  [dot, red](1) at (2, 0) {};
	\node [dot, red] (2) at (-2, 1) {};
	\node [dot, red] (3) at (-1, -1) {};
	\node  (4) at (-4.5, 2.25) {};
	\node  (5) at (-3, -3) {};
	\node  (6) at (0, -4) {};
	\node  (7) at (4, 2) {};
	\node  (8) at (5, -1) {};
	\node  (9) at (5, -3) {};
	\node  (10) at (-1, 2) {};
	\draw[draw = blue, fill = blue, fill opacity =.2] (-1.6,1.2) circle (1.2cm);
	\draw[draw = blue, fill = blue, fill opacity =.2] (-1,-1) circle (.6cm);
	\draw[draw = blue, fill = blue, fill opacity =.2] (1,0)ellipse (1.7cm and .9cm);
	\end{pgfonlayer}
	\begin{pgfonlayer}{background}
	\draw [bend left=45, looseness=1.75] (0.center) to (1.center);
	\draw [bend right=60, looseness=1.25] (2.center) to (3.center);
	\draw [in=165, out=15, looseness=1.25] (4.center) to (2.center);
	\draw [in=60, out=-105, looseness=1.25] (3.center) to (5.center);
	\draw [bend left=15] (3.center) to (6.center);
	\draw [bend left=45] (1.center) to (7.center);
	\draw [in=105, out=-30] (1.center) to (8.center);
	\draw [in=-180, out=-45, looseness=1.25] (1.center) to (9.center);
	\draw [in=135, out=90, looseness=1.50] (2.center) to (10.center);
	\draw [in=315, out=-15] (2.center) to (10.center);
	\draw [bend right=60, looseness=1.50] (0.center) to (1.center);
	\end{pgfonlayer}
	\end{tikzpicture}
};
\node at (-4,-1.2){$(\Tk)^2 S$};
\node at (4,-1.2){$\Tk S$};
\draw  [->, line width = 1.22, draw= gray, dashed](-2,0)--(2,0);
\node at (0,.5){
	$\Tk h$
};
\draw  [->, line width = 1.2, draw= gray](-5,-2)--(-5,-3);
\node at (-5.8,-2.5){$\mu^{\TTk}S$
};

\node at (5,0){ 
	\begin{tikzpicture}[scale = .4]
	\begin{pgfonlayer}{above}
	\node [dot, blue] (0) at (0, 0) {};
	\node  [dot, blue](1) at (0, 0) {};
	\node [dot, blue] (2) at (-1, 1) {};
	\node [dot, blue] (3) at (-1, -1) {};
	\node  (4) at (-4.5, 2.25) {};
	\node  (5) at (-3, -3) {};
	\node  (6) at (0, -4) {};
	\node  (7) at (4, 2) {};
	\node  (8) at (5, -1) {};
	\node  (9) at (5, -3) {};
	\node  (10) at (-1, 2) {};
	\end{pgfonlayer}
	\begin{pgfonlayer}{background}
	\draw [bend left=45, looseness=1.75] (0.center) to (1.center);
	\draw [bend right=60, looseness=1.25] (2.center) to (3.center);
	\draw [in=165, out=15, looseness=1.25] (4.center) to (2.center);
	\draw [in=60, out=-105, looseness=1.25] (3.center) to (5.center);
	\draw [bend left=15] (3.center) to (6.center);
	\draw [bend left=45] (1.center) to (7.center);
	\draw [in=105, out=-30] (1.center) to (8.center);
	\draw [in=-180, out=-45, looseness=1.25] (1.center) to (9.center);
	\draw [bend right=60, looseness=1.50] (0.center) to (1.center);
	\end{pgfonlayer}
	\end{tikzpicture}
};
\draw  [->, line width = 1.2, draw= gray](5,-2)--(5,-3);

\node at (5.5,-2.5){
	$h$
};

\node at (-4,-5.5){$\Tk S$};
\node at (3.5,-5.5){$ S$};
\node at (-5,-4.5){ 
	\begin{tikzpicture}[scale = .4]
	\begin{pgfonlayer}{above}
	\node [dot, red] (0) at (0, 0) {};
	\node  [dot, red](1) at (2, 0) {};
	\node [dot, red] (2) at (-2, 1) {};
	\node [dot, red] (3) at (-1, -1) {};
	\node  (4) at (-4.5, 2.25) {};
	\node  (5) at (-3, -3) {};
	\node  (6) at (0, -4) {};
	\node  (7) at (4, 2) {};
	\node  (8) at (5, -1) {};
	\node  (9) at (5, -3) {};
	\node  (10) at (-1, 2) {};
	\end{pgfonlayer}
	\begin{pgfonlayer}{background}
	\draw [bend left=45, looseness=1.75] (0.center) to (1.center);
	\draw [bend right=60, looseness=1.25] (2.center) to (3.center);
	\draw [in=165, out=15, looseness=1.25] (4.center) to (2.center);
	\draw [in=60, out=-105, looseness=1.25] (3.center) to (5.center);
	\draw [bend left=15] (3.center) to (6.center);
	\draw [bend left=45] (1.center) to (7.center);
	\draw [in=105, out=-30] (1.center) to (8.center);
	\draw [in=-180, out=-45, looseness=1.25] (1.center) to (9.center);
	\draw [in=135, out=90, looseness=1.50] (2.center) to (10.center);
	\draw [in=315, out=-15] (2.center) to (10.center);
	\draw [bend right=60, looseness=1.50] (0.center) to (1.center);
	\end{pgfonlayer}
	\end{tikzpicture}
};
\draw  [->, line width = 1.2, draw= gray, dashed](-2,-4)--(2,-4);
\node at (0,-4.5){
	$h$
};
\node at (5,-4){ 
	\begin{tikzpicture}[scale = 0.35]	\begin{pgfonlayer}{above}
	\node [dot, green] (0) at (0, 0) {};
	\node  (1) at (0, 0) {};
	\node  (2) at (0, 0) {};
	\node  (3) at (0, 0) {};
	\node  (4) at (-7, 0.25) {};
	\node  (5) at (-3, -4.25) {};
	\node  (6) at (0.25, -5) {};
	\node  (7) at (6.25, 3.25) {};
	\node  (8) at (5.25, -1.75) {};
	\node  (9) at (5, -4) {};
	\end{pgfonlayer}
	\begin{pgfonlayer}{background}
	\draw [bend right=60, looseness=1.25] (2.center) to (3.center);
	\draw [in=-150, out=15, looseness=1.25] (4.center) to (2.center);
	\draw [in=60, out=-105, looseness=1.25] (3.center) to (5.center);
	\draw [in=75, out=-75] (3.center) to (6.center);
	\draw [in=150, out=15] (1.center) to (7.center);
	\draw [in=105, out=-30] (1.center) to (8.center);
	\draw [in=-180, out=-45, looseness=1.25] (1.center) to (9.center);
	\end{pgfonlayer}

	\end{tikzpicture}};

\end{tikzpicture}
\caption{Compatibility of the structure map $\hat h \colon \Tk S \to S$ with the monad multiplication $\mu^{\TTk}S\colon(\Tk)^2 S\rightarrow \Tk S$.}\label{fig. COnu h alg}
\end{figure}

\begin{rmk}\label{rmk. no empty graph}
	The monad in \cite{Koc18} is defined identically to $\TTk$ on $\GS$ except that, in \cite{Koc18}, the empty graph $\oslash$ is not included as a $\nul$-graph. Consequently, algebras for the monad in \cite{Koc18} are equipped with a nonunital external product. See also \cref{rmk. monoidal unit}.
\end{rmk}

\begin{rmk}
The monad $\TTk$ has arities $\Gret \subset \GS$ (in fact, $\TT$ is strongly cartesian) and hence there is an abstract nerve theorem in the style of \cite{BMW12} for $\Set$-valued nonunital circuit algebras. However, this construction is not described here since it is straightforward to derive from the -- more interesting (and challenging) -- nerve theorem for unital circuit algebras, \cref{thm: CO nerve}.
\end{rmk}

\begin{ex}
	\label{ex nonunital WP}

Recall that the oriented graphical species $\DicommE \in \GSE$  defined in \cref{ex. oriented slice} has a (unique) trivial circuit algebra structure and hence, there is a monad $\OTTk$ on the category $\GSE\ov \DicommE$ of oriented graphical species in $\E$ whose EM category of algebras is isomorphic to $ \GSE^{\TTk}\ov \DicommE$. By \cref{prop nonunital CA} and \cref{def wp int}, this is the category $\nuWPE$ of nonunital wheeled props in $\E$. 

The monad $\OTTk$ is easy to describe: Following  \cref{ex OGS on graphs}, the endofunctor $\OTk$ on oriented graphical species in $\E$ is obtained by replacing the categories $X\Grisok$ in the definition of $\Tk$ with the categories of directed graphs with labelled input and output ports and port-preserving morphisms in $\oGret = \Gret \ov \Dicomm$, and applying the equivalence $\GSE\ov \DicommE \simeq \prE{\elG[\Dicomm]}$. 

\end{ex}

\begin{rmk}
	\label{rmk nonunital wheeled props}
	In \cite{Sto23}, a \emph{biased} definition (following \cite{JY15}) of (enriched) wheeled props is given and a groupoid-coloured operad $\mathbb W$ for nonunital wheeled props is described and shown to be Koszul. The wheeled prop definition is essentially an oriented version of the description of enriched circuit algebras in\cite[Proposition~5.4]{RayCA1}. The operad $\mathbb W$ is obtained from the operad of oriented downward wiring diagrams (see \cite[Section~5.1]{RayCA1}) by including isomorphisms in the (groupoid) colouring.
	
	\cref{prop nonunital CA} uses the idea that a graph is obtained by iterated ``contraction'' (or gluing ports) of a disjoint union of corollas. This technique is very similar to the \emph{normal form of a graph} used in \cite[Section~5]{Sto23} to prove that the (enriched) nonunital wheeled prop operad $\mathbb W$ is Koszul.

\end{rmk}

\section{Review of modular operads and pointed graphs}\label{sec. mod op}

In \cite{Ray20}, modular operads in $\Set$ are constructed as algebras for a composite monad $\DD\TT$ on $\GS$. The monads $\DD$ and $\TT$ will be used in \cref{sec. iterated} to construct the circuit algebra monad. Therefore, this section provides a review of their definitions, and the distributive law between them, in the more general setting of $\GSE$ (where, as usual, $\E$ is assumed to have sufficient limits and colimits).

\subsection{The monad $\TT$ and nonunital modular operads}

The nonunital modular operad monad $\TT = (T, \mu^\TT, \eta^\TT)$ on $\GS$ was described in \cite[Section~5]{Ray20}. In this paper, 
the nonunital $\E$-valued modular operad monad on $\GSE$ is also denoted by $\TT = \TT_\E$. 

For all finite sets $X$, let $X\Griso \subset X \Grisok$ be the full subgroupoid of connected $X$-graphs. 
The endofunctor $T\colon\GSE \to \GSE$ is defined, for all $S\in \GSE$, by
\begin{equation}\label{eq. T def}
	\begin{array}{llll}
		\Tk S_\S &= & S_\S,& \\
		\Tk S_X &= & \mathrm{colim}_{\X\in X{\Griso}}  S(\X), & \text{ for all finite sets } X.
	\end{array}
\end{equation}

By \cite[Corollary~5.19]{Ray20} (which follows \cite[Lemma~1.5.12]{Koc16}), if $\G$ is a connected graph, and $\Gg \colon \elG \to \Gr$ is a nondegenerate $\G$-shaped graph of graphs such that $\Gg(b)$ is connected for all $(\C, b) \in \elG$, then the colimit $ \Gg(\G)$ is also connected. 

Hence, $\TT = (T, \mu^\TT, \eta^\TT)$, where $\mu^{\TT}, \eta^{\TT}$ are the obvious (co)restrictions to connected graphs, describes a monad on $\GSE$. 

\begin{prop}\label{prop. Talg}
	The EM category $\GSE^\TT$ of algebras for $\TT$ on $\GSE$ is canonically equivalent to $\nuCSME$.  
\end{prop}
\begin{proof}
This is proved in \cite[Proposition~5.29]{Ray20} for the case $\E = \Set$. The proof for general $\E$ is similar. 

The inclusion $X \Griso \hookrightarrow X \Grisok$ induces a monomorphism $TS_X \to  \Tk S_X$. In particular, since $\X^{x \ddagger y}$ is connected for all connected $X$-graphs $\X$ and all pairs $x, y$ of distinct elements of $X$, if $(A, h)$ is an algebra for $\TT$, then $A$ admits a contraction $\zeta$ given by 
$ \zeta_X^{x \ddagger y} \colon A_X^{x \ddagger y} = A(\C_X^{x \ddagger y}) \xrightarrow{h} TA_{X\setminus \{x,y\}}. $ This satisfies (M2) by the proof of \cref{prop nonunital CA}.

Moreover, 
there is a multiplication $\diamond$ on $A$ such that, for all finite sets $X$ and $Y$ and all elements $x \in X$, and $y \in Y$, $\diamond_{X,Y}^{x \ddagger y}$ is induced by 
$A(\mathcal D_{X,Y}^{x \ddagger y} ) \to TA_{(X \amalg Y) \setminus \{x,y\}} \xrightarrow {h} A_{(X \amalg Y) \setminus \{x,y\}} .$

So, to show that $(A, \diamond, \zeta)$ satisfy the remaining axioms (M1)-(M3), we may construct (as in the proof of \cref{prop nonunital CA} and \cite[Proposition~5.29]{Ray20}), for each axiom, a pair of graphs of graphs that have the same colimit in $\Gr$. 

The converse follows exactly the same method as the proof of \cite[Proposition~5.29]{Ray20} and \cref{prop nonunital CA}.	
\end{proof}

\subsection{The unit monad $\DD$}\label{ssec. D monad}
The modular operad monad on $\GSE$ will be described in terms of a distributive law for composing $\TT$ with the \textit{pointed graphical species monad} $\DD = (D, \mu^{\DD}, \eta^{\DD})$. This was defined on graphical species in $\Set$ in \cite{Ray20}, and is generalised here to a monad, also denoted by $\DD = \DD_\E$, on $\GSE$ for $\E$ with sufficient limits and colimits.

	The endofunctor $D  $ on $\GSE$ is given by $DS_\S = S_\S$, $DS_\tau = S_\tau$ and, for finite sets $X$: 
	\[ DS_X   = \left \{ \begin{array}{ll}
	S_X & X \not \cong \two, \text{ and } X \not \cong \nul,\\
	S_\two \amalg S_\S& X = \two,\\
	S_\nul \amalg \tilde {S_\S} & X = \nul,
	\end{array}  \right . \]
	where $ \tilde {S_\S}$ is the coequaliser of the identity and $S_\tau$ on $S_\S$

There are canonical natural transformations $\eta^\DD \colon 1_{\GSE} \Rightarrow D$, provided by the induced morphism $S \hookrightarrow DS$, and $\mu^\DD  \colon D^2 \Rightarrow D$ induced by the canonical projections $D^2S_\two \to DS_\two$, so that $\DD = (D, \mu^\DD, \eta^\DD)$ defines the pointed graphical species monad on $\GSE$. 

Algebras for $\DD$ -- called \textit{pointed graphical species in $\E$} -- are graphical species $S$ in $\E$ equipped with 
	a distinguished unit-like morphism (\cref{defn: formal connected unit}) $\epsilon \colon S_\S \to  S_\two,$ 
	and a distinguished morphism $o \colon S_\S \to S_\nul $ that factors through $ \tilde {S_\S}$.

Let $\fisinvp$ be category obtained from $\fisinv$ by formally adjoining morphisms
$u \colon  \two \to \S$ and $ z\colon  \nul \to \S$, subject to the relations:
\begin{enumerate}[(i)]
	\item 
	$u \circ ch_1 =  id \in\fisinv (\S, \S)$ and  $ u \circ ch_2= \tau \in \fisinv (\S, \S) $;
	\item $\tau \circ u = u \circ \sigma_{\mathbf 2}\in \fisinv (\two, \S)$ ;
	\item 
	$z = \tau \circ z \in \fisinv (\nul, \S)$.
	\label{relations}
\end{enumerate}

It is straightforward to verify that $\fisinvp$ is completely described by:
\begin{enumerate}[(i)]
	\item $\fisinvp(\S, \S) = \fisinv (\S, \S)$ and $\fisinvp (Y,X)  = \fisinv(Y,X)$ whenever $Y \not \cong \nul$ and $ Y \not \cong \two$;
	\item $\fisinvp (\nul, \S) = \{z\}$, and $\fisinvp (\nul, X) = \fisinv (\nul, X) \amalg \{ ch_x\circ z\}_{x \in X}$;
	\item $\fisinvp (\two, \S) = \{u, \tau \circ u\}$, and $\fisinvp (\two, X) = \fisinv (\nul, X) \amalg \{ ch_x\circ u, ch_x \circ \tau \circ u\}_{x \in X}$. 
\end{enumerate}

It follows that:

\begin{lem}\label{lem: elp presheaves}
The EM category $\GSEp$ of algebras for $\DD$ is the category of $\E$-presheaves on $\fisinvp$. In other words, the following are equivalent:
	\begin{enumerate}
		\item $S_*$ is a presheaf on $ \fisinvp$ that restricts to a graphical species $S$ on $ \fisinv$;
		\item $(S, \epsilon, o)$, with $\epsilon  = S_*(u)$ and $o = S_*(z)$ is a pointed graphical species.
	\end{enumerate}
	\label{CGSp}
\end{lem}

As a consequence of \cref{lem: elp presheaves}, the notation $ S_*$ and $ (S, \epsilon, o)$ will be used interchangeably to denote the same pointed graphical species. 

\begin{rmk}\label{rmk contracted units}
	It is convenient to think of $\epsilon$ as a \textit{unit morphism} for $\S$, and  of $o$ as a \textit{contracted unit morphism} for $S$. Likewise, the monad $D$ may be viewed as adjoining formal units to $S_\two$ and formal contracted units to $S_\nul$ as will now be explained. 
\end{rmk}

\subsection{ Pointed graphs}\label{ss. pointed}

The category $\Grp$ of connected graphs and \textit{pointed \'etale morphisms}, obtained in the bo-ff factorisation of the functor $\Gr \to \GSp$ is described in detail in \cite[Section~7.2]{Ray20}. Here, $\Grp$ is extended to the category $\Gretp$ of \textit{all} graphs and pointed \'etale morphisms. In other words, $\Gretp$ is the category defined by the bo-ff factorisation of the functor $ \Gret \xrightarrow{\yet} \GS \to \GSp$, as in the following commuting diagram of functors:
\begin{equation} \label{defining Grp}
\xymatrix{ 
	\fisinvp \ar[rr]^-{\text{ }}_-{\text{ f.f.}}	\ar@/^2.0pc/[rrrr]_-{\text{ f.f.} }			&& \Gretp \ar [rr]_-{\text{ f.f.} }^-{\yetp}				&& \GSp \ar@<2pt>[d]^-{\text{forget}^{\DD}}\\ 
	\fisinv \ar@{^{(}->} [rr]\ar[u]^{\text{b.o.}}	&& \Gret \ar@{^{(}->} [rr]_-{\yet} \ar[u]^{\text{b.o.}}	&& \GS.  \ar@<2pt>[u]^-{\text{free}^{\DD}}
}\end{equation}

The inclusion $\fisinvp \to \Gretp$ is fully faithful (by uniqueness of bo-ff factorisation), and the induced nerve $\yetp\colon \Gretp \to \GSp$ is fully faithful by construction.

Let $\G\in \Gret$ be a graph. By \cref{ssec. D monad}, for each edge $e \in E$, the $ch_e$-coloured unit for $ \yetp \G$ is given by $ \epsilon^\G_e  =  ch_e \circ u \in \Gretp (\C_\two, \G)$, and the corresponding contracted unit is given by $o^\G_{\tilde e} = ch_e \circ z \in \Gretp (\C_\nul, \G)$.

Since the functor $\yetp$ embeds $\Gretp$ as a full subcategory of $\GSp$, I will write $\G$, rather than $\yetp \G$, where there is no risk of confusion. In particular, the element category $\elpG[\yetp \G]$ is denoted simply by $\elpG$ and called the 
\emph{category of pointed elements of a graph $\G$}.

For all pointed graphical species $S_*$, the forgetful functor $\GSp \to \GS$ induces an inclusion $ \ElS[S]\hookrightarrow \elpG[S_*]$, and, for all graphs $\G$, the inclusion $\elG \hookrightarrow \elpG$ is final (see \cite[Section~IX.3]{Mac98}) by \cite[Lemma~7.8]{Ray20}. Hence,  for all pointed graphical species $S_* = (S, \epsilon, o)$ in $\E$,
\begin{equation}
\label{eq: pointed sheaves}  \mathrm{lim}_{ (\C, b) \in \elpG}S(\C)  = \mathrm{lim}_{ (\C, b) \in \elG}S(\C) = S(\G). \end{equation} 

Moreover, by \cite[Section~7.3]{Ray20}, the induced inclusion $\GSEp \hookrightarrow\prE{\Gretp}$ is fully faithful. 

To better understand  morphisms in $\Gretp$, let $\G$ be a graph and 
let $W \subset V_0 \amalg V_2$ be a subset of isolated and bivalent vertices of $\G$. 

The \emph{vertex deletion functor (for $W$)} is the $ \G$-shaped graph of graphs ${\Gdg}^\G_{\setminus W}\colon  \elG \to \Gret $ given by 
\[{\Gdg}^\G_{\setminus W} (\CX, b)  = \left \{ \begin{array}{ll} 
(\shortmid)& \text{ if } (\CX, b) \text{ is a {neighbourhood} of } v \in W\ (\text{so }  X \cong \nul \text{ or } X \cong \two),\\
\CX &\text{ otherwise. }  
\end{array} \right.  \]   

By \cref{colimit exists}, this has a colimit $\Gnov$ in $\Gret$, and, for all $(\C, b) \in \elG$, ${\Gdg}^\G_{\setminus W}$ describes a morphism $\C \to {\Gdg}^\G_{\setminus W} (b)$ in $\fisinvp$. Hence, ${\Gdg}^\G_{\setminus W}$ induces a morphism $\delW \in \Gretp (\G, \Gnov)$.

\begin{defn}\label{def: vertex deletion}

The morphism $\delW\in \Gretp (\G,\Gnov)$ induced by ${\Gdg}^\G_{\setminus W}\colon  \elG \to \Gret $ is called the  \emph{vertex deletion morphism corresponding to $W \subset V_0 \cong V_2$}. 

\end{defn}

\begin{ex}\label{ex. admissible deletion}

Assume that $\G \not \cong \C_\nul$ is connected (so $V_0$ is empty). Let $W \subset V_2$ be a subset of bivalent vertices of $\G$. Unless $W = V(\G)$ and $\G = \Wm$ for some $m \geq 1$, the graph $\Gnov$ may be intuitively described as \emph{``$\G$ with a subset of bivalent vertices deleted''} as in \cref{fig: vertex deletion}.

In this case, the graph $\Gnov$ is described explicitly by:
\[\ \Gnov = \ \xymatrix{
	*[r] { \Enov}\ar@(ul,dl)[]_{\taunov} && { \Hnov} \ar[ll]_-{\snov} \ar[rr]^-{\tnov}&& \Vnov,}\] where
\[ \begin{array}{ll}\Vnov&=V \setminus W,\\
\Hnov &= H \setminus \left (\coprod_{v \in W} \vH \right),\\
\Enov 
&=  E \setminus\left (\coprod_{v \in W} \vE \right), 
 	\end{array}\]
and $\snov, \tnov$ are just the restrictions of $s$ and $t$. The description of the involution $\taunov$ is more complicated than needed here. (The interested reader is referred to \cite[Section~7.2]{Ray20} which contains full descriptions of the vertex deletion morphisms). 

\begin{figure}[htb!]
	\begin{tikzpicture}[scale = .3]

	\draw[->] (-15,0)--(-11,0);
	\node at (-13,.5){\scriptsize{$\delW$}};

	\node at (-25,0){\begin{tikzpicture}[scale = .25]
		
		\draw [thick, cyan]
		(-3,2).. controls (-2.5,1.5) and (-2.5,0.5).. (-3,0)
		(-3,0).. controls (-3.5,-1).. (-5,-1)
		(-5,3) --(-2,4)--(1,4)--(3,3)
		
		(1,-4)--(4, -3.5)--(5,-1)--(5,2)
		(-1,-5) -- (-3,-5)--(-5, -6)--(-6, -9)
		;
		
		\draw[blue, fill = cyan]
		(-3,0) circle (8pt)
		(4, -3.5) circle (8pt)
		(5,-1)circle (8pt)
		(5,2)circle (8pt)
		(-2,4)circle (8pt)
		(1,4)circle (8pt)
		(-3,-5)circle (8pt)(-5, -6)circle (8pt)(-5.5,-7.5)circle (8pt)
		;

		\draw[ thick] 	
		(-6,1).. controls (-6,2).. (-5,3)
		(-5,3).. controls (-4.5,3) and (-3.5,3).. (-3,2)
		(-5,-1).. controls (-6,-0) .. (-6,1)
		(-5,-1)-- (-3,2)
		(-3,2)--(-6,1)
		(-6,1)..controls ( -7.5,1) and (-9, 0.5) ..(-11,0.5)
		(-1,-5)..controls (-0.5,-3.5) and ( 0.5, -3.5)..(1,-4)
		(-1,-5)..controls (-0.5,-5) and ( 0.5, -5)..(1,-4)
		(4,1)--(5,2)
		(5,2)--(3,3)
		(3,3)--(4,1)
		(3,3).. controls (4,3.5) and (5,2.5)..(5,2)
		(3,3).. controls (3,4) and(4,6)  ..(8,6)
		(-3,2).. controls (1,-1.5) and (5,-1).. (5,2)
		;
		
		\filldraw [black]	
		(-6,1) circle (8pt)
		(-5,3) circle (8pt)
		(-3,2) circle (8pt)
		(-5,-1) circle (8pt)
		(1,-4) circle (8pt)
		(-1,-5) circle (8pt)
		(4,1) circle (8pt)
		(5,2) circle (8pt)
		(3,3) circle (8pt);
		
		\draw[black, fill = white]
		(-4,.5)circle (8pt)
		(4,1) circle (8pt);

		(0,6 ) .. controls (2,6) and (7,4) ..(7,0)
		( 7,0) .. controls (7,-2) and (2,-8) ..(0,-8)--(0,0)--(-7.5,0);
		\end{tikzpicture}};
	
	\node at (-2,0){\begin{tikzpicture}[scale = .25]
		
		\draw [thick, red]
		(-3,2).. controls (-2.5,1.5) and (-2.5,0.5).. (-3,0)
		(-3,0).. controls (-3.5,-1).. (-5,-1)
		(-5,3)..controls (-2,4) and (1,4) .. (3,3)
		
		(1,-4)..controls (2, -4) and (6,-2).. (5,2)
		(-1,-5) .. controls (-2.5,-5) and (-5, -4.5) ..(-6, -9)
		;

		\draw[ thick] 	
		(-6,1).. controls (-6,2).. (-5,3)
		(-5,3).. controls (-4.5,3) and (-3.5,3).. (-3,2)
		(-5,-1).. controls (-6,-0) .. (-6,1)
		(-5,-1)-- (-3,2)
		(-3,2)--(-6,1)
		(-6,1)..controls ( -7.5,1) and (-9, 0.5) ..(-11,0.5)
		(-1,-5)..controls (-0.5,-3.5) and ( 0.5, -3.5)..(1,-4)
		(-1,-5)..controls (-0.5,-5) and ( 0.5, -5)..(1,-4)
		(4,1)--(5,2)
		(5,2)--(3,3)
		(3,3)--(4,1)
		(3,3).. controls (4,3.5) and (5,2.5)..(5,2)
		(3,3).. controls (3,4) and(4,6)  ..(8,6)
		(-3,2).. controls (1,-1.5) and (5,-1).. (5,2)
		;
		
		\filldraw [black]	
		(-6,1) circle (8pt)
		(-5,3) circle (8pt)
		(-3,2) circle (8pt)
		(-5,-1) circle (8pt)
		(1,-4) circle (8pt)
		(-1,-5) circle (8pt)
		(4,1) circle (8pt)
		(5,2) circle (8pt)
		(3,3) circle (8pt);
		
		\draw[black, fill = white]
		(-4,.5)circle (8pt)
		(4,1) circle (8pt);

		(0,6 ) .. controls (2,6) and (7,4) ..(7,0)
		( 7,0) .. controls (7,-2) and (2,-8) ..(0,-8)--(0,0)--(-7.5,0);
		\end{tikzpicture}};

\end{tikzpicture}
\caption{Vertex deletion $\delW\colon  \G \to \Gnov$. } 
\label{fig: vertex deletion}
\end{figure}
\end{ex}
\begin{ex} If $\G = \C_\nul$ is an isolated vertex $*$ and $W = \{*\}$, then $\delW$ is precisely $ z \colon \C_\nul \to (\shortmid). $	
	
\end{ex}

\begin{ex}\label{line deletion}

If $\G = {\Lk}$, and $ W  = V$, then ${\Gdg}^\G_{\setminus W}$ is the constant functor to $(\shortmid)$ and hence $\Gnov = (\shortmid)$. The induced morphism is denoted by $u^k \defeq \delW\colon \Lk \to \Gretp$ and described by the following diagram in $\Gretp$:
	\[
	\xymatrix{ 
		(\shortmid) \ar[rr]^-{ch_1}\ar[drr]_-{id_{(\shortmid)}}  && \C_\two \ar[d]^{u} &&{ \text{ \dots }} \ar[ll]_-{ch_2 \circ \tau} \ar[rr]^-{ch_1}
		&& \C_\two \ar[d]^-{u}&& \ar[ll]_-{ch_2 \circ \tau}\ar[dll]^{id_{(\shortmid)}} (\shortmid)\\
		&&(\shortmid)\ar@{=}[r]&&{ \text{ \dots }} &&\ar@{=}[l](\shortmid).&&}\]
	In particular, $u^1 = u \colon  \C_\two \to (\shortmid)$ and $u^0$ is just the identity on $(\shortmid)$.
\end{ex}

\begin{ex}\label{wheel deletion} 
	Let $\G = \W$ be the wheel graph with one vertex $v$. Then $\W_{/\{v\}} = \mathrm{colim}_{ \elG[\W]} {\Gdg}^\W_{/\{v\}} $ exists and is isomorphic to $(\shortmid)$ in $ \Gret$.  The induced morphism $\kappa \defeq \delW[\{v\}]\colon  \W \to (\shortmid)$ is described in (\ref{epsilon}). Hence, there are precisely two morphisms $\kappa$ and $ \tau \circ \kappa$ in $ \Gretp (\W, \shortmid)$: 
	
	\begin{minipage}[t]{0.5\textwidth}
		\begin{equation} \label{epsilon}
			\xymatrix{
				& \W & \\
				(\shortmid) \ar[ur]^{ch_a} \ar@<-2pt>[rr]_{ch_2\circ \tau}  \ar@<2pt>[rr]^{ ch_{1} }  \ar@{=}[dr] && \C_\two \ar[ul]_{ 1_{\C_\two} \mapsto  a} \ar[dl]^{u}\\
				& (\shortmid)& }\end{equation}
		
	\end{minipage}
	\begin{minipage}[t]{0.5\textwidth}
		\begin{equation} \label{tau epsilon}
			\xymatrix{
				& \W & \\
				(\shortmid) \ar[ur]^{ch_a} \ar@<-2pt>[rr]_{ch_2\circ \tau}  \ar@<2pt>[rr]^{ ch_1 } \ar@{=}[d] && \C_\two \ar[ul]_{ 1_{\C_\two} \mapsto  a} \ar[d]^{\sigma_\two}\\
				(\shortmid) \ar@<-2pt>[rr]_{ch_1\circ \tau}  \ar@<2pt>[rr]^{ ch_2}  \ar@{=}[dr] && \C_\two  \ar[dl]^{u}\\
				& (\shortmid).& }\end{equation}
		
	\end{minipage}
	
	More generally, let $ \Wl$ be the wheel graph with $m$ vertices $(v_i) _{i = 1}^m$, and let $\iota \in \Gr(\Lk[m-1], \Wl)$ be an \'etale inclusion. If $W$ is the image of $ V(\Lk[m-1])$ in $V(\Wl)$, then $ V(\Wl) \cong W \amalg \{*\}$, and by (\ref{epsilon}) and \cref{line deletion}, 
	there are two distinct pointed morphisms, $\kappa^m  $ and $\tau \circ \kappa^m$,  in $ \Gretp (\Wl, \shortmid)$. Hence, for all $\G$, 
	\[\Gretp (\Wl, \G) =  \Gret(\Wl, \G)\amalg \{ch_{e} \circ \kappa^m\}_{ e\in E(\G)} \cong \Gretp(\W, \G).\]  \end{ex}

For convenience, let us define $\Wm[0] \defeq (\shorte) = \Wm_{/V(\Wm)}$ ($m \geq 1$). Then for all $k \geq 0$, $u^k \colon \Lk \to (\shortmid)$ factors canonically as $ \Lk \to \Wm[k] \xrightarrow{\kappa^k} (\shortmid)$.

\begin{rmk}(See \cref{ssec Brauer} and \cite[Section~3]{RayCA1}.)
		In \cite[Remark~3.2]{RayCA1}, it was explained that the composition $\bigcirc = \cap \circ \cup$ of Brauer diagrams is not described by a pushout of cospans as in \cite[Diagram~(3.13)]{RayCA1}. Instead, the pushout for $(\cap \circ \cup)$
		is induced by 
		the trivial diagram $(\shortmid) \xrightarrow{id}(\shortmid) \xleftarrow{id} (\shortmid)$, and therefore agrees with $\kappa \colon \W \to (\shortmid)$ described in \cref{wheel deletion}. (See \cite[Section~6]{Ray20} for more details on the combinatorics of contracted units.)
\end{rmk}

\begin{defn}
	\label{def: similar}
	The \emph{similarity category} $\Gretsimp\hookrightarrow \Gretp$ is the identity-on-objects subcategory of $\Gretp$ whose morphisms are generated under composition by vertex deletion morphisms and graph isomorphisms. Morphisms in $\Gretsimp$ are called \emph{similarity morphisms}, and connected components of $\Gretsimp$ are \emph{similarity classes}. Graphs are \emph{similar} if they are in the same connected component of $\Gretsimp$. 
\end{defn}

The following proposition summarises some important basic properties of vertex deletion morphisms. Detailed proofs may be found in \cite[Section~7.2]{Ray20}.

\begin{prop}
	\label{prop. vertex deletion}
	For all graphs $\G$ and all subsets $W \subset V_2$ of bivalent vertices of $\G$,
	\begin{enumerate}
		\item $\delW \colon \G \to \Gnov$ preserves connected components, and may be defined componentwise;
		\item	if $\G$ is connected and $W = V$, then $\G \cong \Lk$ or $\G \cong \Wm$ for $k, m \geq 0$ and $\Gnov \cong (\shortmid)$;
		\item if $\G$ is connected, then, unless $\G \cong \Wm$ or $\G \cong \C_\nul$ and $W = V$ for $m \geq 1$,  then $\delW$ induces an identity on ports: $\delW \colon E_0 (\G) \xrightarrow{=}E_0 (\Gnov)$.
	
	\end{enumerate} 
The pair $(\Gretsimp, \Gret)$ of subcategories of $\Gretp$ defines an orthogonal 
factorisation system on $\Gretp$.%

\end{prop}

\begin{ex}\label{ex. DS}
	For all graphical species $S$ in $\E$, and all graphs $\G$ with no isolated vertices,
	\begin{equation} \label{eq. D on graphs}
	DS(\G) = \coprod_{W \subset V_2} S(\Gnov).
	\end{equation} 
\end{ex}

More generally, since $DS_\nul  = DS (\C_\nul) = S_\nul \amalg \tilde {S_\S}$ :  \begin{equation}\label{eq. D on graphs general} DS(\G) \cong \mathrm{colim}_{(\H,f) \in \G \ov \Gretsimp} S(\H), \quad  \text{  for all graphs } \G \in \Gret.\end{equation}

\begin{ex}
	\label{wheel and line structure notation} Let $S_*  = (S, \epsilon, o)$ be a pointed graphical species in $\E$. For $k,m \geq 0$, 
	there are distinguished monomorphisms in $\E$: 
	\begin{equation}
	\label{eq. unit graph definitions}
 S_*(u^k)\colon S_\S \to S(\Lk), \ \text{ and } S_*(\kappa ^m)\colon S_\S \to S(\Wm).
	\end{equation}
\end{ex}

\subsection{The distributive law for modular operads}\label{ssec. Similar connected X-graphs}

For the remainder of this section, all graphs will be assumed to be connected, unless stated otherwise. 

 Recall from \cref{def. X graph} that an $X$-graph $\X$ is said to be admissible if it has no stick components.
For any finite set $X$, we may enlarge the category $X\Griso$ of admissible connected $X$-graphs and port-preserving isomorphisms, to form a category $\XGrsimp$ that includes vertex deletion morphisms  morphisms out of $X$-graphs together with any labelling of ports from the domain.

\begin{itemize}
	\item If $X \not \cong \nul$, $X \not \cong \two$, then $ \XGrsimp$ is the category whose objects are connected $X$-graphs and whose morphisms are similarity morphisms 
	that preserve the labelling of the ports.
	
	\item For $X = \two $, $\XGrsimp[\two]$ contains the morphisms $ \delW[V] \colon \Lk \to (\shortmid)$, and hence the labelled stick graphs $(\shortmid, id)$ and $(\shortmid, \tau)$. There are no non-trivial morphisms out of these special graphs, and $\X$ is in the same connected component as $(\shortmid, id)$ if and only if $\X = \Lk$ (with the identity labelling) for some $k \in \N$. In particular, $\tau \colon (\shortmid) \to (\shortmid)$ does not induce a morphism in $ \XGrsimp[\two]$.
	
	\item Finally, when $X = \nul$, the morphisms $\delW[V] \colon \Wm \to (\shortmid)$, and $z \colon \C_\nul \to (\shortmid)$ are not boundary-preserving, and, in particular, do not equip $(\shortmid)$ with any labelling of its ports. So, the objects of $\XGrsimp[\nul]$ are the admissible $\nul$-graphs, and $(\shortmid)$. 
	
	In particular, $\Wm, \C_\nul$ and $(\shortmid)$ are in the same connected component of $\nul\XGrsimp$. Since $(\shortmid)$ is not admissible, there are no non-trivial morphisms in $ \XGrsimp[\nul]$ with $(\shortmid)$ as domain.

\end{itemize}
To simplify notation in what follows, we write ${\C_{\nul}}_{\setminus V} \defeq (\shortmid)$ and $ \delW[V] = z \colon \C_\nul \to (\shortmid)$.

The following lemma is immediate from the definitions:
\begin{lem}\label{lem. terminal X}
	For all finite sets $X$ and all connected $X$-graphs $\X$, $\X^\bot\defeq\X_{ /(V_2\amalg V_0)}$ is terminal in the connected component of $\XGrsimp$ containing $\X$.
	
If $\X \cong \Lk$ or $\X \cong \Wm $, $\X \cong \C_\nul$ for some $k, m \geq 0$, then $\X^\bot \cong (\shortmid)$. Otherwise, $\X^\bot$ is an admissible, connected $X$-graph (without bivalent or isolated vertices).

\end{lem}

Let $S$ be a graphical species in $\E$. 
By (\ref{eq. D on graphs general}), if $\X$ is a connected $X$-graph and $\XGrsimp^\X$ is the connected component of $\X$ in $\XGrsimp$, then there is a canonical isomorphism
\[ DS(\X)  \cong  \mathrm{colim}_{\X' \in \XGrsimp^\X}S(\X').\]

By (\ref{eq. D on graphs}), for $X \not \cong \nul$, 
\[TD S_X= \mathrm{colim}_{ \X \in X\Griso} DS(\X)= \mathrm{colim}_{ \X \in X\Griso} \coprod_{W \subset V_2(\X)} S(\X_{\setminus W}).\]

And, since vertex deletion morphisms preserve ports when $X \neq \nul$, each $\X_{\setminus W} \in X \Griso$ is either an admissible $X$-graph or $\X_{\setminus W} \cong (\shortmid)$, and hence there are canonical morphisms
\[ TD S_X 
\to \left \{ \begin{array}{llll}
\mathrm{colim}_{ \X \in X\Griso}  S(\X)  &= 
TS_X& = DTS_X & \text{ when } X \not \cong \nul, X \not \cong \two,\\
 \left( \mathrm{colim}_{ \X \in X\Griso}  S(\X) \right) \amalg  S_\S &  = 
TS_\two \amalg S_\S &= DTS_\two &\text{ when } X = \two
\end{array} \right. \]
(since $(\shortmid)$ is not a $\two$-graph by convention). These describe the morphisms $ \lambda_{\DD\TT} S_X \colon TDS_X \to DTS_X$ in $\E$ for $X \not \cong \nul$.

When $X = \nul$, 
\[\begin{array}{lll}
TDS_\nul & = &\mathrm{colim}_{ \X \in \nul\Griso} DS(\X)\\
& = & \left(\mathrm{colim}_{\overset{ \X \in \nul \Griso}{\X \not \cong \C_\nul}} \coprod_{W \subset V_2(\X)} S(\X_{\setminus W}) \right) \amalg DS(\C_\nul) \\ 
& = &  \left(\mathrm{colim}_{\overset{ \X \in \nul \Griso}{\X \not \cong \C_\nul}} \coprod_{W \subset V_2(\X)} S(\X_{\setminus W}) \right) \amalg S_\nul \amalg \tilde {S_\S}. 
\end{array}\]
Vertex deletion maps $\delW \colon \G \to \Gnov$ from $\nul$-graphs preserve ports unless $\delW \cong z \colon \C_\nul \to (\shortmid)$ or $\delW \cong \kappa ^m \colon \Wm \to (\shortmid)$. So, $\lambda_{\DD\TT} S _\nul \colon TDS_\nul \to DTS_\nul$ is given by the canonical morphism 
\[ TDS_\nul = \left(\mathrm{colim}_{\overset{ \X \in \nul \Griso}{\X \not \cong \C_\nul}} \coprod_{W \subset V_2(\X)} S(\X_{\setminus W}) \right) \amalg S_\nul \amalg \tilde {S_\S} \to  \mathrm{colim}_{ \X \in \nul \Griso} S(\X) \amalg \tilde {S_\S}  = DTS_\nul.  \]

The verification that $\lambda_{\DD\TT}$ satisfies the four distributive law axioms \cite{Bec69} follows by a straightforward application of the definitions. To help the reader gain familiarity with the constructions, I describe just one here, namely that the following diagram of endofunctors on $\GSE$ commutes:
\begin{equation}\label{eq. dist 1 DT}
\xymatrix{T^2 D \ar@{=>}[rr]^-{ T \lambda_{ \DD \TT}}\ar@{=>}[d]_-{ \mu^{\TT}D} && TDT \ar@{=>}[rr]^-{ \lambda_{ \DD \TT}T}&& DT^2 \ar@{=>}[d]^-{ D \mu^\TT}\\
	TD\ar@{=>}[rrrr]^-{ \lambda_{ \DD \TT}}&&&& DT.}
\end{equation}

Recall, from (\ref{eq. T mult}) and (\ref{eq. D on graphs}), that, for all graphical species $S$ in $\E$ and all finite sets $X$, \[\begin{array}{lll}T^2 D S_X &=&\mathrm{colim}_{ \X \in X \Griso} \mathrm{lim}_{(\C,b) \in \elG[\X]} \mathrm{colim}_{ \Y^b \in [b]\Griso} \coprod_{W^b \subset V_2}  S(\Y^b_{\setminus W^b})\\ 
&= &\mathrm{colim}_{ \X \in X \Griso} \mathrm{lim}_{ \Gg \in \GrG[\X]} \coprod_{W^\Gg \subset V_2(\Gg(\X))} S (\Gg (\X)_{\setminus W^{\Gg}}),
\end{array}\]
where, as usual, $\esv \colon \Cv \to \X$ is the canonical morphism. So, to describe the two paths in (\ref{eq. dist 1 DT}), 
 let $\X$ be a connected $X$-graph, $\Gg$ a nondegenerate $\X$-shaped graph of connected graphs, and let $W^\Gg \subset (V_0 \amalg V_2)(\Gg(\X))$ be a subset of (bivalent or isolated) vertices of the colimit $\Gg(\X)$ of $\Gg$ in $\Gr$. 
 
 For each $(\C, b) \in \elG[\X]$, let 
 \[ W^b \defeq V(\Gg(b)) \cap W^\Gg, \ \text{ so } W^{\Gg}  \cong \coprod_{v \in V(\X)} W^{\esv}.\] 
 
 If $(\X, \Gg, W) = (\C_\nul, \Gid[\C_{\nul}], V(\C_\nul))$, then $\Gg(\X) = \C_\nul$, and $(\C_\nul, b)$ is the unique object of $\elG[\X]$. 
 
  Otherwise, since $\X$ is connected, it must be the case that $W^{\Gg} \subset V_2 (\Gg(\X))$.

In each case we may define $W \subset V(\X)$ to be the set of vertices $v$ of $\X$ such that $W^{\esv} = V(\Gg(\esv))$. 
Then, 
there is a unique nondegenerate $\X_{\setminus W}$-shaped graph of connected graphs $\Gg_{\setminus W}$ given by $\Gg_{\setminus W} (\delW \circ b) = \Gg(b)_{\setminus W^b}$, and with colimit $\Gg_{\setminus W} (\X_{\setminus W})$ in $\Gretp$:
\begin{equation}\label{eq. colimits agree}
\Gg_{\setminus W} (\X_{\setminus W}) = \Gg(\X)_{\setminus W^\Gg}.
\end{equation}

	\begin{figure}[htb!]
		\begin{tikzpicture} [scale = .85]
		\node at(-6,0){ 
			\begin{tikzpicture}[scale = .35]
			\begin{pgfonlayer}{above}
			\node [dot, red] (0) at (0, 0) {};
			\node [dot, red] (1) at (1.5, -1) {};
			\node [dot, red] (2) at (0, -1) {};
			\node [dot, red] (3) at (-3, 0) {};
			\node [dot, red] (4) at (-2.5, -1) {};
			\node [dot, red] (5) at (-4, -1) {};
			\node [dot, green] (6) at (-1.5, -2.5) {};
			\node [dot, green] (7) at (-0.5, -3) {};
			\node [dot, green] (8) at (1, -3) {};
			\node [dot, green] (9) at (0, 2) {};
			\node  (10) at (-6, -4) {};
			\node  (11) at (3, -5) {};
			\node  (12) at (3, 4) {};
			\draw[draw = blue, fill = blue, fill opacity = .2] (.65,-.5) circle (1.6cm);
			\draw[draw = blue, fill = blue, fill opacity = .2] (-3.25,-.5) circle (1.6cm);
			\draw[draw = green, fill = green, fill opacity = .2] (0,2) circle (.7cm);
			\draw[rotate around={80:(-.5,-2.8)}, draw = green, fill = green, fill opacity = .2] (-.5,-2.8) ellipse (.7cm and 2cm);
			\end{pgfonlayer}
			\begin{pgfonlayer}{background}
			\draw [bend right, looseness=1.25] (0) to (2);
			\draw [bend right=60, looseness=1.25] (2) to (1);
			\draw [bend right] (0) to (1);
			\draw [in=165, out=-120, loop] (0) to ();
			\draw [bend right] (3) to (5);
			\draw [bend right=45, looseness=1.25] (5) to (4);
			\draw [in=120, out=-165, loop] (4) to ();
			\draw [in=135, out=60] (3) to (0);
			\draw [bend left] (3) to (9);
			\draw [bend left=60, looseness=0.75] (0) to (1);
			\draw [in=150, out=-60] (4) to (6);
			\draw [bend left=15] (6) to (7);
			\draw [bend right=15, looseness=1.25] (7) to (8);
			\draw [in=-75, out=30] (9) to (12.center);
			\draw [in=135, out=0, looseness=1.25] (8) to (11.center);
			\draw (5) to (10.center);
			\end{pgfonlayer}
			\end{tikzpicture}
			
		};
		
		\node at(0,0){ 
			\begin{tikzpicture}[scale = .35]
			\begin{pgfonlayer}{above}
			\node [dot, red] (0) at (0, 0) {};
			\node [dot, red] (1) at (1.5, -1) {};
			\node [dot, red] (2) at (0, -1) {};
			\node [dot, red] (3) at (-3, 0) {};
			\node [dot, red] (4) at (-2.5, -1) {};
			\node [dot, red] (5) at (-4, -1) {};
			\node (6) at (-1.5, -2.5) {};
			\node  (7) at (-0.5, -3) {};
			\node  (8) at (1, -3) {};
			\node  (9) at (0, 2) {};
			\node  (10) at (-6, -4) {};
			\node  (11) at (3, -5) {};
			\node  (12) at (3, 4) {};
			\draw[draw = blue, fill = blue, fill opacity = .2] (.65,-.5) circle (1.6cm);
			\draw[draw = blue, fill = blue, fill opacity = .2] (-3.25,-.5) circle (1.6cm);
			\draw[draw = green, fill = green, fill opacity = .8] (0,2) circle (.7cm);
				\draw[draw = green, fill = green, fill opacity = .8] (7) circle (.7cm);
			\end{pgfonlayer}
			\begin{pgfonlayer}{background}
			\draw [bend right, looseness=1.25] (0) to (2);
			\draw [bend right=60, looseness=1.25] (2) to (1);
			\draw [bend right] (0) to (1);
			\draw [in=165, out=-120, loop] (0) to ();
			\draw [bend right] (3) to (5);
			\draw [bend right=45, looseness=1.25] (5) to (4);
			\draw [in=120, out=-165, loop] (4) to ();
			\draw [in=135, out=60] (3) to (0);
			\draw [bend left] (3) to (9.center);
			\draw [bend left=60, looseness=0.75] (0) to (1);
			\draw [in=150, out=-60] (4) to (6.center);
			\draw [bend left=15] (6.center) to (7.center);
			\draw [bend right=15, looseness=1.25] (7.center) to (8.center);
			\draw [in=-75, out=30] (9.center) to (12.center);
			\draw [in=135, out=0, looseness=1.25] (8.center) to (11.center);
			\draw (5) to (10.center);
			
			\end{pgfonlayer}
			\end{tikzpicture}
			
		};
		\node at(6,0){
			\begin{tikzpicture}[scale = .35]
			\begin{pgfonlayer}{above}
			\node [dot, red] (0) at (0, 0) {};
			\node [dot, red] (1) at (1.5, -1) {};
			\node [dot, red] (2) at (0, -1) {};
			\node [dot, red] (3) at (-3, 0) {};
			\node [dot, red] (4) at (-2.5, -1) {};
			\node [dot, red] (5) at (-4, -1) {};
			\node (6) at (-1.5, -2.5) {};
			\node  (7) at (-0.5, -3) {};
			\node  (8) at (1, -3) {};
			\node  (9) at (0, 2) {};
			\node  (10) at (-6, -4) {};
			\node  (11) at (3, -5) {};
			\node  (12) at (3, 4) {};
			\draw[draw = blue, fill = blue, fill opacity = .2] (.65,-.5) circle (1.6cm);
			\draw[draw = blue, fill = blue, fill opacity = .2] (-3.25,-.5) circle (1.6cm);
			
			\end{pgfonlayer}
			\begin{pgfonlayer}{background}
			\draw [bend right, looseness=1.25] (0) to (2);
			\draw [bend right=60, looseness=1.25] (2) to (1);
			\draw [bend right] (0) to (1);
			\draw [in=165, out=-120, loop] (0) to ();
			\draw [bend right] (3) to (5);
			\draw [bend right=45, looseness=1.25] (5) to (4);
			\draw [in=120, out=-165, loop] (4) to ();
			\draw [in=135, out=60] (3) to (0);
			\draw [bend left] (3) to (9.center);
			\draw [bend left=60, looseness=0.75] (0) to (1);
			\draw [in=150, out=-60] (4) to (6.center);
			\draw [bend left=15] (6.center) to (7.center);
			\draw [bend right=15, looseness=1.25] (7.center) to (8.center);
			\draw [in=-75, out=30] (9.center) to (12.center);
			\draw [in=135, out=0, looseness=1.25] (8.center) to (11.center);
			\draw (5) to (10.center);
			
			\end{pgfonlayer}
			\end{tikzpicture}
			
		};
		
		\draw  [->, line width = 1.22, draw= gray, dashed](-4,0)--(-2,0);
		\node at (-3,.5){
			$T\lambda_{ \DD \TT}$
		};
		\draw  [->, line width = 1.22, draw= gray, dashed](2,0)--(4,0);
		\node at (3,.5){
			$\lambda_{ \DD \TT}T$
		};
		\draw  [->,dashed,  line width = 1.2, draw= gray](-5,-2)--(-5,-3);
		\node at (-5.8,-2.5){$D\mu^{\TT}S$
		};

		\draw  [->,dashed,  line width = 1.2, draw= gray](5,-2)--(5,-3);
		
		\node at (5.8,-2.5){
			$D\mu^{\TT}$
		};
		
		\node at (-6,-4.5){ 
			\begin{tikzpicture}[scale = .35]
			\begin{pgfonlayer}{above}
			\node [dot, red] (0) at (0, 0) {};
			\node [dot, red] (1) at (1.5, -1) {};
			\node [dot, red] (2) at (0, -1) {};
			\node [dot, red] (3) at (-3, 0) {};
			\node [dot, red] (4) at (-2.5, -1) {};
			\node [dot, red] (5) at (-4, -1) {};
			\node [dot, green] (6) at (-1.5, -2.5) {};
			\node [dot, green] (7) at (-0.5, -3) {};
			\node [dot, green] (8) at (1, -3) {};
			\node [dot, green] (9) at (0, 2) {};
			\node  (10) at (-6, -4) {};
			\node  (11) at (3, -5) {};
			\node  (12) at (3, 4) {};
			
			\end{pgfonlayer}
			\begin{pgfonlayer}{background}
			\draw [bend right, looseness=1.25] (0) to (2);
			\draw [bend right=60, looseness=1.25] (2) to (1);
			\draw [bend right] (0) to (1);
			\draw [in=165, out=-120, loop] (0) to ();
			\draw [bend right] (3) to (5);
			\draw [bend right=45, looseness=1.25] (5) to (4);
			\draw [in=120, out=-165, loop] (4) to ();
			\draw [in=135, out=60] (3) to (0);
			\draw [bend left] (3) to (9);
			\draw [bend left=60, looseness=0.75] (0) to (1);
			\draw [in=150, out=-60] (4) to (6);
			\draw [bend left=15] (6) to (7);
			\draw [bend right=15, looseness=1.25] (7) to (8);
			\draw [in=-75, out=30] (9) to (12.center);
			\draw [in=135, out=0, looseness=1.25] (8) to (11.center);
			\draw (5) to (10.center);
			\end{pgfonlayer}
			\end{tikzpicture}
			
		};
		\draw  [->, line width = 1.2, draw= gray, dashed](-2.5,-4)--(2.5,-4);
		\node at (0,-4.5){
			$\lambda_{ \DD \TT}$
		};
		\node at (6,-4){ 
			\begin{tikzpicture}[scale = .35]
			\begin{pgfonlayer}{above}
			\node [dot, red] (0) at (0, 0) {};
			\node [dot, red] (1) at (1.5, -1) {};
			\node [dot, red] (2) at (0, -1) {};
			\node [dot, red] (3) at (-3, 0) {};
			\node [dot, red] (4) at (-2.5, -1) {};
			\node [dot, red] (5) at (-4, -1) {};
			\node (6) at (-1.5, -2.5) {};
			\node  (7) at (-0.5, -3) {};
			\node  (8) at (1, -3) {};
			\node  (9) at (0, 2) {};
			\node  (10) at (-6, -4) {};
			\node  (11) at (3, -5) {};
			\node  (12) at (3, 4) {};
		
			\end{pgfonlayer}
			\begin{pgfonlayer}{background}
			\draw [bend right, looseness=1.25] (0) to (2);
			\draw [bend right=60, looseness=1.25] (2) to (1);
			\draw [bend right] (0) to (1);
			\draw [in=165, out=-120, loop] (0) to ();
			\draw [bend right] (3) to (5);
			\draw [bend right=45, looseness=1.25] (5) to (4);
			\draw [in=120, out=-165, loop] (4) to ();
			\draw [in=135, out=60] (3) to (0);
			\draw [bend left] (3) to (9.center);
			\draw [bend left=60, looseness=0.75] (0) to (1);
			\draw [in=150, out=-60] (4) to (6.center);
			\draw [bend left=15] (6.center) to (7.center);
			\draw [bend right=15, looseness=1.25] (7.center) to (8.center);
			\draw [in=-75, out=30] (9.center) to (12.center);
			\draw [in=135, out=0, looseness=1.25] (8.center) to (11.center);
			\draw (5) to (10.center);
			
			\end{pgfonlayer}
			\end{tikzpicture}
			
		};

	\end{tikzpicture}

	\caption{Diagram (\ref{eq. dist 1 DT}) commutes. Here the vertices in $W$ are marked in green.}
	\label{fig. DTdist}
\end{figure}

Then, the two paths (illustrated in \cref{fig. DTdist}) described by the diagram (\ref{eq. dist 1 DT}) are induced by the following canonical maps (where, in each case, the vertical arrows are the defining universal morphisms):\\

{	\bf{Top-right}}
		\[\xymatrix @R =.25cm@C = .75cm{
			&(\X, (\Gg(b)_{\setminus W^b})_{b}, W) \ar@{..>}[dd] \ar@{..>}[rr]^-{\tiny{\text{ delete } W \text{ in } \X}} && (\X_{\setminus W}, \Gg_{\setminus W}) \ar@{..>}[dd] \ar@/^1pc/@{..>}[rd]^-{\tiny{\text{evaluate colimit } \Gg_{\setminus W} }}&\\
			(\X, (\Gg(b))_{b}, (W^b)_b) \ar@{..>}[dd]\ar@/^1pc/@{..>}[ur]^-{\tiny{ \text{ delete } W^b \text{ in } \Gg(b)}}&&&&\Gg_{\setminus W}(\X_{\setminus W}) \ar@{..>}[dd] \\
	&\mathrm{lim}_{\elG[\X]}S(\Gg(b)_{\setminus W^b})\ar[rr]\ar[dd]&&S(\Gg_{\setminus W}(\X_{\setminus W})\ar@/^1pc/[dr]\ar[dd]&\\
	\mathrm{lim}_{\elG[\X]}	S(\Gg(b)_{\setminus W^b})\ar@/^1pc/[ur]\ar[dd]	&&&&S(\Gg_{\setminus W}(\X_{\setminus W}))\ar[dd] \\
&TDTS_X \ar[rr]_-{\lambda_{\DD\TT} T S_X}	&& DT^2 S_X \ar@/^1pc/[rd]_-{ D \mu^\TT S_X }&\\
T^2 DS_X \ar@/^1pc/[ur]_-{T\lambda_{\DD\TT}S_X}&&&& DTS_X }
\]

{\bf{Left-bottom}}
	\[	\xymatrix{(\X, \Gg, W^{\Gg}) \ar@{..>}[d]\ar@{..>}[rrr]^-{\tiny{\text{ evaluate colimit } \Gg}}&&&(\Gg(\X), W^{\Gg})\ar@{..>}[d] \ar@{..>}[rrr]^-{\tiny{\text{ delete } W^{\Gg} \text{ in } \Gg(\X)}}&&&\Gg(\X)_{\setminus W^{\Gg}}\ar@{..>}[d]\\ 
	S(\Gg(\X)_{\setminus W^{\Gg}})\ar@{=}[rrr]\ar[d]&&& S(\Gg(\X)_{\setminus W^{\Gg}})\ar@{=}[rrr] \ar[d]&&&S(\Gg(\X)_{\setminus W^{\Gg}})\ar[d]\\
		T^2 DS_X \ar[rrr]_-{\mu^{\TT}DS_X}&&&TDS_X \ar[rrr]_-{\lambda_{\DD\TT}  S_X}	&&& DT S_X.}\]

These are equal since $S(\Gg_{\setminus W}(\X_{\setminus W})) = S(\Gg(\X)_{\setminus W^{\Gg}})$ by construction, and hence (\ref{eq. dist 1 DT}) commutes.

The proofs that $\lambda_{\DD\TT}$ satisfies the remaining three distributive law axioms follow similarly (and are somewhat simpler than the proof that (\ref{eq. dist 1 DT}) commutes.

\begin{prop}
	\label{prop. extension exists}
	The monad $\TT$ on $\GSE$ extends to a monad $\TTp= (\Tp, \mu^{\TTp}, \eta^{\TTp})$on $\GSEp$ such that $\GSEp^{\TTp}\cong \GSE^{\DD\TT}$.
\end{prop}
\begin{proof}
	For $\E = \Set$, $\TTp$ has been described in detail in \cite[Section~7]{Ray20}. 
	
	For graphical species in a general category $\E$ with sufficient (co)limits, the endofunctor $\Tp$ is the quotient of $T$ given, for all pointed graphical species $S_* = (S, \epsilon, o)$ in $\E$, by
	\begin{equation}\label{eq. Tp} \Tp S_\S = S_\S, \text { and } \Tp S_X = \mathrm{colim}_{\X \in \XGrsimp} S(\X).  \end{equation}
	The unit $\epsilon ^{\Tp S}\defeq \Tp (\epsilon)  \colon S_\S \to \Tp S_\two$ is described by the obvious composite 
	\[  \xymatrix{S_\S \ar[r]^-{\epsilon} &S_\two\ar[r]^- {\eta^\TT S}& TS_\two \ar@{->>}[r]& \Tp S_\two,}  \] and the contracted unit $o ^{\Tp S}\defeq \Tp (o)  \colon S_\S \to \Tp S_\nul$ is described by
	\[  \xymatrix{S_\S \ar[r]^-{o}& S_\nul \ar[r]^-{\eta^\TT S} & TS_\nul \ar@{->>}[r]& \Tp S_\nul.} \]
	In particular, $\epsilon^{\Tp S}$ is defined by maps $S(u^k) \colon S_\S \to S(\Lk)$ (for $k \geq 1$), and $o^{\Tp S}$ by maps $S(\kappa^m)\colon S_\S \to S(\Wm)$, $m \geq 1$ and $S( o) \colon S_\S \to S(\C_\nul)$ such that $o \circ S_\tau = o$.

	The multiplication $\mu^{\TTp}$ for $\TTp$ is the one induced by $\mu^\TT$, that forgets the pair $(\X, \Gg)$ -- of an $X$-graph $\X$, and a nondegenerate $\X$-shaped graph of connected graphs $\Gg$ -- and replaces it with the colimit $\Gg(\X)$. The unit $\etap$ for $\TTp$ is defined by the maps $S_X =  S(\C_X) \mapsto \mathrm{colim}_{ \X \in \XGrsimp} S(\X)$.
	
	Let $(A_*, h_*)$ be a $\TTp$ algebra in $\GSEp$ such that $A_* = (A, \epsilon, o)$. The unique map $h \colon DTA \to A$ that preserves distinguished (contracted) units and factors through the quotient $TA \twoheadrightarrow \Tp A$ on $TA$ describes a $\DD\TT$ algebra structure on $ A$.
	
	Conversely, if $(A, h)$ is a $\DD\TT$ algebra, then in particular, $A$ has a $\DD$-algebra structure $(A, \epsilon, o)$ and $h$, together with $\lambda_{\DD\TT}$, induce a $\TTp$ algebra structure on $(A, \epsilon, o)$. (See the proof of \cite[Proposition~7.39]{Ray20}). 
	
	It is straightforward (though somewhat laborious) to use the definition of $\lambda_{\DD\TT}$ -- as in the proof of \cite[Proposition~7.39]{Ray20} -- to check that these assignments extend to inverse functors $\GSEp^{\TTp}\leftrightarrows \GSE^{\DD\TT}$.
\end{proof}


It is then straightforward to derive the following theorem (that, in the case $\E = \Set$, was the main result, Theorem~7.46, of \cite{Ray20}) for general $\E$.
\begin{thm}\label{thm CSM monad DT}
	The EM category $\GSE^{\DD\TT}$ of algebras for $\DD\TT$ is canonically isomorphic to $\CSME$.
\end{thm}
\begin{proof}
	Let $(A,h)$ be a $\DD \TT$-algebra in $\E$. Then $h$ induces a $\TT$-algebra structure $h^\TT \colon TA \to A$ by restriction, and a $\DD$-algebra structure $h^\DD = h \circ D \eta^{\TT}A \colon DA \to A$. In particular, $h^\TT$ equips $A$ with the a multiplication $\diamond$ and contraction $\zeta$ such that $(A, \diamond , \zeta)$ is a nonunital modular operad in $\E$, and $h^\DD$, equips $A$ with the structure of a pointed graphical species $A_* = (A, \epsilon, o) $ in $\E$. 
	
	To prove that $(A, h)$ describes a modular operad, it remains to show that $\epsilon$ is a unit for $\diamond.$ But, this is immediate from the definition of the distributive law $\lambda_{\DD\TT}$ or, equivalently, from \cref{prop. extension exists} and (\ref{eq. Tp}).
	

Conversely, if $(S, \diamond, \zeta, \epsilon)$ is a modular operad in $\E$, then $(S, \diamond, \zeta)$ defines a $\TT$-algebra $p^\TT \colon TS \to S$ by \cref{prop. Talg}, and $(S, \epsilon, \zeta \epsilon)$ defines a $\DD$-algebra by \cref{ssec. D monad}. In particular, since $\epsilon$ is a unit for $\diamond$, and $p^\TT$ is induced by iterations of $\diamond$ and $\zeta$, $p^\TT$ induces a $\TTp$-algebra structure on $(S, \epsilon, \zeta\epsilon)$ by (\ref{eq. Tp}). Hence, $(S, \diamond, \zeta, \epsilon)$ describes an algebra for $\DD\TT$.
\end{proof}

\begin{ex}\label{ex Wheeled properads}
(See also Examples~\ref{ex directed graphical species},\ref{ex. wheeled prop CO} and \ref{ex nonunital WP}, and \cref{def wp int}.) The results of this section may be modified for oriented graphical species (see \cref{defn oriented GSE}).  By restricting to connected directed graphs, we may obtain an oriented version of $\TT$ as in \cref{ex nonunital WP}. 
Since there are no orientation preserving automorphisms of the directed exceptional edge $\downarrow$ in $\oGret$, the endofunctor $\OD$ underlying the monad $\ODD$ on $\oGSE $ such that $ \oGSE^{\ODD} \simeq \GSE^{\DD} \ov \DicommE$ is defined, for all oriented graphical species $\tilde S  \colon {\elG[\Dicomm]}^{\mathrm{op}} \to \E$, by $\OD \tilde S(\downarrow)= \tilde S(\downarrow)$ and 
	\[ \mathsf O D\tilde S(X_{\In},X_{\Out})   = \left \{ \begin{array}{ll}
	\tilde S(X_{\In},X_{\Out}) & \text{ for } ({X_{\In},X_{\Out}}) \not \cong ({\nul, \nul}), \text{ and } ({X_{\In},X_{\Out}})\not \cong (\mathbf {1}, \mathbf {1}),\\
	\tilde S (\mathbf {1}, \mathbf {1}) \amalg  \tilde S(\downarrow)&\text{ for } ({X_{\In},X_{\Out}})\cong (\mathbf {1}, \mathbf {1}),\\
\tilde S (\mathbf {0}, \mathbf {0}) \amalg  \tilde S(\downarrow)&\text{ for } ({X_{\In},X_{\Out}} )\cong (\mathbf {0}, \mathbf {0})\\
\end{array}  \right . \]
 A category $\oGretp$ of pointed directed graphs and a distributive law are obtained exactly as for the undirected case. This is equipped with distinguished morphisms $\hat z \colon \C_{\mathbf 0,\mathbf 0} \to (\oedge)$ and $\hat u \colon \C_{\mathbf 1,\mathbf 1} \to (\oedge)$ satisfying
$\hat u \circ i_1 = \hat u \circ o_1  = id_{\oedge}$. 

Algebras for the composite monad on $\oGSE$ so-obtained are wheeled properads in $\E$. 
\end{ex}

\begin{rmk}
	\label{rmk no tensor on MO}
	By \cref{rmk monoidal question}, circuit algebras are, in particular monoid objects in $\GSE$. However, they do not describe monoids in the category $\CSME$ of modular operads since $\CSME \simeq \GSE^{\DD\TT}$ is not closed under the monoidal product $ \otimes$ in $\GSE$: For, given $S, S' \in \GSE$, there is, in general, no map $DT(S \otimes S') \to DT(S)\otimes DT(S')$. And hence, if $(A, h), (A', h')$ are algebras for $\DD\TT$ it is not possible to construct a map $DT(A \otimes A') \to A \otimes A'$.
	
	To see this, consider an element of $T(A \otimes A') _{\two} \subset  T(A \otimes A') _{\two}$ represented by an $A \otimes A'$ structure on $\Lk[2]$, where decoration by $A$ is coloured in filled red, and decoration by $A'$ is unfilled blue: 
	\[ \begin{tikzpicture}[scale = .8]
		\draw (1,0) -- (5,0);
		\filldraw[white] (3,0) circle (22pt);
		
		\draw[red] (1,0)--(2.7,0);
		
		\draw[blue] (3.3,0)--(5,0);
		
		\filldraw[red] (2.7,0) circle (4pt);
	\filldraw[white]	(3.3,0) circle (4pt);
		\draw[blue, ultra thick] (3.3,0) circle (4pt);

		\draw [thick](3,0) circle (22pt);

	\end{tikzpicture}\]
	Since the $TA$ and $TA'$ components are both connected, this represents an element of $TA_{\one} \otimes TA'_{\one}$ (``delete'' the black circle) and hence, applying $h, h'$, we obtain an element of $A \otimes A'$. 
	
	However, in the following $A \otimes A'$ structure on $\Lk[3]$, the $TA$ (red, filled) and $TA'$ (blue, unfilled) components are no longer connected: 
	\[ \begin{tikzpicture}[scale = .8]
		\draw (1,0) -- (11,0);
		\filldraw[white] (3,0) circle (22pt);
		\filldraw[white] (6,0) circle (22pt);
		\filldraw	[white]	(9,0) circle (22pt);
		
		\draw[red] (1,0)--(2.7,0)
		(6.3,0)--(8.7,0);
		\draw[blue] (3.3,0)--(5.7,0)
		(9.3,0)--(11,0);
		\filldraw[red] (2.7,0) circle (4pt);
		\filldraw[white] (3.3,0) circle (4pt);
		\filldraw[white] (5.7,0) circle (4pt);
		\draw[blue, ultra thick] (3.3,0) circle (4pt);
		\draw[blue, ultra thick] (5.7,0) circle (4pt);
		\filldraw[red] (6.3,0) circle (4pt);
		\filldraw[red] (8.7,0) circle (4pt);
			\filldraw[white] (9.3,0) circle (4pt);
		\draw[blue, ultra thick] (9.3,0) circle (4pt);

		\draw [thick](3,0) circle (22pt);
		\draw[thick](6,0) circle (22pt);
		\draw[ thick]	(9,0) circle (22pt);
		
	\end{tikzpicture}\]
	Hence this does not represent an element of $(DTA \otimes DTA')_{\two}$ and so $\GSE^{\DD\TT}\simeq \CSME$ is not closed under the monoidal product $\otimes$ on $\GSE$. 
	
\end{rmk}

\section{A monad for circuit algebras}  \label{sec. monad CO}

Despite the similarity between the definitions of the monads $\TT$ and $\TTk$ on $\GSE$, it is not possible to modify $\lambda_{\DD\TT}$ by replacing $\TT$ with $\TTk$ to obtain a distributive law for $\DD$ and $\TTk$: 

Namely, for all graphical species $S$, since $DS_\two\cong S_\two \amalg S_\S$ by definition, there is a canonical morphism \[S(\shortmid \amalg \shortmid) = S_\S \times S_\S   \to DS_\two \times DS_\two =  DS(\C_\two \amalg \C_\two) \to TDS_{ \two \amalg \two}.  \] 
However, there is, in general, no morphism
\[ (S_\two \amalg S_\S)\times (S_\two \amalg S_\S) \to T S_{ \two \amalg \two}  = DTS_{\two \amalg \two}, \] and therefore $\lambda_{ \DD \TT} \colon TD \Rightarrow DT$ does not extend to a distributive law $\Tk D \Rightarrow D \Tk$. 


The problem of finding the circuit algebra monad on $\GSE$ is easily solved by observing that the monad $\TTk$ is, itself a composite: the free graded monoid monad on $\pr{\Sigma}$ (see \cref{ssec CO}) induces a monad $\LL = (L, \mu^\LL, \eta^\LL)$ on $\GSE$ 
whose algebras are graphical species in $\E$ that are equipped with an external product. Moreover, there is a distributive law $\lambda_{ \LL \TT} \colon TL \Rightarrow LT$ such that $\TTk = \LL \TT$ is the induced composite monad. 

It will then be straightforward to prove (\cref{thm. iterated law}) that, for all (unpointed) graphical species $S$ in $\E$, $LDTS$ describes a circuit algebra in $\E$. 

By \cite{Che11}, there is a distributive law $\lambda_{ \LL(\DD\TT)} \colon (DT)L \Rightarrow L(DT)$, and hence a composite monad $\LL \DD \TT$ on $\GSE$, if there are pairwise distributive laws $\lambda_{ \LL \TT} \colon TL \Rightarrow LT$, and $\lambda_{ \LL\DD} \colon DL \Rightarrow LD$, such that the following \textit{Yang-Baxter diagram} of endofunctors on $\GSE$ commutes:
	\begin{equation}\label{YB}
	\xymatrix{
		&& DTL \ar[rr]^-{D\lambda_{\LL\TT}} && DLT \ar[drr]^-{\lambda_{\LL\DD}T}&& \\
		TDL \ar[urr]^-{\lambda_{\DD\TT}L}\ar[drr]_-{T\lambda_{\LL\DD}}&&  && && LDT.\\
		&& TLD \ar[rr]_-{\lambda_{\LL\TT}D}&& LTD \ar[urr]_-{L\lambda_{\DD\TT}}&& }
	\end{equation}

In \cref{ssec. LL}, the monad $\LL$ is defined and the distributive laws $\lambda_{\LL\TT}: TL \Rightarrow LT$ and $\lambda_{\LL \DD}: DL \Rightarrow LD $ are constructed. In \cref{sec. iterated}, it is proved that the diagram (\ref{YB}) commutes, and the EM category of algebras for the composite monad $\LL\DD \TT$ on $\GSE$ is canonically equivalent to the category $\COE$ of circuit algebras in $\E$.

\subsection{The free external product monad $\LL$}\label{ssec. LL}

Recall  from \cref{ssec. presheaves} that, for any finite set $X$, $\core{X \ov \fin}$ is the maximal subgroupoid of $X \ov \fin$ whose objects are pairs $(Y, f)$ with $f \in \fin (X, Y)$ a morphism of finite sets. 

As usual, let $\E$ be a category with sufficient (co)limits. The free graded monoid endofunctor $L$ on $\GSE$ is given by 
\begin{equation}\label{eq. defn L}\begin{array}{lll}
LS_\S &= & S_\S \\ 
LS _X  &= & \mathrm{colim}_{(Y,f) \in \core{X \ov \fin}} \prod_{y \in Y} S_{  f^{-1}(y)}
\end{array}\end{equation}
for all finite sets $X$. This underlies a monad $\LL  = (L, \mu^\LL, \eta^\LL)$  with monadic multiplication $\mu^\LL$ induced by disjoint union of partitions, and monadic unit $\eta^\LL$ induced by inclusion of trivial partitions.

The following lemma is immediate from the definitions (see also \cref{lem external monoid}):
\begin{lem}\label{lem. L free monoid}
The EM category of algebras for the monad $\LL$ on $\GSE$ is equivalent to the category of graphical species with unital external product in $\E$. In particular $\GSE^{\LL}$ is equivalent to the monoid subcategory of  $\GSE$.
	\end{lem}

In order to describe a distributive law $\lambda_{\LL \TT}\colon TL \Rightarrow$ for $\LL$ and $\TT$, observe first that $L \subset \Tk$ is a subfunctor since, if $S$ is a graphical species in $\E$, and $X$ a finite set, then 
\[ \begin{array}{lll}
LS_X  & =  & \mathrm{colim}_{(Y,f) \in \core{X \ov \fin}} \prod_{y \in Y}S_{  f^{-1}(y)}\\
& = & \mathrm{colim}_{\X \in  \Coriso} S(\X)
\end{array}\] where $\Coriso\subset X\Grisok$ is the groupoid of $X$-graphs $\X_{\mathrm{cor}}$ that are disjoint unions of corollas $\X_{\mathrm{cor}}= \coprod_{ i \in I} \CX[X_i]$ (with $\coprod_{i \in I} X_i = X$). In particular, from this point of view, the multiplication $\mu^\LL$ for $\LL$ is just the restriction of the multiplication $\mu^{\TTk}$ for $\TTk$ (\cref{prop nonunital CA}) to $\X_{\mathrm{cor}}$-shaped graphs of disjoint unions of corollas, 
and the unit $\eta^\LL$ for $\LL$, is the obvious corestriction. 

Given a finite set $X$, and a connected $X$-graph $\X \in X \Griso$, let $\CorGg$ denote the full subgroupoid of $\Gret^{(\X)}$ whose objects are nondegenerate $ \X$-shaped graphs of graphs $\Gg$ such that, for all $(\C_{X_b}, b) \in \elG[\X]$, $ \Gg (\C_{X_b}, b) \in \Coriso$ is a disjoint union of corollas. Then, \[ \begin{array}{lll}
TL S_X& = &\mathrm{colim}_{ \X \in X \Griso}  \mathrm{lim}_{(\C,b) \in \elG[\X]} LS(\C)\\
&=& \mathrm{colim}_{ \X \in X \Griso}  \mathrm{lim}_{ \Gg \in\CorGg} S(\Gg(\X)).
\end{array}\]

Since, for all $\Gg \in \CorGg$, the colimit $\Gg(\X)$ has the structure of an admissible (in general, not connected) $X$-graph, there is a canonical morphism $TLS_X \to \mathrm{colim}_{\X' \in X \Grisok} S(\X')$. 
And, therefore, since every graph is the coproduct in $\Gret$ of its connected components, $S(\X')= \prod_{k \in K}S(\X'_k)$ describes an element of $LTS_X$ for all $X$-graphs $\X'=  \coprod_{k \in K} \X'_k$.

Let $\lambda_{\LL \TT}\colon TL \Rightarrow LT$ be the natural transformation so obtained. In other words, $\lambda_{ \LL \TT}$ is induced by the canonical morphisms, 
\[\mathrm{colim}_{ \X \in X \Griso } \left( \mathrm{colim}_{ \Gg \in \CorGg} S (\Gg(\X)) \right)\to \mathrm{colim}_{ \X' \in X \Griso } S (\X') \  (\text{where } S\in \GSE, \text{ and } X \in \fiso),\] that forgets pairs
 $(\X, \Gg)$ where $\X$ is a connected admissible $X$-graph and $\Gg \in\CorGg$, and remembers only the $X$-graphs given by the colimit $\Gg(\X) = \coprod_{k \in K} \Gg(\X)_k$, in the evaluation of $S(\Gg(\X))$.

\begin{lem}\label{lem. LT is Tk}
The natural transformation $\lambda_{\LL \TT} \colon TL \Rightarrow LT$ satisfies the four axioms of a distributive law, and $\LL \TT = \TTk$ on $\GSE$. Hence, algebras for $\LL\TT$ are nonunital circuit algebras.
\end{lem}
\begin{proof}
Since a graph is the disjoint union of its connected components, $LT = \Tk  \colon \GSE \to \GSE$. 

Moreover, $L \subset \Tk$ and $T \subset \Tk$ so $TL \subset ({\Tk})^2$ is a subfunctor, and $\lambda_{\LL\TT}$ is just the restriction to $TL$ of $\mu^{\TTk} \colon ({\Tk})^2 \Rightarrow \Tk = LT$ on $\GSE$. The four distributive law axioms then follow immediately from the monad axioms.
\end{proof}

Observe also that there is an obvious choice of natural transformation $\lambda_{ \LL \DD} \colon DL \Rightarrow LD$ that defines the distributive law for $\LL$ and $\DD$:

\begin{lem}
	 The natural transformation $\lambda_{\LL \DD} \colon DL \Rightarrow LD$ obtained by extending 
	 the natural transformation 
$L \eta^{\DD} \colon L \Rightarrow LD$ to a $DL$ by
\[ S_\S \to DS_\two \xrightarrow {\eta^\LL DS} LDS_\two, \ \text{ and } S_\S \to DS_\nul \xrightarrow {\eta^\LL DS} LDS_\nul, \ \text{ for all } S \in \GSE\] describes a distributive law.
\end{lem}

\begin{proof} 
Informally, $\lambda_{\LL \DD}$ takes a tuple $(S_{X_i})_{i = 1}^n \in LS_{\prod_i X_i}$ of objects in the image of $S$, and regards it as a tuple in the image of $DS$, and takes the image of $S_\S$ in $DS_\nul$ or $DS_\two$ and regards it as a 1-tuple in the image of $DS \hookrightarrow LDS$. This construction is natural in $S$ and it is simple to show, using the monad unit axiom for $\eta^{\DD}$, that $\lambda_{\LL \DD}$ so obtained satisfies the axioms of a distributive law. 
\end{proof}

\subsection{An iterated distributive law for circuit algebras}\label{sec. iterated}
So far, we have defined three monads $\TT, \DD$ and $\LL$ on $\GSE$, and three distributive laws $\lambda_{\DD\TT}\colon TD \Rightarrow DT$, $\lambda_{\LL\TT}\colon TL\Rightarrow LT$  and $\lambda_{\LL\DD}\colon DL\Rightarrow LD$. It remains to show that the Yang-Baxter diagram (\ref{YB}) commutes and therefore there is an iterated distributive law for the three monads such that the resulting composite monad $\LL \DD \TT $ on $\GSE$ is the desired circuit algebra monad. 

\medspace

Let $\G$ be any graph and $W \subset V_2 $ any subset of bivalent vertices of $\G$. As in \cref{ss. pointed}, let $\Gnov$ be the graph obtained from $\G$ by deleting the vertices in $W$. If $\Gg_{\setminus W} \colon \elG [\Gnov]\to \Gret$ is a nondegenerate $\Gnov$-shaped graph of graphs, then there is an induced nondegenerate $\G$-shaped graph of graphs $\Gg^W \colon \elG \to \Gret$ defined by
\begin{equation}
\label{eq. GgW}
(\C,b) \mapsto \left \{ \begin{array}{ll}
\C_\two & (\C,b) \text{ is a neighbourhood of } v \in W\\
\Gg_{\W} (\delW \circ b) & \text{ otherwise.}
\end{array} \right .
\end{equation} 
Moreover, $\Gg^W$ induces an inclusion $W \hookrightarrow V_2(\Gg^W (\X))$ (see \cref{fig. GgW}), and   
 \begin{equation}
 \label{eq. deleted colims} \Gg^W(\X)_{\setminus W} = \Gg_{\setminus W}(\X_{\setminus W}) .
 \end{equation} 
Hence, the assignment $\Gg_{\setminus W} \to \Gg^W $ defines a canonical full embedding $\Gret^{(\Gnov)}\hookrightarrow \Gret^{(\G)}$ of categories that preserves similarity classes of colimits:   
\begin{figure}[htb!]
	\begin{tikzpicture} [scale = .8]
	\node at (-2,-4){$\Gg_{\setminus W}$};
	\node at(0,-5){ 
		\begin{tikzpicture}[scale = .3]
\begin{pgfonlayer}{above}
\node [dot, red] (0) at (2.5, -0.5) {};
\node [dot, red] (1) at (3, -1) {};
\node [dot, red] (3) at (-1, -0.25) {};
\node [dot, red] (4) at (-0.25, -1) {};
\node [dot, red] (5) at (-1.75, -1) {};
\node  (7) at (1.5, -3) {};
\node  (9) at (2, 2) {};
\node  (10) at (-4, -4) {};
\node  (11) at (5, -5) {};
\node  (12) at (5, 4) {};
\node  (13) at (7, 0) {};
\draw [draw = blue, fill = blue, fill opacity = .2](2.7,-.75) circle (1.3cm);
\draw [draw = blue, fill = blue, fill opacity = .2](-1, -.5) circle (1.3cm);
\end{pgfonlayer}
\begin{pgfonlayer}{background}
\draw [in=165, out=-120, loop] (0) to ();
\draw [bend right=45, looseness=1.25] (5) to (4);
\draw [in=135, out=60] (3) to (0);
\draw [bend left] (3) to (9.center);
\draw [bend right, looseness=1.25] (9.center) to (12.center);
\draw [bend right] (5) to (10.center);
\draw [bend left] (4) to (7.center);
\draw [in=150, out=-90, looseness=0.75] (7.center) to (11.center);
\draw [in=180, out=60] (1) to (13.center);
\end{pgfonlayer}
		\end{tikzpicture}
		
	};
	\node at (6,-4){$\Gg^W$};
\node at(8,-5){ 
	\begin{tikzpicture}[scale = .3]
	\begin{pgfonlayer}{above}
	\node [dot, red] (0) at (2.5, -0.5) {};
	\node [dot, red] (1) at (3, -1) {};
	\node [dot, red] (3) at (-1, -0.25) {};
	\node [dot, red] (4) at (-0.25, -1) {};
	\node [dot, red] (5) at (-1.75, -1) {};
	\node [dot, green] (7) at (1.5, -3) {};
	\node [dot, green] (9) at (2, 2) {};
	\node  (10) at (-4, -4) {};
	\node  (11) at (5, -5) {};
	\node  (12) at (5, 4) {};
	\node  (13) at (7, 0) {};
	\draw [draw = blue, fill = blue, fill opacity = .2](2.7,-.75) circle (1.3cm);
	\draw [draw = blue, fill = blue, fill opacity = .2](-1, -.5) circle (1.3cm);
		\draw [draw = green, fill = green, fill opacity = .2](7) circle (.6cm);
			\draw [draw = green, fill = green, fill opacity = .2](9) circle (.6cm);
	\end{pgfonlayer}
	\begin{pgfonlayer}{background}
	\draw [in=165, out=-120, loop] (0) to ();
	\draw [bend right=45, looseness=1.25] (5) to (4);
	\draw [in=135, out=60] (3) to (0);
	\draw [bend left] (3) to (9.center);
	\draw [bend right, looseness=1.25] (9.center) to (12.center);
	\draw [bend right] (5) to (10.center);
	\draw [bend left] (4) to (7.center);
	\draw [in=150, out=-90, looseness=0.75] (7.center) to (11.center);
	\draw [in=180, out=60] (1) to (13.center);
	\end{pgfonlayer}
	\end{tikzpicture}
	
};
	\node at (-2,1){$\Gnov$};
	\node at(0,0){ 
	\begin{tikzpicture}[scale = .3]
	\begin{pgfonlayer}{above}
	\node [dot, blue] (0) at (2.5, -0.5) {};
	\node [dot, blue] (1) at (2.5, -0.5) {};
	\node [dot, blue] (3) at (-1, -0.5) {};
	\node [dot, blue] (4) at (-1, -0.5) {};
	\node [dot, blue] (5) at (-1, -0.5) {};
	\node  (7) at (1.5, -3) {};
	\node  (9) at (2, 2) {};
	\node  (10) at (-4, -4) {};
	\node  (11) at (5, -5) {};
	\node  (12) at (5, 4) {};
	\node  (13) at (7, 0) {};
	\end{pgfonlayer}
	\begin{pgfonlayer}{background}
	\draw [in=135, out=60] (3) to (0);
	\draw [bend left] (3) to (9.center);
	\draw [bend right, looseness=1.25] (9.center) to (12.center);
	\draw [bend right] (5) to (10.center);
	\draw [bend left] (4) to (7.center);
	\draw [in=150, out=-90, looseness=0.75] (7.center) to (11.center);
	\draw [in=180, out=60] (1) to (13.center);
	\end{pgfonlayer}
	\end{tikzpicture}	
};
	\node at (6,1){$\G$};
	\node at(8,0){ 
	\begin{tikzpicture}[scale = .3]
	\begin{pgfonlayer}{above}
	\node [dot, blue] (0) at (2.5, -0.5) {};
	\node [dot, blue] (1) at (2.5, -0.5) {};
	\node [dot, blue] (3) at (-1, -0.5) {};
	\node [dot, blue] (4) at (-1, -0.5) {};
	\node [dot, blue] (5) at (-1, -0.5) {};
	\node [dot, green] (7) at (1.5, -3) {};
	\node  [dot, green](9) at (2, 2) {};
	\node  (10) at (-4, -4) {};
	\node  (11) at (5, -5) {};
	\node  (12) at (5, 4) {};
	\node  (13) at (7, 0) {};
	\end{pgfonlayer}
	\begin{pgfonlayer}{background}
	\draw [in=135, out=60] (3) to (0);
	\draw [bend left] (3) to (9.center);
	\draw [bend right, looseness=1.25] (9.center) to (12.center);
	\draw [bend right] (5) to (10.center);
	\draw [bend left] (4) to (7.center);
	\draw [in=150, out=-90, looseness=0.75] (7.center) to (11.center);
	\draw [in=180, out=60] (1) to (13.center);
	\end{pgfonlayer}
	\end{tikzpicture}	
};
\end{tikzpicture}
\caption{The colimits of the $\G$-shaped graph of graphs $\Gg^W$ described in (\ref{eq. GgW}), and the $\Gnov$-shaped graph of graphs $\Gg_{\setminus W}$ are similar in $\Gretp$. }
	\label{fig. GgW}
\end{figure}
 In particular, the embedding $\Gret^{(\Gnov)}\hookrightarrow \Gret^{(\G)}$ restricts to an inclusion  $\mathsf{Cor}_\boxtimes^{(\Gnov)}\hookrightarrow \mathsf{Cor}_\boxtimes^{(\G)}$. 
 

\begin{prop}\label{prop. iterated law}
	The triple of distributive laws $\lambda_{ \DD \TT}\colon TD \Rightarrow DT$, $\lambda_{ \LL \TT}\colon TL \Rightarrow LT$, and $\lambda_{ \LL \DD }\colon DL  \Rightarrow LD$ on $\GS$, describe a composite monad $\LL \DD\TT$ on $\GSE$. 

\end{prop}
\begin{proof}
By \cite[Theorem~1.6]{Che11}, we must check that the diagram (\ref{YB}) commutes for $\lambda_{ \DD \TT}, \lambda_{ \LL \TT}$, and $\lambda_{ \LL \DD }$:

For all graphical species $S$ in $\E$ and all finite sets $X$, 
\[
TDLS_X  = \mathrm{colim}_{ \X \in X \Griso} \coprod_{W \subset (V_0 \amalg V_2)(\X)} \mathrm{lim}_{ \Gg_{\setminus W}\in \mathsf{Cor}_\boxtimes^{\scriptscriptstyle{(\X_{\setminus W})}}} S(\Gg_{\setminus W}(\X_{\setminus W})) \] 
is indexed by triples 
$(\X, W, \Gg _{\setminus W})$, where $\X$ is a connected $X$-graph, $W \subset V_2(\X)$ (or $\X \cong \C_\nul$ and $W= V(\C_\nul)$) and $\Gg_{\setminus W} $ is a nondegenerate $\X_{\setminus W}$-shaped graph of graphs such that $\Gg_{\setminus W} (b)$ is a disjoint union of corollas for each $(\C_{X_b}, b) \in \elG[\X_{\setminus W}]$.

Observe first that, if $ W = V(\X)$, then either $\X \cong \C_\nul$ or $\X \cong \Wm$ in $\nul \Griso$, or $\X \cong \Lk $ in $ \two \Griso$. In each case $\X_{\setminus W} \cong (\shortmid)$ and hence $\Gg_{\setminus W}$ is the trivial graph of graphs $id \mapsto (\shortmid)$. In particular, the maps $S(\Gg_{\setminus W} (\X_{\setminus W})) \to TDLS_X \to LDTS_X$ induced by each path $TDL \Rightarrow  LDT$ in (\ref{YB}) coincide, since both factor through $S_\S$. 
 
Assume therefore that $W \neq V(\X)$. The two paths in (\ref{YB}) are then described by the following data:

\begin{description}
	\item[Top]
		\[\small{\xymatrix@C = .8cm{
			&	{	\begin{array}{c}
				(\X_{\setminus W}, \Gg_{\W}),\\ \X_{\setminus W}\in X \Griso
				\end{array}}
			\ar@{..>}[rrr]^-{\tiny{\text{ take colimit of } \Gg_{\setminus W}}}\ar@{..>}[d]&&& 	{	\begin{array}{c}
			\Gg_{\setminus W}(\X_{\setminus W})\\ \in X \Grisok
				\end{array}}	\ar@/^1.0pc/@{..>}[rd]^-{L\eta^{\DD}T}\ar@{..>}[d]&\\
		{	\begin{array}{c}
	(\X, W,  \Gg_{\setminus W}),\\ \X \in X \Griso
			\end{array}}\ar@/^1.0pc/@{..>}[ur]^-{\tiny{\text{delete } W \text{ in } \X}}\ar@{..>}[d] &S(\Gg_{\setminus W}(\X_{\setminus W}))\ar[d]\ar[rrr]&&&
		S(\Gg_{\setminus W}(\X_{\setminus W}))\ar[d]\ar@/^1.0pc/[rd]& {	\begin{array}{c}
			\Gg_{\setminus W}(\X_{\setminus W}) \\  \in X \Grisok
			\end{array}} \ar@{..>}[d]
			\\
		 S(\Gg_{\setminus W}(\X_{\setminus W}))\ar[d]	\ar@/^1.0pc/[ru]& DTLS_X\ar[rrr]_-{D\lambda_{\LL\TT}S} &&& DLTS_X \ar@/^1.0pc/[rd]_-{\lambda_{\LL\DD}TS}& S(\Gg_{\setminus W}(\X_{\setminus W}))\ar[d]\\
			TDLS_X	\ar@/^1.0pc/[ur]_-{\lambda_{ \DD \TT}LS} &&&&  & LDTS_X}}\]
where, in each case, we evaluate $S(\Gg_{\setminus W}(\X_{\setminus W}))$ and take the appropriate colimit of such objects.
	\item[Bottom]
The bottom path of (\ref{YB}) gives
	\[ \small{\xymatrix@C = .8cm{ {	\begin{array}{c}
					(\X, W,  \Gg_{\setminus W}) ,\\ \X\in X \Griso
				\end{array}}
	\ar@{..>}[d]	\ar@/_1.0pc/@{..>}[dr]^-{\tiny{\text{ forget } W}}
&&&&&{	\begin{array}{c}
\Gg^W(\X)_{\setminus W} \\ \in X \Grisok
	\end{array}} 
 \ar@{..>}[d] \\
		S( \Gg_{\setminus W}(\X_{\setminus W}))\ar[d]	\ar@/_1.0pc/[dr]&   {	\begin{array}{c}
			(\X, \Gg^{W}, W)  ,\\ \X\in X \Griso
			\end{array}}  \ar@{..>}[d]  \ar@{..>}[rrr]^-{\tiny{\text{ take colimit of } \Gg^W}}  &&& {	\begin{array}{c}
		(\Gg^W(\X), W)  ,\\ \Gg^W(\X)\in X \Grisok
			\end{array}}  \ar@/_1.0pc/@{..>}[ru]^-{\tiny{\text{ delete } W \text{ in } \Gg^W(\X)}} \ar@{..>}[d]&
	S(\Gg^W(\X)_{\setminus W})\ar[d]	\\
	TDLS_X\ar@/_1.0pc/[dr]_{T \lambda_{\DD\LL}S} & S(\Gg^W(\X)_{\setminus W})\ar[d] \ar[rrr]&&&S(\Gg^W(\X)_{\setminus W})\ar@/_1.0pc/[ru]\ar[d]&LDTS_X \\
	& TLDS_X\ar[rrr]_-{\lambda_{\LL\TT}DS}&&& LTDS_X \ar@/_1.0pc/[ru]_-{L \lambda_{\DD\TT} S}&}}\]
where $\Gg^W$ is the nondegenerate $\X$-shaped graph of graphs in $\mathsf{Cor}_\boxtimes^{(\X)}$ obtained from $\Gg_{/W}$ as in (\ref{eq. GgW}).
\end{description}

So, $ \Gg_{\setminus W}(\X_{\setminus W}) = \Gg^W(\X)_{\setminus W}$ by (\ref{eq. deleted colims}). Hence, the diagram (\ref{YB}) commutes, and $(\lambda_{\LL\DD}, \lambda_{\LL\TT}, \lambda_{\DD\TT})$ describe an iterated distributive law for $\LL, \DD$ and $\TT$ on $\GSE$.
\end{proof}

It therefore remains to prove:
\begin{thm}\label{thm. iterated law}
The EM category of algebras for the composite monad $\LL \DD\TT$ on $\GSE$ is isomorphic to the category $\COE$ of circuit algebras in $\E$. 
\end{thm}

\begin{proof}
By \cite{Bec69}, for a composite monad is equipped with an algebra structure for each of the component monads.  Hence, an algebra $(A, h)$ for the composite monad $\LL \DD \TT$ on $\GSE $ has an $\LL \TT$-algebra structure $h^{\LL\TT} \colon LTA \to A$, and also a $\DD \TT$ algebra structure $h^{\DD\TT} \colon DTA \to A$. Each structure is given by the corresponding restriction $LTA \to LDT A \xrightarrow{h} A$, respectively $DTA \to LDT A \xrightarrow{h} A$. Therefore, by \cref{lem. LT is Tk}, $h$ induces an external product $\boxtimes$, and contraction $\zeta$ on $A$ such that $(A, \boxtimes, \zeta)$ is a nonunital circuit algebra. And, by \cref{thm CSM monad DT}, $h$ induces a unital multiplication $(\diamond, \epsilon)$ on $A$, such that 
\[\diamond_{X,Y}^{x \ddagger y}  = \zeta_{X \amalg Y}^{x \ddagger y}\boxtimes_{X, Y}\colon (A_X \times A_Y)^{x\ddagger y } \to A_{X \amalg Y \setminus \{x, y\}}.\] Hence 
\[  \xymatrix{ A_X \ar[rr]^-{ (id \times ch_x)\circ \Delta}&&( A_X \times A_\S )^{x \ddagger 2}\ar[rr] ^-{ id \times \epsilon} &&(A_X \times A_\two)^{x \ddagger 2} \ar[rr]^-{\zeta^{x \ddagger 2}_{X \amalg \two}\circ \boxtimes}&&
	A_X }\] is the identity on $A_X$. So, $(A, h)$ has a canonical circuit algebra structure.

Conversely, a circuit algebra $(S, \boxtimes, \zeta, \epsilon)$ in $\E$ has an underlying modular operad structure $(S, \diamond, \zeta, \epsilon)$, where $\diamond_{X,Y}^{x \ddagger y}  = \zeta_{X \amalg Y}^{x \ddagger y}\boxtimes_{X, Y}$ and hence $(S, \boxtimes, \zeta, \epsilon)$ defines a $\DD \TT$-algebra structure $h^{\DD\TT} \colon DTS \to S$ by \cref{thm CSM monad DT}. By \cref{lem. L free monoid}, $(S, \boxtimes)$ has an $\LL$-algebra structure where $h^\LL \colon LS \to S$ is given by $\boxtimes_{X,Y} \colon S_X \times S_Y$ on $S_X \times S_Y$. It follows immediately from the circuit algebra axioms that $h^\LL \circ L h^{\DD\TT} \colon LDT S \to S$ gives $S$ the structure of an $\LL\DD \TT$-algebra.
                                                                   
These maps are functorial by construction and extend to an equivalence $\COE \simeq \GSE^{\LL\DD\TT}$.\end{proof}

\begin{cor}\label{cor coloured algebras}
		If $\E = \Set$, then for each palette $(\CCC, \omega)$ in $\Set$, the subcategory $\CCO$ of $(\CCC, \omega)$-coloured circuit algebras  is the EM-category of algebras for the monad $\LL \DD \TT$ restricted to $\CGS$. By \cref{prop. CA CO}, this is equivalent to the category $\mathsf{Alg}_{\Set}(\CWD) = \CCA$ of $\Set$-valued algebras for the operad $\CWD$ of $(\CCC, \omega)$-coloured wiring diagrams defined in \cite[Section~5.2]{RayCA1}. 
\end{cor}

\begin{ex}
	\label{ex wheeled prop monad} 
	 By \cref{def wp int} and \cref{thm. iterated law}, the category $\WPE$ of wheeled props in $\E$ is isomorphic to the slice category $\GSE^{\LL\DD\TT} \ov \DicommE$. In particular -- as in Examples \ref{ex nonunital WP} and \ref{ex Wheeled properads} -- the constructions of Sections \ref{ssec. LL} and \ref{sec. iterated} may be modified, in terms of directed graphs, to obtain oriented versions of the monads and distributive laws and a composite wheeled prop monad on $\oGSE$. 
	
\end{ex}

	\begin{rmk}
	\label{rmk operad dist}
	Like monads, operads encode algebraic structure. Indeed, for any operad in $\Set$, one may define a monad with the same algebras.  Markl \cite{Mar96} defined distributive laws for operads and proved that the operad obtained by composing Koszul operads via an \textit{operad distributive law} is Koszul. This result was used in \cite{KW24} to compose an operad (of downward wiring diagrams \cite[Section~5.1]{RayCA1}) associated to the restriction of $\TT$ to $\CGS[\{*\}]$ and an operad associated to the restriction of $\LL$ to $\CGS[\{*\}]$. So, to determine whether the operad (of wiring diagrams) for monochrome unital circuit algebras obtained from $\LL\DD\TT$ is Koszul, it should be sufficient to explore whether the 
	operad obtained by restricting $\DD$ is. This is not expected to be the case, and the question is not further explored here. 
\end{rmk}
 \section{A nerve theorem for circuit algebras}\label{sec. nerve}

 This section recalls the nerve theorem for $\Set$-valued modular operads \cite[Theorem~8.2]{Ray20} and gives a proof, by the same methods, of an analogous result for $\Set$-valued circuit algebras. 

\subsection{Overview of the method}
Let $\Grp \hookrightarrow \GSp$ be the full subcategory of connected graphs and pointed \'etale morphisms from \cref{ss. pointed}, that is obtained in the bo-ff factorisation of the functor $\Gr \to \GS \to \GSp$, and, as in \cref{ssec. Similar connected X-graphs}, let $\TTp$ be the monad on $\GSp$ whose EM category of algebras is equivalent to the category $\CSM$ of modular operads (in $\Set$). 

So, there is a commuting diagram of functors

\begin{equation} \label{eq: MO nerve big picture}
\xymatrix{ 
	&&\Klgr\ar@{^{(}->} [rr]_-{\text{f.f.}}						&& \CSM\ar@<2pt>[d]^-{\text{forget}^{\TTp}} \ar [rr]^-N				&&\pr{\Klgr}\ar[d]^{j^*}	\\
	\fisinvp \ar@{^{(}->} [rr]	^-{\text{}}_-{\text{ f.f.}}		&& \Grp \ar@{^{(}->} [rr] ^-{\text{}}_-{\text{ f.f.}} \ar[u]^{j}_{\text{b.o.}}			&& \GSp \ar@<2pt>[d]^-{\text{forget}^{\DD}} \ar@{^{(}->} [rr]^-{\text{   }}_-{\text{f.f.}}	 \ar@<2pt>[u]^-{\text{free}^{\TTp}}	&& \pr{\Grp}\ar[d] \\
	\fisinv\ar@{^{(}->} [rr]^-{\text{}}_-{\text{ f.f.}}\ar[urr]_-{\text{dense}}\ar[u]^{\text{b.o.}}	&& \Gr \ar@{^{(}->} [rr]^-{\text{}}_-{\text{ f.f.}}	 \ar[u]_{\text{b.o.}}	&& \GS \ar@{^{(}->} [rr]_-{\text{   }}_-{\text{f.f.}}	 \ar@<2pt>[u]^-{\text{free}^{\DD}}		&& \pr{\Gr}}
\end{equation}
where $\Klgr$ is the category, whose objects are indexed by connected graphs, obtained in bo-ff factorisation of the functor $\Gr \to \GS \to \CSM$. 

In \cite[Section~8]{Ray20}, it was shown (for $\E = \Set$) that, though the modular operad monad $\DD\TT$ on $\GS$ does not have arities, the monad $\TTp$ on $\GSp$ does, whereby the following theorem holds: 

\begin{thm}{\cite[Theorem~8.2]{Ray20}}\label{MO nerve}
The monad $\TTp$, and the category $\Grp \hookrightarrow \GSp$, satisfy the conditions of \cite{BMW12} ($\TTp$ has arities $\Grp$). Hence, the nerve $N\colon \CSM \to \pr{\Klgr} $ is fully faithful. 
Its essential image consists of those presheaves $P$ on $\Klgr$ such that 
$P(\G)  = \mathrm{lim}_{(\C,b) \in \elG} P(\C)$ for all graphs $ \G$.

\end{thm} 

 The circuit algebra monad $\LL \DD \TT$ on $ \GS$ also does not have arities. However, since $\GSp^{\TTp} \cong \GS^{\DD\TT}$ (\cref{prop. extension exists}) and $\LL$ lifts functorially to $\GS^{\DD\TT}$, the distributive law $\lambda_{\LL \TT}\colon TL \Rightarrow LT$ on $\GS$ extends to a distributive law $\Tp \widetilde L \Rightarrow \widetilde L \Tp$ for $\TTp$ and the lift $\widetilde \LL $ of $\LL $ to $\GSp$. By \cref{thm. iterated law}, the EM category of algebras for this composite monad $\TTpk$ on $\GSp$ is precisely the category $\CO$ of circuit algebras in $\Set$. 

By \cref{ss. pointed} (see also  \cite[Section~7.3]{Ray20}), the full subcategory $\Gretp \subset \GSp$ of all graphs and pointed morphisms induces a fully faithful inclusion $\GSp \hookrightarrow \pr{\Gretp}$. Hence, if $\Klgrt$ is the category obtained in the bo-ff factorisation of $\Gret \to \GS \to \CO$.  we may consider the following diagram of functors in which the left and middle squares commute, and the right squares commute up to natural isomorphism: 
\begin{equation} \label{eq: CO Gr diag picture}
\xymatrix{ 
	&&{\Klgrt}\ar@{^{(}->} [rr]_-{\text{f.f.}}						&& \CO\ar@<2pt>[d]^-{\text{forget}^{\widetilde \LL\TTp}} \ar [rr]^-{N_{{\Klgrt}}	}			&&\pr{{\Klgrt}}\ar[d]^{j^*}	\\
	\fisinvp \ar@{^{(}->} [rr]	^-{\text{}}_-{\text{ f.f.}}		&& \Gretp \ar@{^{(}->} [rr] ^-{\text{}}_-{\text{ f.f.}} \ar[u]^{j}_{\text{b.o.}}			&& \GSp \ar@<2pt>[d]^-{\text{forget}^{\DD}} \ar@{^{(}->} [rr]^-{\text{   }}_-{\text{f.f.}}	 \ar@<2pt>[u]^-{\text{free}^{\widetilde \LL\TTp}}	&& \pr{\Gretp}\ar[d] \\
	\fisinv\ar@{^{(}->} [rr]^-{\text{}}_-{\text{ f.f.}}\ar[urr]_-{\text{dense}}\ar[u]^{\text{b.o.}}	&& \Gret \ar@{^{(}->} [rr]^-{\text{}}_-{\text{ f.f.}}	 \ar[u]_{\text{b.o.}}	&& \GS \ar@{^{(}->} [rr]_-{\text{   }}_-{\text{f.f.}}	 \ar@<2pt>[u]^-{\text{free}^{\DD}}		&& \pr{\Gret}.}
\end{equation}

The remainder of this paper is dedicated to proving the following theorem using a modification of the proof of \cref{MO nerve} (\cite[Theorem~8.2]{Ray20}) to show that $\Gretp$ provides arities for $\TTpk$:

\begin{thm}\label{nerve theorem} 	\label{thm: CO nerve}The functor $N\colon\CO \to\pr{\Klgrt}$ is full and faithful. Its essential image consists of precisely those presheaves $P$ on $\Klgrt$ such that 
	\begin{equation}
	\label{eq. Segal nerve sec}P(\G)  = \mathrm{lim}_{(\C,b) \in \elG} P(\C) \text{ for all graphs } \G.
	\end{equation}
\end{thm}

Recall that the category $\WPin = \GS^{\LL \DD \TT} \ov \Dicomm \cong \oGS^{\LL\DD\TT}$ is the category of wheeled props in $\Set$.   Hence, \cref{nerve theorem} implies an abstract nerve theorem for wheeled props. 
For, let $\mathsf O \Klgrt \defeq \Klgrt \ov \Dicomm$ be the induced oriented graphical category. 
Since $\pr{ \mathbf G}\ov \Dicomm \simeq \pr{\mathbf G \ov \Dicomm}$ for $\mathbf G = \Gret, \Gretp, \Klgrt$, it follows that: 

\begin{cor}\label{cor wheeled props nerve}
	
The induced functor $\mathsf O N\colon \WPin \to\pr{\mathsf O \Klgrt}$ is full and faithful. Its essential image consists of precisely those presheaves $P$ on $\mathsf O  \Klgrt$ that restrict to graphical species. 

\end{cor}

Sections \ref{ssec LTp}-\ref{ssec factorisation} are devoted to constructing the graphical category $\Klgrt$ and giving a direct proof of \cref{nerve theorem}, in the style of \cref{MO nerve} (see also, \cite[Theorem~8.2]{Ray20}). Further proofs, using results from \cite{BMW12} and \cite{CH21} on strongly cartesian monads, are outlined in Remarks \ref{rmk bmw proof} and \ref{rmk patterns proof}.

\subsection {The monad $\widetilde \LL\TTp$ on $\GSEp$}\label{ssec LTp}

The first step is to understand the monad $\widetilde \LL\TTp$ on $\GSEp$. Here $\E$ is, as usual, an arbitrary category with sufficient limits and colimits. 

The lift $\widetilde{\LL}$ of $\LL$ to $\GSEp$ is given by $\widetilde L (S_*) = (LS, \eta^\LL \epsilon ,\eta^\LL o)$, for all pointed graphical species $S_* = (S, \epsilon, o)$ in $\E$ (where $\mu^{\widetilde \LL}, \eta^{\widetilde \LL}$ are induced by $\mu^\LL$ and $\eta^\LL$ in the obvious way.)

Let $X$ be a finite set. The category $\XGrsimp$ of similar connected $X$-graphs was defined in \cref{ssec. Similar connected X-graphs} and the monad $\TTp = (\Tp, \mu^{\TTp}, \eta^{\TTp})$ with  $\Tp S_X  = \mathrm{colim}_{\X \in \XGrsimp} S(\X)$ was discussed in \cref{ssec. Similar connected X-graphs}. 

Similarity morphisms preserve connected components. Hence the category $X\XGretsimp$ of \textit{all} $X$-graphs and similarity morphisms, may be defined by 
\begin{equation}\label{eq. XGretsimp} X\XGretsimp \defeq \mathrm{colim}_{(Y,f) \in \core{X\ov\fin} }
\coprod_{y \in Y} \left(f^{-1}(y)\right) \XGrsimp.\end{equation}

The endofunctor $\Tpk$ is described by the obvious quotient of $\Tk $: Let $S_* = (S, \epsilon, o)$ be a pointed graphical species in $\E$ and let $X$ be a finite set, then \begin{equation}\label{eq. lift LT}
\begin{array}{lll}{\Tpk} S_X& = &\mathrm{colim}_{(Y,f) \in \core{X\ov\fin} }
\left( \prod_{y \in Y}  \Tp S_{f^{-1}(y)} \right)\\
& = &\mathrm{colim}_{(Y,f) \in \core{X\ov\fin} } \left( \prod_{y \in Y}
\left(\mathrm{colim}_{(\G, \rho) \in \XGrsimp[(f^{-1}(y))]} S(\G)\right)\right)\\
& = &\mathrm{colim}_{\X \in X\XGretsimp} S(\X).
\end{array}\end{equation}
And, 
\[ \begin{array}{lll}
\Tp \widetilde L S_X& = &\mathrm{colim}_{ \X \in \XGrsimp}  \mathrm{lim}_{(\C,b) \in \elG[\X]} LS(\C)\\
&=& \mathrm{colim}_{ \X \in \XGrsimp}  \mathrm{lim}_{ \Gg \in\CorGg} S(\Gg(\X)).
\end{array}\] 

As $\GSEp^{\TTp}\cong \GSE^{\DD\TT}$ and $\LL$ lifts to $\GSE^{\DD\TT}$, there is a distributive law $\Tp \widetilde L \Rightarrow \Tpk$,  induced by the distributive law $\lambda_{\LL \TT} \colon TL \Rightarrow LT = \Tk$, which is compatible with similarity morphisms by construction: for each finite set $X$, a colimit indexed by pairs 
$(\X, \Gg)$ of a connected admissible $X$-graph $\X$, and $\Gg \in \CorGg$, is replaced by a 
colimit indexed over 
(in general not connected) admissible $X$-graphs $\Gg(\X)$.

\subsection{A graphical category for circuit algebras}\label{ss. graphical category}

For the remainder of the paper, $\E = \Set$.

The category $\Gretp$ obtained in the bo-ff factorisation of $\Gret \to \GS \to \GSp$ has been discussed in \cref{ss. pointed}. Let $\Klgrt$ be the category, whose objects are graphs, obtained in the bo-ff factorisation of $\Gret \to \GS \to \CO$ (and also of  $\Gretp \to \GSp \to \CO$). This is the full subcategory of the Kleisli category ${\GSp}_{\TTpk}$ of $\TTpk$ on free circuit algebras of the form $\Tpk \yetp \G = LDT \yet \G$ for graphs $\G \in \Gret$. 
By definition, morphisms $ \beta \in \Klgrt(\G, \H)$ are given by morphisms in $\GSp (\yetp \G, \Tpk \yetp \H)$.

To understand $\Klgrt$, the first step is therefore to describe $\Tpk \yetp \H$ for all graphs $\H$.

Since each connected component of $X\XGretsimp $ contains an (in general, not connected) admissible $X$-graph $\X$, it follows from (\ref{eq. lift LT}) that, for all finite sets $X$, elements of $ (\Tpk \yetp \H)_X \cong \Klgrt(\CX, \H) $ are represented by pairs $(\X, f)$ where $\X$ 
is an admissible $X$-graph and $f \in \Gretp(\X, \H)$. In particular, $f$ factors uniquely as a morphism $\X \to \X_{\setminus W_f}$ in $\Gretsimp$ followed by a morphism $f^\bot \in \Gret(\X_{\setminus W_f}, \H)$.
Pairs $(\X^1, f^1)$ and $(\X^1, f^2)$ represent the same element $[\X,f]_* \in  (\Tpk \yetp \H)_X  $ if and only if there is an object $\X_f \in X\XGretsimp$ and a commuting diagram 
\begin{equation}\label{connected graph factors}\xymatrix{ 
	\X^1 \ar[rr]^-{g^1} \ar[drr]_{f^1}& & \X^f  \ar[d]^{f^\bot }&&\ar[ll]_-{g^2} \X^2 \ar[dll]^{f^2} \\&&\H&&}\end{equation} in $\Gretp$ such that $g^j$ are similarity morphisms in $X\XGretsimp$ for $j = 1,2$, and $f^\bot \in \Gret (\X^f, \H)$ is an (unpointed) \'etale morphism.

In particular, for all $\H$, $ e \in E(\H)$ and $m \geq 1$, the following special case of (\ref{connected graph factors}) commutes in $\Gretp$:
\begin{equation}\label{0 connected} \xymatrix{ \C_\nul \ar[rr]^-{  z} \ar[drr]_{ch_e \circ z } & &(\shortmid)\ar[d]^{ch_e}&& \Wl \ar[ll]_{\kappa } \ar[dll]^-{ch_e \circ \kappa}\\
	&&\H.&&}\end{equation} 

Since $\Klgrt (\G, \H) \cong \Tpk \yetp \H (\G)   = \mathrm{lim}_{ (\C, b) \in \elG} \Tpk \yetp \H(\C)$, morphisms $\gamma \in \Klgrt(\G, \H)$ are represented by pairs $(\Gg, f)$, where $\Gg: \elG \to \Gret$ is a nondegenerate $\G$-shaped graph of graphs, and $f \in \Gretp(\Gg(\G), \H)$ is a morphism in $\Gretp$ (see \cite[Section~8.1]{Ray20}). 

The following lemma generalises \cite[Lemma~8.9]{Ray20}, and is proved in exactly the same manner.

 \begin{lem}\label{lem: Klgr representatives} Let $\G$ and $\H$ be graphs and, for $i = 1,2$, let $\Gg^i$ be a nondegenerate $\G$-shaped graph of graphs with colimit $\Gg^i(\G)$, and $f^i \in \Gretp(\Gg^i(\G), \H)$.
	Then $(\Gg^1, f^1), (\Gg^2, f^2)$ represent the same element $\gamma$ of $\Klgrt(\G, \H)$ if and only if there is a representative $(\Gg, f)$ of $\gamma$, and a commuting diagram in $\Gretp$ of the following form: 
	\begin{equation}\label{eq: well-defined kleisli}\xymatrix{\Gg^1(\G) \ar[rr] \ar[drr]_{f^1} && \Gg(\G)\ar[d]^{f} &&\ar[ll] \Gg^2(\G) \ar[dll]^{f^2}\\&&\H,&&}\end{equation} where the morphisms in the top row are vertex deletion morphisms (including $z \colon \C_\nul \to (\shortmid)$) in $\Gretsimp$, and $f \in \Gret (\Gg(\G), \H)$. 
\end{lem}

 The following terminology is from \cite{JK11}. 
\begin{defn}\label{free} A \emph{(pointed) free morphism in $\Klgrt$} is a morphism in the image of the inclusion $\Gretp \hookrightarrow \Klgrt$. An \emph{unpointed free morphism} in $\Klgrt$ is a morphism in the image of $\Gret \hookrightarrow \Gretp \hookrightarrow \Klgrt$. A \emph{refinement morphism $[\Gg]$ in $\Klgrt$} is a morphism in $\Klgrt(\G,\H)$ with a representative of the form $(\Gg, id_{\H})$. \end{defn}

So, a free morphism in $\Klgrt(\G, \H)$ has a representative of the form $(id, f)$ where $id \colon \elG \to \Gret$ is the trivial identity $\G$-shaped graph of graphs with colimit $\G$ and $f \in \Gretp(\G, \H)$. Refinements $[\Gg]$ in $\Klgrt(\G, \H)$ are described by nondegenerate $\G$-shaped graph of graphs $\Gg$ with colimit $\H$, and hence induce isomorphisms $E_0(\G) \xrightarrow{\cong} E_0(\H)$ on boundaries. 

In particular, morphisms in $\Klgrt(\G, \H)$ factor as $\G \xrightarrow{[\Gg]} \Gg(\G) \xrightarrow{\delW} \Gg(\G)_{\setminus W}\xrightarrow {f} \H$. Here $\Gg\colon \elG \to \Gret$ is a nondegenerate $\G$-shaped graph of graphs with colimit $\Gg(\G) \in \Gret$, 
and $f \in \Gret(\Gg(\G)_{\setminus W}, \H)$ is an (unpointed) \'etale morphism $\Gret$.

\subsection{Factorisation categories and the nerve theorem}\label{ssec factorisation}

Let ${\GSp}_{\TTpk}$ be the Kleisli category of $\TTpk$ on $\GSp$. By \cite[Proposition~2.5]{BMW12}, to show that the monad $\TTpk$, together with the category $\Gretp$ satisfy the conditions of the nerve theorem \cite[Theorem~1.10]{BMW12} (that is, $\TTpk$ has arities $\Gretp$) and therefore the nerve functor $N \colon \CO \to \pr{\Klgrt}$ is fully faithful, it is sufficient to show that certain categories, obtained by factoring morphisms in ${\GSp}_{\TTpk}$, are connected. 

To construct these categories, let $S_*$ be a pointed graphical species in $\Set$, $ \G$ a graph in $\Gret$, and let $\beta \in {\GSp}_{\TTpk}(\G, S_*) $ be a morphism in the Kleisli category of $\TTpk$. By the Yoneda lemma, \[{\GSp}_{\TTpk}(\G, S_*) \cong \GSp (\G, \Tpk S_*) \cong \Tpk S_* (\G) = \mathrm{\mathrm{colim}_{(\C,b) \in \elG}} \Tpk S_*(\C) \] where
\[ \Tpk S_*(\C)\cong \left \{ \begin{array}{ll}
S_\S & \C \cong (\shortmid),\\
\mathrm{colim}_{\X \in X\XGretsimp} S (\X) & \C \cong \CX.
\end{array}\right.\]

So, $\beta \in{\GSp}_{\TTp}(\G, S_*)$ is represented by pairs $(\Gg, \alpha)$ of
  \begin{itemize}
  	\item a nondegenerate $\G$-shaped graph of graphs $\Gg$ with colimit $\Gg(\G)$, viewed as a refinement $[\Gg] \in \Klgrt (\G, \Gg(\G))  = {\GSp}_{\TTp}(\G, \Gg(\G))$, 
  	\item an element $\alpha \in S(\Gg (\G))$ viewed, by the Yoneda lemma, as a morphism $\alpha \in {\GSp}_{\TTp}(\yetp\Gg(\G), S_*)$ 
  \end{itemize}
such that the induced composition $\G \xrightarrow {[\Gg]} \Gg(\G) \xrightarrow {\alpha}  S(\Gg(\G)) \twoheadrightarrow (\Tpk S_*)(\G)$ is precisely $\beta$.

\begin{defn}\label{def. factcat}
	Objects of the  \emph{factorisation category} $\factcat$ of $\beta$ are pairs $(\Gg, \alpha)$ that represent $\beta$ as above, and morphisms in $\factcat((\Gg^1, \alpha^1), (\Gg^2, \alpha^2))$ are commuting diagrams in ${\GSp}_{\TTpk}$
\begin{equation}\label{factcat mor}
\xymatrix{&& {[\Gg^1]}(\G)\ar[dd]_{g} \ar[drr]^{ \alpha^1}&&\\
	\G \ar[urr]\ar[drr]&&&& S_*\\
	&&{[\Gg^2]}(\G) \ar[urr]_{ \alpha^2}&&}
\end{equation} such that $g $ is a morphism in $ \Gretp\hookrightarrow {\GSp}_{\TTpk}$.

\end{defn}

\begin{lem}\label{connected} For all pointed graphical species $S_*$ in $\Set$, all graphs $\G \in \Gret$, and all $ \beta \in \GSp(\G, \Tpk S)$, the category $\factcat$ is connected.
\end{lem}

\begin{proof} 
If $\G \cong (\shortmid)$, then $\Gg$ is isomorphic to the identity $(\shortmid)$-shaped graph of graphs. 
 If $\G \cong \CX$ for some finite set $X$, and $\beta \in \Tpk S_X \cong  {\GSp}_{\TTp}(\G, S_*)$, then objects of $\factcat$ have the form $(\X, \alpha)$ where $\X\in X \Grisok$ is an admissible (in general not connected) $X$-graph, $\alpha \in S_*(\X)$ and $\beta$ is the image of $\alpha$ under the universal morphism $S_*(\X) \to (\Tpk S_* )_X$. Since $(\Tpk S_* )_X = \mathrm{colim}_{ \X' \in \XGretsimp}S(\X')$, if $(\X', \alpha')$ also represents $\beta$, then $\X$ and $\X'$ are connected in $\XGretsimp$ and hence they are similar in $\Gretp$. 
Hence, if $\G = \C$ is in the image of $\fisinvp \subset \Gretp$, then $\factcat$ is connected for all $\beta \in{\GSp}_{\TTp}(\C, S_*) \cong \Tpk S(\C)$.

For general $\G \in \Gretp$, and $\beta \in{\GSp}_{\TTp}(\G, S_*)$, observe first that, since colimits of graphs of graphs are computed componentwise, we may assume that $\G$ is connected. Assume also that $\G \not \cong \C_\nul$. Let $(\Gg^i, \alpha^i)$, $i = 1, 2$ be objects of $\factcat$. Then, for each $(\CX, b) \in \elG$, $\Gg^1(b)$ and $\Gg^2(b)$ are in the same connected component of $X \XGretsimp$, since the lemma holds when $\G \cong \CX$. Therefore, $\Gg^1(\G)$ and $\Gg^2(\G)$ are similar in $\Gretp$. 
\end{proof}

\cref{nerve theorem} now follows from \cite[Sections~1~\&~2]{BMW12}.

\begin{proof}[Proof of \cref{nerve theorem}]
The induced nerve functor $N\colon \CO \to \pr{\Klgrt}$ is fully faithful by \cref{connected}, \cite[Proposition~2.5]{BMW12}, and \cite[Propositions~ 1.5~\&~1.9]{BMW12}.

Moreover, by \cite[Theorem~1.10]{BMW12}, its essential image is the subcategory of those presheaves on $\Klgrt$ whose restriction to $\Gretp$ are in the image of the fully faithful embedding $\GSp \hookrightarrow \pr{\Gretp}$. So, a presheaf $P$ on $\Klgrt$ described the nerve of a circuit algebra if and only if 
$ P(\G)  \cong \mathrm{lim}_{(\C,b) \in \elpG} P(\C)$, and hence, by finality of $\elG \subset \elpG$, 
$ P(\G)  = \mathrm{lim}_{(\C,b) \in \elG} P(\C).$ Therefore, the Segal condition (\ref{eq. Segal nerve sec}) is satisfied and \cref{thm: CO nerve} is proved. 
\end{proof}

Following from Remarks \ref{rmk patterns} and \ref{rmk BMW prop}, \cref{thm: CO nerve} also admits proofs using extensions of strongly cartesian monads as described in \cite[Proposition~4.4]{BMW12}, and extendable algebraic patterns \cite{CH21}:
\begin{rmk}
	\label{rmk bmw proof} Recall \cref{rmk BMW prop}.
	By (\ref{eq. T def}), for all finite sets $X$, 
	\[	\Tk S_X =  \mathrm{colim}_{\X\in X{\Griso}}  S(\X) = \mathrm{colim}_{\X\in X{\Griso}}  \mathrm{lim}_{(\C,b) \in \elG[\X]} S(\C). \] So (e.g.,~by \cite[Section~1]{CH21}) $\TTk$ is strongly cartesian with canonical arities $\Gret \subset \GS$.
	
	By \cref{ssec. D monad}, the right adjoint functor $i^* \colon \GSp \to \GS$ associated to the monad $\DD$ on $\GS$ is just the pullback to presheaves of the bijective on objects inclusion $i \colon\fisinv \hookrightarrow \fisinvp$. And, by construction, $(LT) i^* = i^*(\Tpk)\colon \GSp \to \GS$ so the monad $\TTpk$ on $\GSp$ is the pushforward along $i$ of $\LL \TT = \TTk$ on $\GS$. Hence, by \cite[Proposition~4.4]{BMW12}, $\TTpk$ on $\GSp$ is strongly cartesian with canonical arities $\Gretp$.

\end{rmk}

\begin{rmk}
	\label{rmk patterns proof} Recall Remarks \ref{rmk strongly cartesian} and \ref{rmk patterns}.  
	By (\ref{eq. lift LT}), for all finite sets $X$,
	\[	\begin{array}{lll}{\Tpk} S_X & = &\mathrm{colim}_{\X \in X\XGretsimp} S(\X)\\
		& = &\mathrm{colim}_{\X \in X\XGretsimp} \mathrm{lim}_{(\C,b) \in \elG[\X]}S(\C)\end{array}\]
	And, by the previous discussion, a morphism in $\Klgrt$ factor as a refinement (induced by a graphs of graphs) followed by a $\Gretp$ morphism. 

	Hence (e.g.,~by \cite[Section~1]{CH21}) $\Klgrt$ describes an extendable algebraic pattern, and 
	$\TTpk$ is the associated strongly cartesian monad, with canonical arities $\Gretp$. 
	\cref{thm: CO nerve} then follows immediately from \cite{BMW12, Web07}. 

\end{rmk}

It remains, finally to remark on the homotopy theory of circuit algebras (see also the Introduction).

\begin{rmk}
	\label{rmk. No weak version}
Both the nerve theorem \cite[Theorem~8.2]{Ray20} for modular operads on which \cref{thm: CO nerve} is built, as well as Hackney, Robertson and Yau's nerve theorem \cite[3.6]{HRY19b}, lead to well-defined models of $(\infty, 1)$-modular operads in which the fibrant objects in the appropriate category of presheaves are precisely those that satisfy the weak Segal condition: in the case of \cite{Ray20} it followed from the results of \cite{CH15} that these were given by $P \colon \Klgr^{\mathrm{op}} \to \sSet$ such that, for all (connected) graphs $\G \in \Klgr$,
\[P(\G) \simeq \mathrm{lim}_{(\C,b) \in \elG} P(\C). \]

However, describing a similar model for circuit algebras is more challenging.  The presence of unconnected graphs means that, neither the method of \cite{HRY19b}, nor that of \cite[Corollary~8.14]{Ray20} may be directly extended to a model structure on $\prE[\sSet]{\Klgrt}$ (or some subcategory thereof) that gives a good description of weak circuit algebras. It is for later work to investigate the existence of a such a model.

\end{rmk}

\bibliography{Compactbib}{}

\begin{thebibliography}{10}

\bibitem{SGA4}
Michael Artin, Alexander Grothendieck, J.~L. Verdier, N.~Bourbaki, P.~Deligne,
  and Bernard Saint-Donat, editors.
\newblock {\em S{\'e}minaire de g{\'e}om{\'e}trie alg{\'e}brique du
  {Bois}-{Marie} 1963--1964. {Th{\'e}orie} des topos et cohomologie {\'e}tale
  des sch{\'e}mas. ({SGA} 4). {Un} s{\'e}minaire dirig{\'e} par {M}. {Artin},
  {A}. {Grothendieck}, {J}. {L}. {Verdier}. {Avec} la collaboration de {N}.
  {Bourbaki}, {P}. {Deligne}, {B}. {Saint}-{Donat}. {Tome} 1: {Th{\'e}orie} des
  topos. {Expos{\'e}s} {I} {\`a} {IV}. 2e {\'e}d.}, volume 269 of {\em Lect.
  Notes Math.}
\newblock Springer, Cham, 1972.

\bibitem{Ban16}
Markus Banagl.
\newblock High-{Dimensional} {Topological} {Field} {Theory}, {Positivity}, and
  {Exotic} {Smooth} {Spheres}.
\newblock Preprint, {arXiv}:1508.01337 [math.{AT}], 2015.

\bibitem{BND17}
Dror Bar-Natan and Zsuzsanna Dancso.
\newblock Finite type invariants of w-knotted objects {II}: tangles, foams and
  the {K}ashiwara-{V}ergne problem.
\newblock {\em Math. Ann.}, 367(3-4):1517--1586, 2017.

\bibitem{BB17}
M.~A. Batanin and C.~Berger.
\newblock Homotopy theory for algebras over polynomial monads.
\newblock {\em Theory Appl. Categ.}, 32:Paper No. 6, 148--253, 2017.

\bibitem{Bec69}
Jon Beck.
\newblock Distributive laws.
\newblock In {\em Sem. on {T}riples and {C}ategorical {H}omology {T}heory
  ({ETH}, {Z}\"{u}rich, 1966/67)}, pages 119--140. Springer, Berlin, 1969.

\bibitem{BMW12}
Clemens Berger, Paul-Andr\'{e} Melli\`es, and Mark Weber.
\newblock Monads with arities and their associated theories.
\newblock {\em J. Pure Appl. Algebra}, 216(8-9):2029--2048, 2012.

\bibitem{BG19}
John Bourke and Richard Garner.
\newblock Monads and theories.
\newblock {\em Adv. Math.}, 351:1024--1071, 2019.

\bibitem{Bra37}
Richard Brauer.
\newblock On algebras which are connected with the semisimple continuous
  groups.
\newblock {\em Ann. of Math. (2)}, 38(4):857--872, 1937.

\bibitem{CH15}
Giovanni Caviglia and Geoffroy Horel.
\newblock Rigidification of higher categorical structures.
\newblock {\em Algebr. Geom. Topol.}, 16(6):3533--3562, 2016.

\bibitem{Che11}
Eugenia Cheng.
\newblock Iterated distributive laws.
\newblock {\em Math. Proc. Cambridge Philos. Soc.}, 150(3):459--487, 2011.

\bibitem{CH21}
Hongyi Chu and Rune Haugseng.
\newblock Homotopy-coherent algebra via {Segal} conditions.
\newblock {\em Adv. Math.}, 385:95, 2021.
\newblock Id/No 107733.

\bibitem{DF18}
Celeste Damiani and Vincent Florens.
\newblock Alexander invariants of ribbon tangles and planar algebras.
\newblock {\em J. Math. Soc. Japan}, 70(3):1063--1084, 2018.

\bibitem{DHR20}
Zsuzsanna Dancso, Iva Halacheva, and Marcy Robertson.
\newblock Circuit algebras are wheeled props.
\newblock {\em J. Pure Appl. Algebra}, 225(12):106767, 33, 2021.

\bibitem{DHR21}
Zsuzsanna Dancso, Iva Halacheva, and Marcy Robertson.
\newblock A topological characterisation of the {Kashiwara}-{Vergne} groups.
\newblock {\em Trans. Am. Math. Soc.}, 376(5):3265--3317, 2023.

\bibitem{DM23}
Harm Derksen and Visu Makam.
\newblock Invariant theory and wheeled {PROPs}.
\newblock {\em J. Pure Appl. Algebra}, 227(9):30, 2023.
\newblock Id/No 107302.

\bibitem{Mat25}
Adri{\'a}n Do{\~n}a~Mateo.
\newblock Pushforward monads.
\newblock {\em Theory Appl. Categ.}, 44:927--963, 2025.

\bibitem{DCH19}
Gabriel~C. Drummond-Cole and Philip Hackney.
\newblock Dwyer-{Kan} homotopy theory for cyclic operads.
\newblock {\em Proc. Edinb. Math. Soc., II. Ser.}, 64(1):29--58, 2021.

\bibitem{Gig17}
Jacques Gignoux, Guillaume Ch{\'e}rel, Ian~D. Davies, Shayne~R. Flint, and Eric
  Lateltin.
\newblock Emergence and complex systems: The contribution of dynamic graph
  theory.
\newblock {\em Ecological Complexity}, 31:34--49, 2017.

\bibitem{Gre23}
David~G. Green.
\newblock {Emergence in complex networks of simple agents}.
\newblock {\em Journal of Economic Interaction and Coordination},
  18(3):419--462, July 2023.

\bibitem{Hac24}
Philip Hackney.
\newblock Categories of graphs for operadic structures.
\newblock {\em Math. Proc. Camb. Philos. Soc.}, 176(1):155--212, 2024.

\bibitem{Hac22}
Philip Hackney.
\newblock Segal conditions for generalized operads.
\newblock In {\em Higher structures in topology, geometry, and physics. AMS
  special session, virtual, March 26--27, 2022}, pages 161--194. Providence,
  RI: American Mathematical Society (AMS), 2024.

\bibitem{HRY15}
Philip Hackney, Marcy Robertson, and Donald Yau.
\newblock {\em Infinity properads and infinity wheeled properads}, volume 2147
  of {\em Lecture Notes in Mathematics}.
\newblock Springer, Cham, 2015.

\bibitem{HRY19a}
Philip Hackney, Marcy Robertson, and Donald Yau.
\newblock A graphical category for higher modular operads.
\newblock {\em Adv. Math.}, 365:107044, 2020.

\bibitem{HRY19b}
Philip Hackney, Marcy Robertson, and Donald Yau.
\newblock Modular operads and the nerve theorem.
\newblock {\em Adv. Math.}, 370:107206, 39, 2020.

\bibitem{Hal16}
Iva Halacheva.
\newblock Alexander type invariants of tangles, 2016.
\newblock arXiv: 1611.09280.

\bibitem{Jon94}
V.~F.~R. Jones.
\newblock The {P}otts model and the symmetric group.
\newblock In {\em Subfactors ({K}yuzeso, 1993)}, pages 259--267. World Sci.
  Publ., River Edge, NJ, 1994.

\bibitem{JK11}
A.~Joyal and J.~Kock.
\newblock Feynman graphs, and nerve theorem for compact symmetric
  multicategories (extended abstract).
\newblock {\em Electronic Note in Theoretical Computer Science}, 270(2):105 --
  113, 2011.

\bibitem{Joy81}
Andr\'{e} Joyal.
\newblock Une th\'{e}orie combinatoire des s\'{e}ries formelles.
\newblock {\em Adv. in Math.}, 42(1):1--82, 1981.

\bibitem{KW24}
Ralph~M. Kaufmann and Benjamin~C. Ward.
\newblock Schwarz {Modular} {Operads} {Revisited}.
\newblock Preprint, {arXiv}:2404.17540 [math.{AT}], 2024.

\bibitem{Koc16}
Joachim Kock.
\newblock Graphs, hypergraphs, and properads.
\newblock {\em Collect. Math.}, 67(2):155--190, 2016.

\bibitem{Koc18}
Joachim Kock.
\newblock Cospan construction of the graph category of {B}orisov and {M}anin.
\newblock {\em Publ. Mat.}, 62(2):331--353, 2018.

\bibitem{KRW21}
A.~Kupers and O.~Randal-Williams.
\newblock On the {Torelli} {Lie} algebra.
\newblock {\em Forum Math. Pi}, 11:47, 2023.
\newblock Id/No e13.

\bibitem{Lam66}
Joachim Lambek.
\newblock {\em Completions of categories}.
\newblock Seminar lectures given in 1966 in Z\"{u}rich. Lecture Notes in
  Mathematics, No. 24. Springer-Verlag, Berlin-New York, 1966.

\bibitem{LZ15}
G.~I. Lehrer and R.~B. Zhang.
\newblock The {B}rauer category and invariant theory.
\newblock {\em J. Eur. Math. Soc. (JEMS)}, 17(9):2311--2351, 2015.

\bibitem{Mac98}
Saunders Mac~Lane.
\newblock {\em Categories for the working mathematician}, volume~5 of {\em
  Graduate Texts in Mathematics}.
\newblock Springer-Verlag, New York, second edition, 1998.

\bibitem{MMS09}
M.~Markl, S.~Merkulov, and S.~Shadrin.
\newblock Wheeled {PROP}s, graph complexes and the master equation.
\newblock {\em J. Pure Appl. Algebra}, 213(4):496--535, 2009.

\bibitem{Mar96}
Martin Markl.
\newblock Distributive laws and {Koszulness}.
\newblock {\em Ann. Inst. Fourier}, 46(2):307--323, 1996.

\bibitem{Mer10}
Sergei~A. Merkulov.
\newblock Wheeled props in algebra, geometry and quantization.
\newblock In {\em European {C}ongress of {M}athematics}, pages 83--114. Eur.
  Math. Soc., Z\"{u}rich, 2010.

\bibitem{MW07}
Ieke Moerdijk and Ittay Weiss.
\newblock Dendroidal sets.
\newblock {\em Algebr. Geom. Topol.}, 7:1441--1470, 2007.

\bibitem{PBHF23}
Evan Patterson, Andrew Baas, Timothy Hosgood, and James Fairbanks.
\newblock A diagrammatic view of differential equations in physics.
\newblock {\em Mathematics in Engineering}, 5(2):1--59, 2023.

\bibitem{Ray20}
Sophie Raynor.
\newblock Graphical combinatorics and a distributive law for modular operads.
\newblock {\em Adv. Math.}, 392:Paper No. 108011, 87, 2021.

\bibitem{RayCA1}
Sophie Raynor.
\newblock Functorial, operadic and modular operadic combinatorics of circuit
  algebras.
\newblock {\em J. Pure Appl. Algebra}, 229(11):37, 2025.
\newblock Id/No 108105.

\bibitem{RS20}
Hebing Rui and Linliang Song.
\newblock Representations of {B}rauer category and categorification.
\newblock {\em J. Algebra}, 557:1--36, 2020.

\bibitem{SS15}
Steven~V. Sam and Andrew Snowden.
\newblock Stability patterns in representation theory.
\newblock {\em Forum Math. Sigma}, 3:Paper No. e11, 108, 2015.

\bibitem{Sch96}
Albert Schwarz.
\newblock Grassmannian and string theory.
\newblock {\em Communications in Mathematical Physics}, 199:1--24, 1996.

\bibitem{Sto23}
Kurt Stoeckl.
\newblock Koszul operads governing props and wheeled props.
\newblock {\em Adv. Math.}, 454:80, 2024.
\newblock Id/No 109869.

\bibitem{Sto22}
Robin Stoll.
\newblock {Modular operads as modules over the Brauer properad}.
\newblock {\em Theor. Appl. Categor.}, 38:1538--1607, 2022.

\bibitem{Str72}
Ross Street.
\newblock The formal theory of monads.
\newblock {\em J. Pure Appl. Algebra}, 2:149--168, 1972.

\bibitem{Str25}
Michelle Strumila.
\newblock Quasi {Modular} {Operads}.
\newblock Preprint, {arXiv}:2504.06522 [math.{CT}] (2025), 2025.

\bibitem{Tub14}
Daniel Tubbenhauer.
\newblock Virtual {K}hovanov homology using cobordisms.
\newblock {\em J. Knot Theory Ramifications}, 23(9):1450046, 91, 2014.

\bibitem{War22}
Benjamin~C. Ward.
\newblock Massey products for graph homology.
\newblock {\em Int. Math. Res. Not.}, 2022(11):8086--8161, 2022.

\bibitem{Web07}
Mark Weber.
\newblock Familial 2-functors and parametric right adjoints.
\newblock {\em Theory Appl. Categ.}, 18:No. 22, 665--732, 2007.

\bibitem{JY15}
Donald Yau and Mark~W. Johnson.
\newblock {\em A foundation for {PROP}s, algebras, and modules}, volume 203 of
  {\em Mathematical Surveys and Monographs}.
\newblock American Mathematical Society, Providence, RI, 2015.

\end{thebibliography}
\bibliographystyle{plain}

\end{document}